\documentclass[12pt]{article}
\usepackage{amsmath}
\usepackage{amsmath,amssymb, amsthm, latexsym, amsfonts, epsfig, color,graphicx,}
\usepackage{mathrsfs}
\usepackage{indentfirst}
\usepackage{chngcntr}
\usepackage{hyperref}

\hypersetup{colorlinks=true,linkcolor=black}

\begin{document}
	\renewcommand{\a}{\alpha}
	\newcommand{\D}{\Delta}
	\newcommand{\ddt}{\frac{d}{dt}}
	\counterwithin{equation}{section}
	\newcommand{\e}{\epsilon}
	\newcommand{\eps}{\varepsilon}
	\newtheorem{theorem}{Theorem}[section]
	
	\newtheorem{proposition}{Proposition}[section]
	\newtheorem{lemma}[proposition]{Lemma}
	\newtheorem{remark}{Remark}[section]
	\newtheorem{example}{Example}[section]
	\newtheorem{definition}{Definition}[section]
	\newtheorem{corollary}{Corollary}[section]
	\makeatletter
	\newcommand{\rmnum}[1]{\romannumeral #1}
	\newcommand{\Rmnum}[1]{\expandafter\@slowromancap\romannumeral #1@}
	\makeatother

\title{Dispersive decay estimates for the magnetic Schr\"odinger equations}

\author{Zhiwen Duan\, \, \, Lei Wei$^{\star}$\\
{\small {\it  School of Mathematics and Statistics, Huazhong University of Science }}\\
{\small {\it and Technology, Wuhan, {\rm 430074,} P.R.China.}}  \\
{\small {\it $^{\star}$Corresponding author. Email: weileimath@hust.edu.cn }}\\
{\small {\it Contributing author. Email: duanzhw@hust.edu.cn}}}
\date{}
\maketitle

\centerline{\large\bf Abstract}	
\par	
In this paper, we present a proof of dispersive decay for both linear and nonlinear magnetic Schr\"odinger equations. To achieve this, we introduce the fractional distorted Fourier transforms with magnetic potentials and define the fractional differential operator $\arrowvert J_{A}(t)\arrowvert^{s}$. By leveraging the properties of the distorted Fourier transforms and the Strichartz estimates of $\arrowvert J_{A}\arrowvert^{s}u$, we establish the dispersive bounds with the decay rate $t^{-\frac{n}{2}}$. This decay rate provides valuable insights into the spreading properties and long-term dynamics of the solutions to the magnetic Schr\"odinger equations.\\


\noindent{\bf Authors’ Contributions:} All authors contributed equally to this work. All authors read and approved the final manuscript.

\noindent{\bf Funding Statement:} This work was supported by National Natural Science Foundation of China (NSFC) (61671009,12171178).

\noindent{\bf Conflicts of Interest:} The authors declare that they have no conflicts of interest to report regarding the present study.

\newpage
\tableofcontents
\newpage

\section{Introduction}
We consider the magnetic Schr\"odinger operator  $$-\Delta_{A}:=-(\nabla+iA)^{2}=-\Delta-iA\cdot\nabla-i\nabla\cdot A+\arrowvert A\arrowvert^{2}$$
and the nonlinear magnetic Schr\"odinger equation
\begin{equation}\label{1.1}
	\left\{
	\begin{split}
		&(i\partial_{t}+\Delta_{A})u=\arrowvert u\arrowvert^{p-1}u, \ x\in \mathbb{R}^{n},\ t>1,\\
		&u(1)=u_{0},
	\end{split}
	\right.
\end{equation}
where $p>1+\lceil\frac{n}{2}\rceil$, and  $A(x)=(A_{1}(x),...,A_{n}(x))$ is a real-valued vector function which satisfies the following hypothesis:
\begin{equation}\label{1.2}
	\arrowvert A(x)\arrowvert+\arrowvert\nabla A(x)\arrowvert+\arrowvert\nabla\nabla A(x)\arrowvert\leqslant C\langle x \rangle^{-\beta},
\end{equation}
for some $\beta>n+2$.

Therefore, we claim that $\sigma(\Delta_{A})=(-\infty,0)$, where $\sigma(\Delta_{A})$ is the spectrum of operator $\Delta_{A}$. In fact, for any $0\neq u\in \mathscr{S}$, we have
\begin{equation}\label{1.3}
((\nabla+iA)^{2}u,u)_{L^{2}}
=-\int_{\mathbb{R}^{n}}\sum\limits_{j=1}^{n}\arrowvert(\partial_{x_{j}}+iA_{j})u\arrowvert^{2}dx<0.
\end{equation}
Furthermore, from Lemma A.2 in \cite{KK} (also see \cite{KT}), we obtain that zero is neither an eigenvalue nor a resonance of the magnetic Schr\"odinger operator $\Delta_{A}$ and $\sigma_{sc}(\Delta_{A})=\emptyset$, where $0$ is called a resonance if there exists a distribution solution of $-\Delta_{A}u=0$ such that $u\in L^{2,-\sigma}$ but  $u\notin L^{2}$ for any $\sigma>\frac{1}{2}$ in \cite{JK}, and   $\sigma_{sc}(\Delta_{A})$ is the singular continuous spectrum of $\Delta_{A}$.

\begin{remark}\label{remark1.1}
	For $n=1$, we recall the following fact: let $G(x)$ be a real smooth function and $\mathcal{T}$ be the gauge transformation:
	$$\mathcal{T}u(t,x)=e^{-iG(x)}u(t,x).$$
	Then $\mathcal{T}$ transforms \eqref{1.1} into the same one with $A$ being replaced by $A+\nabla_{x}G$. In particular, taking $G(x)=-\int_{0}^{x}A(y)dy$, the potential $A$ is eliminated from \eqref{1.1} by the gauge transformation $\mathcal{T}$. Therefore, we only consider the case where the space dimension $n\geqslant2$.
\end{remark}

From the kernel expression of $e^{-it\Delta}$, the dispersive bound of the free Schr\"odinger equations follows
$$
\|e^{-it\Delta}\psi\|_{L^{\infty}(\mathbb{R}^{n})}\leqslant C\arrowvert t\arrowvert^{-\frac{n}{2}}\|\psi\|_{L^{1}(\mathbb{R}^{n})}.
$$
It is naturally to ask whether such a bound is applicable to more general Schr\"odinger equations. Let's recall it.

For the linear Schr\"odinger equations with a real-valued potential $V(x)$: $i\partial_{t}u-\Delta_{V}u=0$, where $\Delta_{V}:=\Delta-V(x)$, the original dispersive estimates expressed $e^{-it\Delta_{V}}$ as a mapping between the weighted $L^{2}$-spaces,  and the weights are exponential \cite{R} or polynomial \cite{JK}. A significant advance was made by Journ\'{e}, Soffer and Sogge \cite{JSS}, who proved the bound of $\|e^{-it\Delta_{V}}\|_{L^{1}\rightarrow L^{\infty}}$:
\begin{equation}\label{1.4}
	\|e^{-it\Delta_{V}}\mathscr{P}_{ac}\psi\|_{L^{\infty}}\leqslant C\arrowvert t\arrowvert^{-\frac{n}{2}}\|\psi\|_{L^{1}},\ \forall\  \psi\in\mathscr{S}(\mathbb{R}^{n}),\ n\geqslant3,
\end{equation}
where $\mathscr{P}_{ac}$ is the projection onto the absolute continuous spectrum of $-\Delta_{V}$, and the potential $V(x)$ satisfies some decay conditions. In addition, they needed to assume that
\begin{equation}\label{1.5}
	\ \ \ \ zero\ is\ neither\ an\ eigenvalue\ nor\ a\  resonance\ of -\Delta_{V}.
\end{equation}
 For the case $n=3$, Yajima \cite{Ya2} proved that the wave operators are bounded and the condition applied to potential $V$ is weaker than \cite{JSS}. Goldberg and Schlag \cite{GS} used different method from both \cite{JSS} and \cite{Ya2} to get the dispersive bound in \eqref{1.4} under different assumptions on $V$ for the case $n=1$ and $n=3$ respectively. Rodnianski and Schlag \cite{RS1} proved the dispersive estimates of Schr\"odinger equations with time-dependent and time-independent potential $V(x)$. Goldberg \cite{Go2} proved the dispersive estimates with the potential $V(x)\in L^{p}\cap L^{q}$, $p<\frac{3}{2}<q$, and satisfies the zero-energy condition \eqref{1.5}. For other dimensions, Schlag \cite{Sc1} proved $L^{1}(\mathbb{R}^{2})\rightarrow L^{\infty}(\mathbb{R}^{2})$ bound and the decay rate satisfied by $V(x)$ is greater than $3$. Cardoso, Cuevas and Vodev \cite{CCV} showed that the dispersive estimates hold for $n=4,5$. For even dimensional case \cite{FY} and odd dimensional case \cite{Ya3}, the $L^{p}$ boundedness of wave operators for Schr\"odinger operators with threshold singularities was obtained respectively, and so on.

For the magnetic Schr\"{o}dinger equations has gained significant attention and interest among scientists in recent years because taht the dispersive estimates for magnetic Schr\"{o}dinger equations are importance in the study of quantum mechanics and wave propagation in the presence of magnetic fields. In \cite[Theorem 4]{Ya1}, Yajima proved the short-time dispersive estimate for the magnetic Schr\"{o}dinger equation, where the magnetic potential $A(t,x)$ satisfies some conditions and depends on the time $t$. Especially, for the time-independent magnetic potential, Yajima \cite{Ya1} proved the following estimate:
\begin{equation}\label{1.6}
	\|e^{it(-i\nabla-A(x))^{2}}\psi\|_{L^{\infty}_{x}(\mathbb{R}^{n})}\leqslant C\arrowvert t\arrowvert^{-\frac{n}{2}}\|\psi\|_{L^{1}_{x}(\mathbb{R}^{n})},\ \arrowvert t\arrowvert\leqslant T\ll1,\ n\geqslant3.
\end{equation}
After, Komech and Kopylova \cite{KK},   obtained the time-decaying rate of the solutions in weighted $L^{2}$-space for the following Schr\"{o}dinger equation
\begin{equation}\label{1.7}
	\|e^{it((-i\nabla-A(x))^{2}+V(x))}\mathscr{P}_{c}\psi\|_{L^{2,-\sigma}}\leqslant C\langle t \rangle^{-\frac{3}{2}}\|\psi\|_{L^{2,\sigma}},\ t\in \mathbb{R},
\end{equation}
where $\sigma>\frac{5}{2}$, $\mathscr{P}_{c}$ is the projection onto the continuous spectrum of the operator $(-i\nabla-A(x))^{2}+V(x)$.

For the nonlinear Schr\"odinger equation $i\partial_{t}u-\Delta u=\arrowvert u\arrowvert^{p-1}u,$ it is natural to ask whether the solutions of this equation have the decay rate of the solution to the linear Schr\"odinger equation. In \cite{St3,St2}, Strauss proved that the zero solution is the only asymptotically free solution when $1<p\leqslant1+\frac{2}{n}$ for $n\geqslant2$ and when $1<p\leqslant2$ for $n=1$ (see also \cite{St1}). Using the idea of Glassey \cite{Gl}, the time-decaying estimates were extended to the case $1<p\leqslant 1+\frac{2}{n}$ for $n\geqslant1$ by Barab \cite{Ba}. When  $1+\frac{2}{n}<p<1+\frac{4}{n}$ for $n\geqslant1$, the time-decaying estimates were proved by Mckean and Shatah \cite{MS}.

For the nonlinear Schr\"odinger equations with potential $V(x)$, Cuccagna, Georgiev and Visciglia \cite{CGV} for the one-dimensional case proved that
$$\|u(t)\|_{L^{\infty}(\mathbb{R})}\leqslant Ct^{-\frac{1}{2}}\|u_{0}\|_{\Sigma_{s}}, \,\, s>\frac{1}{2},\ t>1,$$
where the initial datum $u_0=u(1)$ is small in the Hilbert space $H^{s}(\mathbb{R})\cap\arrowvert x\arrowvert^{s}L^{2}(\mathbb{R})$ denoted by $\Sigma_{s}$, i.e. $\|u_{0}\|_{\Sigma_{s}}\ll1$. Their arguments in \cite{CGV} were based on the distorted Fourier transforms (see \cite{H} for details) and the equivalence property  $\|(-\Delta)^{\frac{s}{2}}u\|_{L^{2}(\mathbb{R})}\sim\|(-\Delta+V)^{\frac{s}{2}}u\|_{L^{2}(\mathbb{R})}$ for $0\leqslant s<\frac{1}{2}$. The proof of the equivalence was based on the explicit expression of the resolvent kernel in one dimensional space. As far as I know, the nonlinear Schrödinger equation with magnetic potential remains a challenging problem that is actively being investigated.

In this paper, we focus on studying the magnetic Schrödinger operator $-\Delta_{A}$ and the nonlinear Schr\"{o}dinger equation \eqref{1.1}. Our main result is the derivation of the dispersive decay with an optimal decay rate of $t^{-\frac{n}{2}}$. To achieve this, we draw inspiration from the method used in \cite{CGV} and introduce the fractional magnetic distorted Fourier transform for the fractional Schr\"{o}dinger operators $(-\Delta_{A})^{\frac{s}{2}}$ ($s>0$). This approach allows us to establish the desired dispersive estimates and gain insights into the decay behavior of the solutions to the nonlinear Schr\"{o}dinger equation in the presence of a magnetic potential.

Before showing our work, we recall the well-posedness of solutions for the magnetic Schr\"{o}dinger equations. The well-posedness of weak solutions of the nonlinear Schr\"odinger equation was obtained in the case of $A=0$ in \cite{K1,K2,T1}, and in the case of $A\neq0$ in  \cite{Bo,Na}. For the local well-posedness of strong solutions of the same equation, it was acquired in \cite{K2,T2} when $A=0$, and in \cite{NS} when $A\neq0$. In Appendix A of this paper, we obtain that the initial value problem of the nonlinear magnetic Schr\"odinger equation is globally well-posed in  $L^{\infty}((1,\infty)\times\mathbb{R}^{n})$.\\

$Outline\ of\ this\ paper:$

In the key Section 2, we consider the linear magnetic Schr\"odinger equation and obtain the sharp time decay estimates of the $\|u\|_{L_x^{\infty}}$. On the one hand, we construct the fractional magnetic distorted Fourier transforms at the resolvent points, $$F^{A}u(\xi)=\mathscr{F}((I+V_{x}R^{s}_{0}(\tau))^{-1}u)(\xi),\ \tau>0.$$  In order to make this construct meaningful, we prove the boundedness of operators $V_{x}$ and $R^{s}_{0}(\tau)$. Next, we have the pseudo-intertwining property of the distorted Fourier transforms, $$\arrowvert\xi\arrowvert^s\mathscr{F}u=F^{A}\mathscr{H}u+\tau(F^{A}u-\mathscr{F}u),$$ where $\mathscr{H}$ is the self-adjoint realization of $(-\Delta_{A})^{\frac{s}{2}}$ for any $s>0$ in $L^{2}(\mathbb{R}^{n})$. It is worth noting that we choose the resolvent point $\tau$ near zero to deal with the remainder, then we obtain the equivalence property for $0<s<\frac{n}{2}$,
$$\|(-\Delta)^{\frac{s}{2}}u\|_{L^{2}_{x}(\mathbb{R}^{n})}\sim\|(-\Delta_{A})^{\frac{s}{2}}u\|_{L^{2}_{x}(\mathbb{R}^{n})}.$$

On the other hand, applying the intertwining property of the distorted Fourier transforms at the spectral points (for details see \cite{DW})
$$F_\pm^{A}\mathscr{H}u=\arrowvert\xi\arrowvert^sF_\pm^{A}u$$
to estimate the $L^{\infty}$-norm of solution for linear magnetic Schr\"odinger equation, we have
\begin{equation}\label{1.8}
\|u\|_{L^{\infty}_{x}(\mathbb{R}^{n})}\leqslant Ct^{-\frac{n}{2}}\|u\|^{1-\frac{n}{2s}}_{L^{2}_{x}(\mathbb{R}^{n})}\cdot\|\arrowvert J_{A}\arrowvert^{s}u\|^{\frac{n}{2s}}_{L^{2}_{x}(\mathbb{R}^{n})},\ s>\frac{n}{2},
\end{equation}
(see Lemma \ref{lemma2.15} and $Step\ 4$ in proof of Theorem \ref{theorem1.1}), where
\begin{flalign*}
\arrowvert J_{A}\arrowvert^{s}:=M(t)(-t^{2}\Delta_{A})^{\frac{s}{2}}M(-t),\ M(t):=e^{\frac{i\arrowvert x\arrowvert^{2}}{4t}}.
\end{flalign*}
Therefore, it is enough to show the boundedness of $\|\arrowvert J_{A}\arrowvert^{s}u\|_{L^{2}_{x}(\mathbb{R}^{n})}$. Note that
\begin{flalign}\label{1.9}
(i\partial_{t}+\Delta_{A})\arrowvert J_{A}\arrowvert^{s}u=[i\partial_{t}+\Delta_{A},\arrowvert J_{A}\arrowvert^{s}]u=it^{s-1}M(t)V(s)M(-t)u,
\end{flalign}
where $V(s)$ presentes in \eqref{2.15}. Indeed, the estimate of $V(s)$ is a crucial step in the proof process, and the most challenging aspect lies in obtaining the estimate of the resolvent $R_{A}$. 

Our main challenges lie in proving the estimates and the first-order derivative estimates of $R_{A}(\tau)$ while also obtaining the power of $\tau$ in the estimates. Previous works by Cuenin and Kenig \cite{CK} have provided the half-order derivative estimates and $L^{q'}(\mathbb{R}^{n})\rightarrow L^{q}(\mathbb{R}^{n})$ estimates ($q\in[2,\frac{2n}{n-2}]$) of the magnetic resolvent. To address this, we start by establishing some estimates of the free resolvent $R_{0}(\tau)=(\tau-\Delta)^{-1}$ with weight functions using the resolvent identity. Then, we use these resolvent estimates and the integral expression of $V(s)$ to handle the estimate of $V(s)$ when $\tau\in(0,1)$. In the process, in order to reduce the power of $\tau$, we exchange the regularity of space with the decay of the potential $A(x)$. Through this approach, we are able to obtain the desired estimate of $V(s)$.

Combining the equivalence property of $(-\Delta)^{\frac{s}{2}}$ and $(-\Delta_{A})^{\frac{s}{2}}$, and the Strichartz estimate of $\arrowvert J_{A}\arrowvert^{s}u$, we have the following decay result.
\begin{theorem}\label{theorem1.1}
	Let $A(x)$ satisfy hypothesis \eqref{1.2}, then for any $s>\frac{n}{2}$, there exists a constant $C>0$ such that
	\begin{equation}\label{1.10}
		\|e^{it\Delta_{A}}u_{0}\|_{L^{\infty}_{x}(\mathbb{R}^{n})}\leqslant Ct^{-\frac{n}{2}}\|u_{0}\|_{\Sigma_{s}},\  t\geqslant1,
	\end{equation}
	where $\|u_{0}\|_{\Sigma_{s}}:=\|(-\Delta)^{\frac{s}{2}}(e^{-\frac{i\arrowvert x\arrowvert^{2}}{4}}u_{0})\|_{L^{2}_{x}}$.
\end{theorem}
\begin{remark}\label{remark1.2}
	The space of the initial value $\|u_{0}\|_{\Sigma_{s}}$ is  intersection of some common Sobolev spaces. See Appendix B for details.
\end{remark}

In Section 3, we investigate the dispersive decay for the nonlinear magnetic Schr\"odinger equation \eqref{1.1}. The method used to prove the dispersive bound is similar to the approach employed in dealing with the linear Schr\"odinger equation in Section 2. However, the main difference lies in the estimation of the nonlinear term $\arrowvert J_{A}\arrowvert^{s}(\arrowvert u\arrowvert^{p-1}u)$, which is obtained through an iterative method applied to the relationship between $\arrowvert J\arrowvert^{s}$ and $\arrowvert J_{A}\arrowvert^{s}$. As a result, we obtain the following second result:
\begin{theorem}\label{theorem1.2}
	Let $A(x)$ satisfy hypothesis \eqref{1.2} and $p>1+\lceil\frac{n}{2}\rceil$. Then for any $s>\frac{n}{2}$, and $\|u_{0}\|_{\Sigma_{s}}$ is sufficiently small, there exists a constant $C>0$ such that the solution of equation \eqref{1.1} satisfies the following dispersive decay
	\begin{equation}\label{1.11}
		\|u\|_{L^{\infty}_{x}(\mathbb{R}^{n})}\leqslant Ct^{-\frac{n}{2}}\|u_{0}\|_{\Sigma_{s}},\  t\geqslant1.
	\end{equation}
\end{theorem}

\phantomsection
\addcontentsline{toc}{section}{Notation}
\hspace{-5mm}\textbf{Notation}:
\begin{itemize}
	\setlength{\baselineskip}{1.5\baselineskip}
	
	\item $\mathscr{S}(\mathbb{R}^{n})$ is Schwartz space, i.e. the set of all real- or complex-valued $C^{\infty}$ functions on $\mathbb{R}^{n}$ such that for every nonnegative integer $\kappa$ and every multi-index $\alpha$, $\sup\limits_{x\in\mathbb{R}^{n}}(1+\arrowvert x\arrowvert^{2})^{\frac{\kappa}{2}}\arrowvert D^{\alpha}u(x)\arrowvert<\infty$; $\mathscr{S}'(\mathbb{R}^{n})$ is space of tempered distribution on $\mathbb{R}^{n}$, i.e. the topological dual of $\mathscr{S}(\mathbb{R}^{n})$.

	\item $\mathscr{F}u(x)=(2\pi)^{-\frac{n}{2}}\int_{\mathbb{R}^{n}}e^{-ix\cdot \xi}u(\xi)d\xi$; $(\mathscr{F}^{-1}u)(x)=(2\pi)^{-\frac{n}{2}}\int_{\mathbb{R}^{n}}e^{ix\cdot \xi}u(\xi)d\xi$.
	
	\item
	$H^{s}(\mathbb{R}^{n})=\{u\in L^{2}(\mathbb{R}^{n}):(1+\arrowvert\xi\arrowvert^{2})^{\frac{s}{2}}\hat{u}\in L^{2}(\mathbb{R}^{n})\}$; $\dot{H}^{s}(\mathbb{R}^{n})=\{u\in L^{2}(\mathbb{R}^{n}):\arrowvert\xi\arrowvert^{s}\hat{u}\in L^{2}(\mathbb{R}^{n})\}$; $\dot{H}_{A}^{s}(\mathbb{R}^{n})=\{u\in L^{2}(\mathbb{R}^{n}):(\mathscr{F}(-\Delta_{A})^{\frac{s}{2}})\cdot\hat{u}\in L^{2}(\mathbb{R}^{n})\}$.

	\item $L^{2,\sigma}(\mathbb{R}^{n})=\{u(x):(1+\arrowvert x\arrowvert^{2})^{\frac{\sigma}{2}}u(x)\in L^{2}(\mathbb{R}^{n})\}$; $H^{\kappa,\sigma}(\mathbb{R}^{n})=\{u(x):D^{\alpha}u\in L^{2,\sigma}(\mathbb{R}^{n}), 0\leqslant\arrowvert\alpha\arrowvert\leqslant\kappa\}.$
	
	\item
	$\dot{B}_{p,q}^{s}(\mathbb{R}^{n})=\{u\in\mathscr{S}'(\mathbb{R}^{n}): \|u\|_{\dot{B}_{p,q}^{s}}<\infty\}$, with $\|u\|_{\dot{B}_{p,q}^{s}}=\mathop{\sup}\limits_{j\in\mathbb{Z}}2^{sj}\|\mathscr{F}^{-1}(\varphi_{j}\hat{u})\|_{L^{p}}$ when $q=\infty$,  
	$\|u\|_{\dot{B}_{p,q}^{s}}=(\sum\limits_{j=-\infty}^{+\infty}(2^{sj}\|\mathscr{F}^{-1}(\varphi_{j}\hat{u})\|_{L^{p}})^{q})^{\frac{1}{q}}$ when $q<\infty$,  where $\varphi_{j}(\xi)=\eta(\frac{\xi}{2^{j}})-\eta(\frac{\xi}{2^{j-1}})$, $\eta(\xi)=1$ if $\arrowvert\xi\arrowvert\leqslant1$ and $\eta(\xi)=0$ if $\arrowvert\xi\arrowvert\geqslant2$,  $s\in\mathbb{R}$ and $1\leqslant p,q\leqslant\infty$.
	
	\item	
	$\lceil k\rceil$, where $k\in\mathbb{R}$,  means to take an integer upward, i.e. the smallest integer greater than this number.
	
	\item $[\mathbb{A},\mathbb{B}]:=\mathbb{A}\mathbb{B}-\mathbb{B}\mathbb{A}$, $\mathbb{A}$ and $\mathbb{B}$ are any two operators.
	
	\item $\|\mathbb{A}u\|_{L^{2}(\mathbb{R}^{n})}\sim\|\mathbb{B}u\|_{L^{2}(\mathbb{R}^{n})}$ means that there exist constants $C_{1}$, $C_{2}$ such that $\|\mathbb{A}u\|_{L^{2}}\leqslant C_{1}\|\mathbb{B}u\|_{L^{2}}$ and $\|\mathbb{B}u\|_{L^{2}}\leqslant C_{2}\|\mathbb{A}u\|_{L^{2}}$, where $\mathbb{A}$ and $\mathbb{B}$ are any two operators, $u\in\mathscr{S}(\mathbb{R}^{n})$.
	
	\item
	$\arrowvert J(t)\arrowvert^{s}=M(t)(-t^{2}\Delta)^{\frac{s}{2}}M(-t)$,  where $M(t)=e^{\frac{i\arrowvert x\arrowvert^{2}}{4t}}$, $s>0$;
	$\arrowvert J_{A}(t)\arrowvert^{s}=M(t)(-t^{2}\Delta_{A})^{\frac{s}{2}}M(-t)$, where $M(t)=e^{\frac{i\arrowvert x\arrowvert^{2}}{4t}}$, $s>0$.
	
	\item
	$R_{0}(\tau)=(\tau-\Delta)^{-1}$,  $R_{0}^{s}(\tau)=(\tau+(-\Delta)^{\frac{s}{2}})^{-1}$ for $\tau>0$, $s>0$.
	
	\item
	$R_{A}(\tau)=(\tau-\Delta_{A})^{-1}$, $R_{A}^{s}(\tau)=(\tau+(-\Delta_{A})^{\frac{s}{2}})^{-1}$ for $\tau>0$, $s>0$.

	\item $V_{x}=(-\Delta_{A})^{\frac{s}{2}}-(-\Delta)^{\frac{s}{2}}$, $s>0$.
	
	\item
	$F^{A}u(\xi)=\mathscr{F}((I+V_{x}R_0^{s}(\tau ))^{-1}u)(\xi)$ for $\tau>0$, $s>0$;
	$F_\pm^{A}u(\xi)=\mathscr{F}((I+V_{x}R_0^{s}(\lambda\pm i0 ))^{-1}u)(\xi)$ almost everywhere in $M_\lambda=\{\xi\in\mathbb{R}^n;~\arrowvert\xi\arrowvert^s=\lambda\}$.
	
	\item
	$\mathscr{H}_{0}$ is the self-adjoint realization of $(-\Delta)^{\frac{s}{2}}$ for any $s>0$ in $L^{2}(\mathbb{R}^{n})$;
	$\mathscr{H}$ is the self-adjoint realization of $(-\Delta_{A})^{\frac{s}{2}}$ for any $s>0$ in $L^{2}(\mathbb{R}^{n})$.
	
\end{itemize}

\bigskip
\section{The dispersive decay of the linear magnetic Schr\"{o}dinger equation}

In this section, we prove Theorem \ref{theorem1.1}. The main tools are the distorted Fourier transforms and the Strichartz estimates. Combining the interwining properties of the distorted Fourier transforms and the estimates for commutators of operators, we obtain the decay bound. Before that, we need some preparations.

\subsection{Preliminaries}
First we introduce the dilation operator and the multiplier operator for every $t>0$
\begin{equation*}
	D(t)u(x)=(2it)^{-\frac{n}{2}}u(\frac{x}{2t}),
\end{equation*}
\begin{equation}\label{2.1}
	M(t)u(x)=e^{\frac{i\arrowvert x\arrowvert^{2}}{4t}}u(x),
\end{equation}
then
\begin{equation*}
	e^{it\Delta}=M(t)D(t)\mathscr{F}^{-1}M(t).
\end{equation*}
Furthermore, we have
\begin{equation*}
	e^{it\Delta}x_{j}e^{-it\Delta}=M(t)2ti\partial x_{j}M(-t).
\end{equation*}
Therefore, we introduce the operators
\begin{equation*}
	J_{j}=M(t)2ti\partial x_{j}M(-t)=2tie^{\frac{i\arrowvert x\arrowvert^{2}}{4t}}\partial x_{j}e^{-\frac{i\arrowvert x\arrowvert^{2}}{4t}}=2ti\partial x_{j}+x_{j},
\end{equation*}
and we have
$$[i\partial_{t}+\Delta,J_{j}]=0,$$
where $j=1,...,n$.

By direct calculation, it follows that
\begin{flalign}
	&[i\partial_{t}+\Delta,M(t)]=M(t)(\frac{ni}{2t}+\frac{ix\cdot \nabla}{t}),&\nonumber\\ &[i\partial_{t}+\Delta,M(-t)]=M(-t)(-\frac{ni}{2t}-\frac{\arrowvert x\arrowvert^{2}}{2t^{2}}-\frac{ix\cdot \nabla}{t});&\label{2.2}
\end{flalign}
and
\begin{flalign}
	&[i\nabla\cdot A,M(t)]=[iA\cdot\nabla,M(t)]=-\frac{x\cdot A}{2t}M(t),\ [\arrowvert A\arrowvert^{2},M(t)]=0,&\nonumber\\ &[i\partial_{t}+\Delta_{A},(-t^{2}\Delta_{A})^{\frac{s}{2}}]=\frac{is}{t}(-t^{2}\Delta_{A})^{\frac{s}{2}}.&\label{2.3}
\end{flalign}
Applying \eqref{2.2} and \eqref{2.3}, we obtain that
\begin{flalign}\label{2.4}
	&[i\partial_{t}+\Delta_{A},M(t)]\nonumber\\
	=&\frac{\arrowvert x\arrowvert^{2}}{4t^{2}}M(t)+M(t)(\frac{ni}{2t}-\frac{\arrowvert x\arrowvert^{2}}{4t^{2}}+\frac{ix\cdot \nabla}{t})-\frac{x\cdot A}{t}M(t).
\end{flalign}
Similarly, we have
\begin{flalign}\label{2.5}
	&[i\partial_{t}+\Delta_{A},M(-t)]\nonumber\\
	=&-\frac{\arrowvert x\arrowvert^{2}}{4t^{2}}M(-t)+M(-t)(-\frac{ni}{2t}-\frac{\arrowvert x\arrowvert^{2}}{4t^{2}}-\frac{ix\cdot \nabla}{t})+\frac{x\cdot A}{t}M(-t).
\end{flalign}

Next we introduce the following two operators for any $s>0$
\begin{equation}\label{2.6}
	\arrowvert J(t)\arrowvert^{s}:=M(t)(-t^{2}\Delta)^{\frac{s}{2}}M(-t),
\end{equation}
\begin{equation}\label{2.7}
	\arrowvert J_{A}(t)\arrowvert^{s}:=M(t)(-t^{2}\Delta_{A})^{\frac{s}{2}}M(-t).
\end{equation}
\begin{lemma}\label{lemma2.1}
	For any $s>0$ and $t>0$, we have
	$$[i\partial_{t}+\Delta_{A},\arrowvert J_{A}(t)\arrowvert^{s}]=it^{s-1}M(t)V(s)M(-t),$$
	where
	\begin{equation}\label{2.8}
		V(s)=s(-\Delta_{A})^{\frac{s}{2}}+[x\cdot \nabla,(-\Delta_{A})^{\frac{s}{2}}]-i[(-\Delta_{A})^{\frac{s}{2}},x\cdot A].
	\end{equation}
\end{lemma}
\begin{proof}
	Applying \eqref{2.4} and\eqref{2.5}, it follows that
	\begin{equation*}
		\begin{split}
			&[i\partial_{t}+\Delta_{A},\arrowvert J_{A}(t)\arrowvert^{s}]\\
			=&[i\partial_{t}+\Delta_{A},M(t)](-t^{2}\Delta_{A})^{\frac{s}{2}}M(-t)+M(t)[i\partial_{t}+\Delta_{A},(-t^{2}\Delta_{A})^{\frac{s}{2}}]M(-t)\\
			&+M(t)(-t^{2}\Delta_{A})^{\frac{s}{2}}[i\partial_{t}+\Delta_{A},M(-t)]\\			
			=&\frac{ni}{2t}\arrowvert J_{A}(t)\arrowvert^{s}+\frac{i}{t}M(t)x\cdot \nabla(-t^{2}\Delta_{A})^{\frac{s}{2}}M(-t)-\frac{x\cdot A}{t}M(t)(-t^{2}\Delta_{A})^{\frac{s}{2}}M(-t)\\
			&+\frac{is}{t}\arrowvert J_{A}(t)\arrowvert^{s}-\frac{ni}{2t}\arrowvert J_{A}(t)\arrowvert^{s}-\frac{i}{t}M(t)(-t^{2}\Delta_{A})^{\frac{s}{2}}M(-t)x\cdot\nabla\\
			&-M(t)(-t^{2}\Delta_{A})^{\frac{s}{2}}M(-t)\frac{\arrowvert x\arrowvert^{2}}{2t^{2}}+M(t)(-t^{2}\Delta_{A})^{\frac{s}{2}}\frac{x\cdot A}{t}M(-t)\\
			&+M(t)[(-t^{2}\Delta_{A})^{\frac{s}{2}},\frac{x\cdot A}{t}]M(-t)\\
		=&\frac{is}{t}\arrowvert J_{A}(t)\arrowvert^{s}+\frac{i}{t}M(t)[x\cdot \nabla,(-t^{2}\Delta_{A})^{\frac{s}{2}}]M(-t)+M(t)[(-t^{2}\Delta_{A})^{\frac{s}{2}},\frac{x\cdot A}{t}]M(-t)\\
			=&it^{s-1}(sM(t)(-\Delta_{A})^{\frac{s}{2}}M(-t)+M(t)[x\cdot \nabla,(-\Delta_{A})^{\frac{s}{2}}]M(-t)\\ &-iM(t)[(-\Delta_{A})^{\frac{s}{2}},x\cdot A]M(-t))\\
			=&it^{s-1}M(t)V(s)M(-t).
		\end{split}
	\end{equation*}
	This completes the proof of Lemma \ref{lemma2.1}.
\end{proof}
For operator $V(s)$, we will show an integral representation. To do this, we need the following lemmas:

\begin{lemma}\label{lemma2.2}
	Let $A(x)$ satisfy hypothesis \eqref{1.2}, and
	\begin{equation}\label{2.9}
		V_{1}:=(2+x\cdot\nabla)\arrowvert A\arrowvert^{2}-2i\nabla\cdot A-2ix\cdot(\nabla A)\cdot\nabla-ix\cdot(\nabla\nabla A),
	\end{equation}
	where $x\cdot(\nabla A)\cdot\nabla u=\sum\limits_{k,j=1}^{n}(x_{k}(\partial_{x_{k}}A_{k})\partial_{x_{k}}u+x_{k}(\partial_{x_{k}}A_{j})\partial_{x_{j}}u)$, and $x\cdot(\nabla\nabla A)u=\sum\limits_{k,j=1}^{n}(x_{k}(\partial^{2}_{x_{k}x_{k}}A_{k})u+x_{k}(\partial^{2}_{x_{k}x_{j}}A_{j})u)$. Then for any $\tau>0$, we have
	\begin{flalign*}
		&[x\cdot\nabla,(-\Delta_{A})(\tau-\Delta_{A_{\varepsilon}})^{-1}]\\
		=&2\tau\Delta_{A}(\tau-\Delta_{A})^{-2}+\tau(\tau-\Delta_{A})^{-1}V_{1}(\tau-\Delta_{A})^{-1}.
	\end{flalign*}
\end{lemma}
\begin{proof}
	Let us calculate that
	\begin{equation*}
		\begin{split}
			&[x\cdot\nabla,-\Delta_{A}]\\
			=&[x\cdot\nabla,-\Delta]-i[x\cdot\nabla,A\cdot\nabla]-i[x\cdot\nabla,\nabla\cdot A]+[x\cdot\nabla,\arrowvert A\arrowvert^{2}]\\		
			=&2\Delta_{A}-2i\nabla\cdot A+(2+x\cdot\nabla)\arrowvert A\arrowvert^{2}-2ix\cdot(\nabla  A)\cdot\nabla-ix\cdot(\nabla\nabla A).
		\end{split}
	\end{equation*}
	Therefore, according to the notation of $V_{1}$, it follows that
	\begin{equation}\label{2.10}
		[x\cdot\nabla,-\Delta_{A}]=2\Delta_{A}+V_{1}.
	\end{equation}
	
	From \eqref{2.10}, we obtain  that
	\begin{flalign*}
		&[x\cdot\nabla,(-\Delta_{A})(\tau-\Delta_{A})^{-1}]\nonumber\\
		=&[x\cdot\nabla,-\Delta_{A}](\tau-\Delta_{A})^{-1}+\Delta_{A}(\tau-\Delta_{A})^{-1}[x\cdot\nabla,-\Delta_{A}](\tau-\Delta_{A})^{-1}\nonumber\\
		=&2\tau\Delta_{A}(\tau-\Delta_{A})^{-2}+\tau(\tau-\Delta_{A})^{-1}V_{1}(\tau-\Delta_{A})^{-1},
	\end{flalign*}
	which completes the proof of Lemma \ref{lemma2.2}.
\end{proof}

\begin{lemma}\label{lemma2.3}
	Let $A(x)$ satisfy hypothesis \eqref{1.2}, and \begin{equation}\label{2.11}
		V_{2}:=2A\cdot\nabla+2x\cdot(\nabla A)\cdot\nabla+2{\rm div} A+x\cdot(\nabla\nabla A)-\arrowvert A\arrowvert^{2}-x\cdot(\nabla A)\cdot A,
	\end{equation}
	where $x\cdot(\nabla A)\cdot Au=\sum\limits_{k,j=1}^{n}(x_{k}A_{k}(\partial_{x_{k}}A_{k})u+x_{k}A_{j}(\partial_{x_{j}}A_{k})u)$. Then for any $\tau>0$, we have
	$$[x\cdot A,(-\Delta_{A})(\tau-\Delta_{A})^{-1}]=\tau(\tau-\Delta_{A})^{-1}V_{2}(\tau-\Delta_{A})^{-1}.$$
\end{lemma}
\begin{proof}
	It is easy to calculate that
	\begin{equation*}
		\begin{split}
			&[x\cdot A,-\Delta]=2A\cdot\nabla+2x\cdot(\nabla A)\cdot\nabla+2{\rm div}A+x\cdot(\nabla\nabla A);\\
			&[x\cdot A,A\cdot\nabla]=-\arrowvert A\arrowvert^{2}-x\cdot(\nabla A)\cdot A;\\
			&[x\cdot A,{\rm div}A]=[x\cdot A,\arrowvert A\arrowvert^{2}]=0.
		\end{split}
	\end{equation*}
	Using the commutations above, we deduce that
	\begin{flalign}
		&[x\cdot A,-\Delta_{A}]\nonumber\\
		=&[x\cdot A,-\Delta]-2i[x\cdot A,A\cdot\nabla]-i[x\cdot A,{\rm div}A]+[x\cdot A,\arrowvert A\arrowvert^{2}]\nonumber\\
		=&2A\cdot\nabla+2x\cdot(\nabla A)\cdot\nabla+2{\rm div}A+x\cdot(\nabla\nabla A)-\arrowvert A\arrowvert^{2}-x\cdot(\nabla A)\cdot A.\label{2.12}
	\end{flalign}
	
	According to \eqref{2.12} and the notation of $V_{2}$, we have
	\begin{flalign}
		&[x\cdot A,(-\Delta_{A})(\tau-\Delta_{A})^{-1}]\nonumber\\
		=&[x\cdot A,-\Delta_{A}](\tau-\Delta_{A})^{-1}
		+\Delta_{A}(\tau-\Delta_{A})^{-1}
		[x\cdot A,-\Delta_{A}](\tau-\Delta_{A})^{-1}\nonumber\\
		=&V_{2}(\tau-\Delta_{A})^{-1}+\Delta_{A}(\tau-\Delta_{A})^{-1}V_{2}
		(\tau-\Delta_{A})^{-1}\nonumber\\
		=&\tau(\tau-\Delta_{A})^{-1}V_{2}(\tau-\Delta_{A})^{-1}.\nonumber
	\end{flalign}
	This completes the proof of Lemma \ref{lemma2.3}.
\end{proof}

Combining these lemmas above, we have the following integral representation of $V(s)$.
\begin{lemma}\label{lemma2.4}
	Let $A(x)$ satisfy hypothesis \eqref{1.2}, $V_{1}$ and $V_{2}$ be the operators in \eqref{2.9} and \eqref{2.11} respectively, and $V(s)$ be the operator in \eqref{2.8}. Then for $0<s<2$ and $\tau>0$, we have the integral representation of $V(s)$
	\begin{flalign}\label{2.13}
		V(s)=&c(s)\int^{\infty}_{0}\tau^{\frac{s}{2}}(\tau-\Delta_{A})^{-1}V_{1}(\tau-\Delta_{A})^{-1}d\tau\nonumber\\
		+&ic(s)\int^{\infty}_{0}\tau^{\frac{s}{2}}(\tau-\Delta_{A})^{-1}V_{2}(\tau-\Delta_{A})^{-1}d\tau,
	\end{flalign}
	where $(c(s))^{-1}=\int^{\infty}_{0}\tau^{\frac{s}{2}-1}(\tau+1)^{-1}d\tau$.
\end{lemma}
\begin{proof}
	Let us recall the formula (see \cite{I})	
	\begin{equation}\label{2.14}
		(-\Delta_{A})^{\frac{s}{2}}=c(s)(-\Delta_{A})\int^{\infty}_{0}\tau^{\frac{s}{2}-1}(\tau-\Delta_{A})^{-1}d\tau, \ 0<s<2.
	\end{equation}	
	
	Inserting \eqref{2.14} into \eqref{2.8}, we obtain that
	\begin{equation*}
		\begin{split}
			V(s)&=s(-\Delta_{A})^{\frac{s}{2}}+c(s)\int^{\infty}_{0}\tau^{\frac{s}{2}-1}[x\cdot\nabla,(-\Delta_{A})(\tau-\Delta_{A})^{-1}]d\tau\\
			&+ic(s)\int^{\infty}_{0}\tau^{\frac{s}{2}-1}[x\cdot A,(-\Delta_{A})(\tau-\Delta_{A})^{-1}]d\tau.
		\end{split}
	\end{equation*}
	Using results from Lemma \ref{lemma2.2} and Lemma \ref{lemma2.3}, it follows that
	\begin{flalign}\label{2.15}
		V(s)&=s(-\Delta_{A})^{\frac{s}{2}}+2c(s)\int^{\infty}_{0}\tau^{\frac{s}{2}}\Delta_{A}(\tau-\Delta_{A})^{-2}d\tau\nonumber\\
		&+c(s)\int^{\infty}_{0}\tau^{\frac{s}{2}}(\tau-\Delta_{A})^{-1}V_{1}(\tau-\Delta_{A})^{-1}d\tau\nonumber\\
		&+ic(s)\int^{\infty}_{0}\tau^{\frac{s}{2}}(\tau-\Delta_{A})^{-1}V_{2}(\tau-\Delta_{A})^{-1}d\tau.
	\end{flalign}
	For the second term of the right hand side of \eqref{2.15}, we have
	\begin{flalign}\label{2.16}
		&2c(s)\int^{\infty}_{0}\tau^{\frac{s}{2}}\Delta_{A}(\tau-\Delta_{A})^{-2}d\tau\nonumber\\
		=&-2c(s)\int^{\infty}_{0}\lambda\cdot\frac{s}{2}\int^{\infty}_{0}\tau^{\frac{s}{2}-1}(\tau+\lambda)^{-1}d\tau dE_{\lambda}\nonumber\\
		=&-s(-\Delta_{A})^{\frac{s}{2}}.
	\end{flalign}
	
	Hence, combining \eqref{2.15} and \eqref{2.16}, it follows that
	\begin{equation*}
		\begin{split}
			V(s)=&s(-\Delta_{A})^{\frac{s}{2}}-s(-\Delta_{A})^{\frac{s}{2}}+c(s)\int^{\infty}_{0}\tau^{\frac{s}{2}}(\tau-\Delta_{A})^{-1}V_{1}(\tau-\Delta_{A})^{-1}d\tau\\
			&+ic(s)\int^{\infty}_{0}\tau^{\frac{s}{2}}(\tau-\Delta_{A})^{-1}V_{2}(\tau-\Delta_{A})^{-1}d\tau\\
			=&c(s)\int^{\infty}_{0}\tau^{\frac{s}{2}}(\tau-\Delta_{A})^{-1}V_{1}(\tau-\Delta_{A})^{-1}d\tau\\
			&+ic(s)\int^{\infty}_{0}\tau^{\frac{s}{2}}(\tau-\Delta_{A})^{-1}V_{2}(\tau-\Delta_{A})^{-1}d\tau.
		\end{split}
	\end{equation*}
	The proof of Lemma \ref{lemma2.4} is completed.
\end{proof}

Note that $V_{1}$ and $V_{2}$ have similar terms. Hence, we only prove the estimate of $V_{1}$. Now we separate it into two parts:
\begin{flalign}
	V_{1}=&(2+x\cdot\nabla)\arrowvert A\arrowvert^{2}-2i\nabla\cdot A-2ix\cdot(\nabla A)\cdot\nabla -ix\cdot(\nabla\nabla A)\nonumber\\
	=&((2+x\cdot\nabla)\arrowvert A\arrowvert^{2}-ix\cdot(\nabla\nabla A)-2i({\rm div}A))-(2iA\cdot\nabla+2ix\cdot\nabla A\cdot\nabla)\nonumber\\
	:=&\tilde{A}-\tilde{\tilde{ A}}\cdot\nabla,\label{2.17}
\end{flalign}
where $\tilde{A}=(2+x\cdot\nabla)\arrowvert A\arrowvert^{2}-ix\cdot(\nabla\nabla A)-2i({\rm div}A)$ and $\tilde{\tilde{ A}}=2iA+2ix\cdot(\nabla A)$.

Substituting \eqref{2.17} to the following integral representation
$$
c(s)\int^{\infty}_{0}\tau^{\frac{s}{2}}(\tau-\Delta_{A})^{-1}V_{1}(\tau-\Delta_{A})^{-1}d\tau,
$$
we obtain
\begin{flalign*}
	c(s)\int^{\infty}_{0}\tau^{\frac{s}{2}}(\tau-\Delta_{A})^{-1}(\tilde{A}-\tilde{\tilde{ A}}\cdot\nabla)(\tau-\Delta_{A})^{-1}d\tau,
\end{flalign*}
hence we need to study the resolvent $R_{A}=(\tau-\Delta_{A})^{-1}$ with some weight functions and the first-order derivative of $R_{A}$.

From the resolvent identity
$$R_{A}(\tau)=R_{0}(\tau)(I+(\arrowvert A\arrowvert^{2}-iA\cdot\nabla-i\nabla\cdot A)R_{0}(\tau))^{-1},$$
it implies that the terms of free resolvent $R_{0}(\tau)$ with weight functions shall be researched. Furthermore, in order to ensure the integrability of the integral $$\int^{\infty}_{0}\tau^{\frac{s}{2}}(\tau-\Delta_{A})^{-1}(\tilde{A}-\tilde{\tilde{ A}}\cdot\nabla)(\tau-\Delta_{A})^{-1}d\tau,$$we need to obtain the power of $\tau$ from the estimates of $R_{0}(\tau)$.

\begin{lemma}\label{lemma2.5}
	Let $R_{0}(\tau)=(\tau-\Delta)^{-1}$. Then we have the weighted resolvent estimates for $\tau>0$, $1\leqslant r\leqslant2$ and $N\geqslant0$
	\begin{flalign}
		&\|R_{0}(\tau)u\|_{L^{r}}\leqslant C\tau^{-1}\|u\|_{L^{r}},\label{2.18}\\
		&\|\langle x\rangle ^{-N}R_{0}(\tau)u\|_{L^{r}}\leqslant C\tau^{-1+\frac{N}{2}}\|u\|_{L^{r}},\label{2.19}\\
		&\|R_{0}(\tau)\langle x\rangle ^{-N}u\|_{L^{r}}\leqslant C\tau^{-1+\frac{N}{2}}\|u\|_{L^{r}},\label{2.20}\\
		&\|\langle x\rangle ^{N_{1}}R_{0}(\tau)\langle x\rangle ^{-N}u\|_{L^{r}}\leqslant C\tau^{-1+\frac{N-N_{1}}{2}}\|u\|_{L^{r}},\label{2.21}
	\end{flalign}
	where $N_{1}\leqslant N$.
\end{lemma}

\begin{proof}
	By scaling, the free resolvent operator $R_{0}(\tau)$ is given by the kernel	
	\begin{flalign}\label{2.22}	
		R_{0}(\tau,\arrowvert x-y\arrowvert)=\tau^{-1+\frac{n}{2}}\int_{\mathbb{R}^{n}}\frac{e^{i\tau^{\frac{1}{2}}(x-y)\cdot\xi}}{\arrowvert\xi\arrowvert^{2}+\tau}d\xi:=\tau^{-1+\frac{n}{2}}R_{0}(1,\tau^{\frac{1}{2}}\arrowvert x-y\arrowvert).
	\end{flalign}	
	
	First, using \eqref{2.22}, we estimate that
	\begin{flalign}\label{2.23}		
		\|R_{0}(\tau)u\|_{L^{r}}^{r}
		=&\tau^{-r-\frac{n}{2}+\frac{n}{2}}\int_{\mathbb{R}^{n}}\arrowvert R_{0}(1)u(x)\arrowvert^{r}dx\nonumber\\	
		\leqslant&C\tau^{-r}\|u\|_{L^{r}}^{r}
	\end{flalign}	
	where $R_{0}(1)=(1-\Delta)^{-1}$. The last inequality of \eqref{2.23} is true since that $R_{0}(1)$ is bounded operator from $L^{r}(\mathbb{R}^{n})$ into  $L^{r}(\mathbb{R}^{n})$. Therefore, from \eqref{2.23} we have
	$$\|R_{0}(\tau)u\|_{L^{r}}\leqslant C\tau^{-1}\|u\|_{L^{r}},$$
	which completes the proof of \eqref{2.18}.
	
	Next, again using \eqref{2.22}, we estimate that
	\begin{flalign}\label{2.24}	
		\|\langle x\rangle ^{-N}R_{0}(\tau)u\|_{L^{r}}^{r}=&\tau^{-r+\frac{rN}{2}-\frac{n}{2}+\frac{n}{2}}\int_{\mathbb{R}^{n}}\arrowvert \langle x\rangle^{-rN}R_{0}(1)u(x)\arrowvert^{r}dx\nonumber\\	
		\leqslant&C\tau^{-r+\frac{rN}{2}}\|u\|_{L^{r}}^{r},
	\end{flalign}
	then from \eqref{2.24} we obtain
	$$\|\langle x\rangle ^{-N}R_{0}(\tau)u\|_{L^{r}}\leqslant C\tau^{-1+\frac{N}{2}}\|u\|_{L^{r}},$$
	which completes the proof of \eqref{2.19}.
	
	Similarly, we also obtain \eqref{2.20}. Indeed, using scaling as above, the term  $\int_{\mathbb{R}^{n}}\arrowvert R_{0}(1) \langle x\rangle^{-rN}u(x)\arrowvert^{r}dx$
	is same as $\int_{\mathbb{R}^{n}}\arrowvert\langle x\rangle^{-rN} R_{0}(1)u(x)\arrowvert^{r}dx.$
	
	Finally, using \eqref{2.22}, we estimate that
	\begin{flalign}\label{2.25}	
		&\|\langle x\rangle ^{N_{1}}R_{0}(\tau)\langle x\rangle ^{-N}u\|_{L^{r}}^{r}\nonumber\\
		=&\tau^{-r+\frac{r(N-N_{1})}{2}-\frac{n}{2}}\int_{\mathbb{R}^{n}}\langle x\rangle^{rN_{1}}\arrowvert\int_{\mathbb{R}^{n}}R_{0}(1,\arrowvert x-y\arrowvert)\langle y\rangle^{-N}u(\tau^{-\frac{1}{2}}y)dy\arrowvert^{r}dx\nonumber\\	
		\leqslant&\tau^{-r+\frac{r(N-N_{1})}{2}}\|u\|_{L^{r}}^{r}\cdot\int_{\mathbb{R}^{n}}\langle x\rangle^{r(N_{1}-N)}(\int_{\mathbb{R}^{n}}\arrowvert R_{0}(1,\arrowvert x-y\arrowvert)\langle x-y\rangle^{-N}\arrowvert^{r'}dy)^{\frac{r}{r'}}dx\nonumber\\
		\leqslant&C\tau^{-r+\frac{r(N-N_{1})}{2}}\|u\|_{L^{r}}^{r},\ \frac{1}{r}+\frac{1}{r'}=1.	
	\end{flalign}
	In fact, the resolvent $R_{0}(\tau)$ is given by the kernel
	\begin{flalign*}
		\frac{i}{4}(\frac{i\tau^{\frac{1}{2}}}{2\pi\arrowvert x-y\arrowvert})^{\frac{n}{2}-1}H_{\frac{n}{2}-1}^{(1)}(i\tau^{\frac{1}{2}}\arrowvert x-y\arrowvert),
	\end{flalign*}
	where $H_{\frac{n}{2}-1}^{(1)}$ is the first Hankel function. And we have the following estimates (see \cite{AS}),
	\begin{equation*}
		\arrowvert H_{\frac{n}{2}-1}^{(1)}(i\tau^{\frac{1}{2}}\arrowvert x-y\arrowvert)\arrowvert\leqslant
		\left\{
		\begin{split}
			&C(\tau^{\frac{1}{2}}\arrowvert x-y\arrowvert)^{-(\frac{n}{2}-1)},\, \qquad\   \tau^{\frac{1}{2}}\arrowvert x-y\arrowvert\leqslant1;\\
			&C(\tau^{\frac{1}{2}}\arrowvert x-y\arrowvert)^{-\frac{1}{2}}e^{-\tau^{\frac{1}{2}}\arrowvert x-y\arrowvert},\ \tau^{\frac{1}{2}}\arrowvert x-y\arrowvert>1.
		\end{split}
		\right.
	\end{equation*}
	Therefore, from \eqref{2.25} we obtain that
	$$\|\langle x\rangle ^{N_{1}}R_{0}(\tau)\langle x\rangle ^{-N}u\|_{L^{r}}\leqslant C\tau^{-1+\frac{N-N_{1}}{2}}\|u\|_{L^{r}},$$
	which completes the proof of \eqref{2.21}.
	
	Combining \eqref{2.23}, \eqref{2.24} and \eqref{2.25}, \eqref{2.18}-\eqref{2.21} follow. Hence, the proof of Lemma \ref{lemma2.5} is completed.
\end{proof}

\begin{remark}\label{remark2.1}
	When $0<\tau<1$, from \eqref{2.19}-\eqref{2.21}, we know that we can use the decay of negative weight functions to increase the power of $\tau$.
	
	When $\tau\gg1$, \eqref{2.18} is the resolvent estimate with the optimal decay rate $\tau^{-1}$.
\end{remark}

By applying these results above of resolvent $R_{0}(\tau)$, we can obtain the following estimates of $V(s)$.
\begin{lemma}\label{lemma2.6}
	When $n\leqslant3$, let $V(s)$ be the operator in  Lemma \ref{lemma2.1}, for any  $1\leqslant r\leqslant2$ and $\frac{2}{r}-1<s<1+\alpha$, we have
	\begin{equation}\label{2.26}
		\|V(s)u\|_{L_{x}^{r}}\leqslant C \|u\|_{\dot{H}_{x}^{\alpha}},\forall\ u\in \mathscr{S}(\mathbb{R}^{n}),
	\end{equation}
	where $0<\alpha<\min\{1,s\}$.
\end{lemma}
\begin{proof}
	Let us recall Lemma \ref{lemma2.4}
	\begin{equation*}
		\begin{split}
			V(s)=&c(s)\int^{\infty}_{0}\tau^{\frac{s}{2}}(\tau-\Delta_{A})^{-1}V_{1}(\tau-\Delta_{A})^{-1}d\tau\\
			+&ic(s)\int^{\infty}_{0}\tau^{\frac{s}{2}}(\tau-\Delta_{A})^{-1}V_{2}(\tau-\Delta_{A})^{-1}d\tau,
		\end{split}
	\end{equation*}
	where
	\begin{equation*}
		\begin{split}
			&V_{1}:=(2+x\cdot\nabla)\arrowvert A\arrowvert^{2}-2i\nabla\cdot A-2ix\cdot(\nabla A)\cdot\nabla -ix\cdot(\nabla\nabla A),\\
			&V_{2}:=2A\cdot\nabla+2x\cdot(\nabla A)\cdot\nabla+2{\rm div}A+x\cdot(\nabla\nabla A)-\arrowvert A\arrowvert^{2}-x\cdot(\nabla A)\cdot A.
		\end{split}
	\end{equation*}
	Therefore,
	\begin{equation*}
		\begin{split}	
			\|V(s)u\|_{L_{x}^{r}}&\leqslant c(s)\int^{\infty}_{0}\tau^{\frac{s}{2}}\|(\tau-\Delta_{A})^{-1}V_{1}(\tau-\Delta_{A})^{-1}u\|_{L_{x}^{r}}d\tau\\
			&+c(s)\int^{\infty}_{0}\tau^{\frac{s}{2}}\|(\tau-\Delta_{A})^{-1}V_{2}(\tau-\Delta_{A})^{-1}u\|_{L_{x}^{r}}d\tau.
		\end{split}
	\end{equation*}
	
	Now we estimate $$\int^{\infty}_{0}\tau^{\frac{s}{2}}\|(\tau-\Delta_{A})^{-1}V_{1}(\tau-\Delta_{A})^{-1}u\|_{L_{x}^{r}}d\tau$$ and separate $V_{1}$ into two parts $V_{1}:=\tilde{A}-\tilde{\tilde{A}}\cdot\nabla$ in \eqref{2.17}. Furthermore, the integral of $\tau$ is considered in two parts, i.e.  $\int^{\infty}_{0}=\int^{1}_{0}+\int^{\infty}_{1}$.
	
	When $\tau\in(0,1)$, we discuss the following two parts of $V_{1}$ separately:
	
	$Setp\ 1.$ Estimate for $\tilde{A}$,
	\begin{flalign}\label{2.27}
		&\|(\tau-\Delta_{A})^{-1}\tilde{A}(\tau-\Delta_{A})^{-1}u\|_{L_{x}^{r}}\nonumber\\
		=&\|(\tau-\Delta_{A})^{-1}\langle x\rangle^{-\frac{\beta}{2}}\cdot\tilde{A}\cdot\langle x\rangle^{-\frac{\beta}{2}}(\tau-\Delta_{A})^{-1}u\|_{L_{x}^{r}}\nonumber\\
		\leqslant&C\|(\tau-\Delta_{A})^{-1}\langle x\rangle^{-\frac{\beta}{2}}\|_{L^{2}\rightarrow L^{r}}\cdot\|\langle x\rangle^{-\frac{\beta}{2}}(\tau-\Delta_{A})^{-1}u\|_{L^{2}}.
	\end{flalign}
	
	For the first term on the right of \eqref{2.27}, using the resolvent identity
	\begin{equation}\label{2.28}
		R_{A}+R_{A}(\arrowvert A\arrowvert^{2}-iA\cdot\nabla-i\nabla\cdot A)R_{0}=R_{0},
	\end{equation}
	i.e.
	\begin{flalign*}
		R_{A}=&R_{0}(I+(\arrowvert A\arrowvert^{2}-iA\cdot\nabla-i\nabla\cdot A)R_{0})^{-1}\\
		:=&R_{0}(I+V_{3}(x)R_{0})^{-1},
	\end{flalign*}
	where $V_{3}(x):=-iA\cdot\nabla-i\nabla\cdot A+\arrowvert A\arrowvert^{2}$,
	we have
	\begin{flalign}\label{2.29}
		&\|(\tau-\Delta_{A})^{-1}\langle x\rangle^{-\frac{\beta}{2}}\|_{L^{2}\rightarrow L^{r}}\nonumber\\
		=&\|R_{0}\langle x\rangle^{-\frac{\beta}{2}}\cdot\langle x\rangle^{\frac{\beta}{2}}(I+V_{3}(x)R_{0})^{-1}\langle x\rangle^{-\frac{\beta}{2}}\|_{L^{2}\rightarrow L^{r}}\nonumber\\
		\leqslant&\|R_{0}\langle x\rangle^{-\frac{\beta}{2}}\|_{L^{2}\rightarrow L^{r}}\cdot\|\langle x\rangle^{\frac{\beta}{2}}(I+V_{3}(x)R_{0})^{-1}\langle x\rangle^{-\frac{\beta}{2}}u\|_{L^{2}}.
	\end{flalign}
	
	Applying $\|R_{0}\langle x\rangle^{-\sigma}\|_{L^{r}\rightarrow L^{r}}\leqslant C\tau^{-\frac{1}{r}}$ with $\sigma=2(1-\frac{1}{r})$ in Lemma \ref{lemma2.5}, we obtain for $1\leqslant r\leqslant2$ and $\beta>n+2\sigma\geqslant n+2$
	\begin{flalign}
		&\|R_{0}\langle x\rangle^{-\frac{\beta}{2}}\|_{L^{2}\rightarrow L^{r}}\nonumber\\
		\leqslant&C\|R_{0}\langle x\rangle^{-\sigma}\|_{L^{r}\rightarrow L^{r}}\cdot\|\langle x\rangle^{-\frac{\beta}{2}+\sigma}\|_{ L^{2}\rightarrow L^{r}}\nonumber\\
		\leqslant& C\tau^{-\frac{1}{r}}.\label{2.30}
	\end{flalign}
	
	Furthermore, applying the limiting absorption principle (see Theorem 4.1 in \cite{Ag}), we have
	\begin{equation}\label{2.31}
		R_{0}:L^{2,\frac{\beta}{2}}\rightarrow H^{2,-\frac{\beta}{2}},\ {\rm for} \  \frac{\beta}{2}>\frac{1}{2},\tau>0.
	\end{equation}
	Computing directly, we obtain that
	\begin{flalign}\label{2.32}
		\|\tilde{A}u\|_{L^{2,\frac{\beta}{2}}}\leqslant C(\int_{R^{n}}\langle x\rangle^{-2\beta}\langle x\rangle^{\beta}\arrowvert u\arrowvert^{2}dx)^{\frac{1}{2}}\leqslant C\|u\|_{L^{2,-\frac{\beta}{2}}}.
	\end{flalign}
	Therefore, combining \eqref{2.31} and \eqref{2.32}, we obtain that (also see \cite{MMS})
	\begin{equation*}
		\tilde{A}\cdot\nabla R_{0}:L^{2,\frac{\beta}{2}}\rightarrow L^{2,\frac{\beta}{2}}.
	\end{equation*}
	Define the operator $T\in B(L^{2,\frac{\beta}{2}},L^{2,\frac{\beta}{2}})$ for $u\in L^{2,\frac{\beta}{2}}$ by
	\begin{equation}\label{2.33}
		Tu=(\arrowvert A\arrowvert^{2}-iA\cdot\nabla-i\nabla\cdot A)R_{0}u,
	\end{equation}
	then $(I+T)^{-1}$ exists in $B(L^{2,\frac{\beta}{2}},L^{2,\frac{\beta}{2}})$. Indeed, using the resolvent identity \eqref{2.28} and \eqref{2.33}, it follows that for any $u\in L^{2}$ and $g:=(I+T)u$, we have
	$$R_{A}(I+T)u=R_{0}u,$$
	i.e.
	\begin{equation}\label{2.34}
		R_{A}g=R_{0}u.	
	\end{equation}
	Letting $u$ vary on $L^{2}$, it follows from \eqref{2.34} that the range of $(I+T)\Subset L^{2}$, which implies that the range of $(I+T)=L^{2,\frac{\beta}{2}}$. From this it follows by well know results on compact operators in a Hilbert space (the Fredholm-Riesz theory) that the inverse $(I+T)^{-1}$ exists in $B(L^{2,\frac{\beta}{2}},L^{2,\frac{\beta}{2}})$, that is (also see the proof of Lemma 3.5 in \cite{KK})
	\begin{equation}\label{2.35}
		\|\langle x\rangle^{\frac{\beta}{2}}(I+A\cdot\nabla R_{0})^{-1}\langle x\rangle^{-\frac{\beta}{2}}\|_{L^{2}\rightarrow L^{2}}\leqslant C.
	\end{equation}
	
	Consequently, according \eqref{2.30} and \eqref{2.35} into \eqref{2.29}, we obtain that
	\begin{equation}\label{2.36}
		\|(\tau-\Delta_{A})^{-1}\langle x \rangle^{-\frac{\beta}{2}}\|_{L^{2}\rightarrow L^{r}}\leqslant C\tau^{-\frac{1}{r}},
	\end{equation}
	where $\beta>n+2$, $1\leqslant r\leqslant2$.
	
	For the second term on the right of \eqref{2.27}, using the resolvent identity and the estiamtes of $R_{0}(\tau)$, we have
	\begin{flalign}
		&\|\langle x\rangle^{-\frac{\beta}{2}}(\tau-\Delta_{A})^{-1}u\|_{L^{2}}\nonumber\\
		=&\|\langle x\rangle^{-\frac{\beta}{2}}(R_{0}(\tau)-R_{A}V_{3}(x)R_{0}(\tau))u\|_{L^{2}}\nonumber\\
		\leqslant&\|\langle x\rangle^{-\frac{\beta}{2}}R_{0}(\tau)u\|_{L^{2}}+C\|\langle x\rangle^{-\frac{\beta}{2}}R_{A}(\langle x\rangle^{-\beta}+\langle x\rangle^{-\beta}\cdot\nabla)R_{0}(\tau)u\|_{L^{2}}\nonumber\\
		\leqslant&\|\langle x\rangle^{-\frac{\beta}{2}}R_{0}(\tau)u\|_{L^{2}}+C\|\langle x\rangle^{-\frac{\beta}{2}}R_{0}(\tau)(I+V_{3}(x)R_{0})^{-1}\langle x\rangle^{-\beta}R_{0}(\tau)u\|_{L^{2}}\nonumber\\
		&+\|\langle x\rangle^{-\frac{\beta}{2}}R_{0}(\tau)(I+V_{3}(x)R_{0})^{-1}\langle x\rangle^{-\beta}\cdot\nabla R_{0}(\tau)u\|_{L^{2}}\nonumber\\
		\leqslant&\|\langle x\rangle^{-\frac{\beta}{2}}R_{0}(\tau)u\|_{L^{2}}+C\|\langle x\rangle^{-\frac{\beta}{2}}R_{0}\langle x\rangle^{-\frac{\beta}{2}}\cdot\langle x\rangle^{\frac{\beta}{2}}(I+V_{3}(x)R_{0})^{-1}\langle x\rangle^{-\frac{\beta}{2}}\nonumber\\
		&\cdot\langle x\rangle^{-\frac{\beta}{2}}R_{0}(\tau)u\|_{L^{2}}+C\|\langle x\rangle^{-\frac{\beta}{2}}R_{0}\langle x\rangle^{-\frac{\beta}{2}}\cdot\langle x\rangle^{\frac{\beta}{2}}(I+V_{3}(x)R_{0})^{-1}\langle x\rangle^{-\frac{\beta}{2}}\nonumber\\
		&\cdot\langle x\rangle^{-\frac{\beta}{2}}\nabla R_{0}(\tau)u\|_{L^{2}}\nonumber\\
		\leqslant&C\|\langle x\rangle^{-\frac{\beta}{2}}R_{0}(\tau)u\|_{L^{2}}+C\|\langle x\rangle^{-\frac{\beta}{2}}\nabla R_{0}(\tau)u\|_{L^{2}}\nonumber\\
		\leqslant&C\tau^{-\frac{1}{2}}\|\langle x\rangle^{-\frac{\beta}{2}+1}u\|_{L^{2}}+C\|\langle x\rangle^{-\frac{\beta}{2}}(-\Delta)^{\frac{1-\alpha}{2}} R_{0}(\tau)(-\Delta)^{\frac{\alpha}{2}}u\|_{L^{2}}\nonumber\\
		\leqslant&C\tau^{-\frac{1}{2}}\|\langle x\rangle^{-\frac{\beta}{2}+1}\|_{L^{\frac{n}{\alpha}}}\cdot\|u\|_{L^{\frac{2n}{n-2\alpha}}}+C\|\langle x\rangle^{-\frac{\beta}{2}}\|_{L^{\frac{n}{\alpha}}}\cdot\|(-\Delta)^{\frac{1-\alpha}{2}} R_{0}(-\Delta)^{\frac{\alpha}{2}}u\|_{L^{\frac{2n}{n-2\alpha}}}\nonumber\\
		\leqslant&C\tau^{-\frac{1}{2}}\|(-\Delta)^{\frac{\alpha}{2}}u\|_{L^{2}}+C\|\nabla R_{0}(\tau)(-\Delta)^{\frac{\alpha}{2}}u\|_{L^{2}}\nonumber\\
		\leqslant&C\tau^{-\frac{1}{2}}\|(-\Delta)^{\frac{\alpha}{2}}u\|_{L^{2}},\label{2.37}
	\end{flalign}
	where $\beta>2\alpha+2$, $0<\alpha<1$.
	
	Combining \eqref{2.36} and \eqref{2.37}, we have
	\begin{equation}\label{2.38}
		\|(\tau-\Delta_{A})^{-1}\tilde{A}(\tau-\Delta_{A})^{-1}u\|_{L_{x}^{r}}\leqslant C\tau^{-(\frac{1}{r}+\frac{1}{2})}\|u\|_{\dot{H}^{\alpha}_{x}},
	\end{equation}
	where $0<\alpha<1$ and $\beta>n+2$.
	
	$Setp\ 2.$ Estimate for $ \tilde{\tilde{A}}\cdot\nabla$,
	\begin{flalign}
		&\|(\tau-\Delta_{A})^{-1}\tilde{\tilde{ A}}\cdot\nabla(\tau-\Delta_{A})^{-1}u\|_{L_{x}^{r}}\nonumber\\
		=&\|(\tau-\Delta_{A})^{-1}\langle x \rangle^{-\frac{\beta}{2}}\cdot\tilde{\tilde{ A}}\cdot\langle x \rangle^{-\frac{\beta}{2}}\nabla(\tau-\Delta_{A})^{-1}u\|_{L_{x}^{r}}\nonumber\\
		\leqslant&C\|(\tau-\Delta_{A})^{-1}\langle x \rangle^{-\frac{\beta}{2}}\|_{L^{2}\rightarrow L^{r}}\cdot\|\langle x \rangle^{-\frac{\beta}{2}}\nabla(\tau-\Delta_{A})^{-1}u\|_{L^{2}}.\label{2.39}
	\end{flalign}
	
	The first term on the right of \eqref{2.39} has been estimated in \eqref{2.36}. Now, we main estimate the second term on the right of \eqref{2.39}.
	\begin{flalign}
		&\|\langle x \rangle^{-\frac{\beta}{2}}\nabla(\tau-\Delta_{A})^{-1}u\|_{L^{2}}\nonumber\\
		=&\|\langle x\rangle^{-\frac{\beta}{2}}\nabla(R_{0}(\tau)-R_{A}V_{3}(x)R_{0}(\tau))u\|_{L^{2}}\nonumber\\
		\leqslant&\|\langle x\rangle^{-\frac{\beta}{2}}\nabla R_{0}(\tau)u\|_{L^{2}}+C\|\langle x\rangle^{-\frac{\beta}{2}}\nabla R_{A}(\langle x\rangle^{-\beta}+\langle x\rangle^{-\beta}\cdot\nabla)R_{0}(\tau)u\|_{L^{2}}\nonumber\\
		\leqslant&\|\langle x\rangle^{-\frac{\beta}{2}}\nabla R_{0}(\tau)u\|_{L^{2}}+C\|\langle x\rangle^{-\frac{\beta}{2}}\nabla R_{0}(\tau)(I+V_{3}(x)R_{0})^{-1}\langle x\rangle^{-\beta}R_{0}(\tau)u\|_{L^{2}}\nonumber\\
		&+\|\langle x\rangle^{-\frac{\beta}{2}}\nabla R_{0}(\tau)(I+V_{3}(x)R_{0})^{-1}\langle x\rangle^{-\beta}\cdot\nabla R_{0}(\tau)u\|_{L^{2}}\nonumber\\
		\leqslant&C\|\langle x\rangle^{-\frac{\beta}{2}}\nabla R_{0}(\tau)u\|_{L^{2}}+C\|\langle x\rangle^{-\frac{\beta}{2}}R_{0}(\tau)u\|_{L^{2}}\nonumber\\
		\leqslant&C\tau^{-\frac{1}{2}}\|(-\Delta)^{\frac{\alpha}{2}}u\|_{L^{2}}.\label{2.40}
	\end{flalign}
	
	Combining \eqref{2.39} and \eqref{2.40}, we have
	\begin{equation}\label{2.41}
		\|(\tau-\Delta_{A})^{-1}\tilde{\tilde{ A}}\cdot\nabla(\tau-\Delta_{A})^{-1}u\|_{L_{x}^{r}}\leqslant C\tau^{-(\frac{1}{r}+\frac{1}{2})}\|u\|_{\dot{H}^{\alpha}_{x}},
	\end{equation}
	where $0<\alpha<1$ and $\beta>n+2$.
	
	Consequently, we obtain from $Setp\ 1$ and  $Setp\ 2$ that for $\tau\in(0,1)$
	\begin{equation*}
		\|(\tau-\Delta_{A})^{-1}(\tilde{A}-\tilde{\tilde{ A}}\cdot\nabla)(\tau-\Delta_{A})^{-1}u\|_{L_{x}^{r}}\leqslant C\tau^{-(\frac{1}{2}+\frac{1}{r})}\|u\|_{\dot{H}^{\alpha}_{x}},
	\end{equation*}
	\begin{equation}\label{2.42}
		1\leqslant r\leqslant2,\ 0<\alpha<1,\  \beta>n+2.
	\end{equation}
	
	When $\tau\in(1,\infty)$, we discuss again the following two parts of $V_{1}$ separately:
	
	$Setp\ 3.$ Estimate for $\tilde{A}$,
	\begin{flalign}
		&\|(\tau-\Delta_{A})^{-1}\tilde{A}(\tau-\Delta_{A})^{-1}u\|_{L_{x}^{r}}\nonumber\\
		\leqslant&C\|(\tau-\Delta_{A})^{-1}\langle x\rangle^{-\beta}(\tau-\Delta_{A})^{-1}u\|_{L_{x}^{r}}\nonumber\\
		\leqslant&C\|(\tau-\Delta_{A})^{-1}\|_{L^{r}\rightarrow L^{r}}\cdot\|\langle x\rangle^{-\frac{\beta}{2}} \|_{L^{2}\rightarrow L^{r}}\cdot\|\langle x\rangle^{-\frac{\beta}{2}}(\tau-\Delta_{A})^{-1}u\|_{L^{2}}.\label{2.43}
	\end{flalign}
	
In fact, for $\tau>0$, from Lemma \ref{lemma2.5}, we have
\begin{flalign}\label{2.44}
\|(\tau-\Delta)^{-1}u\|_{L^{r}}\leqslant C \tau^{-1}\|u\|_{L^{r}},\ i. e.,\ \|u\|_{L^{r}}\leqslant C \tau^{-1}\|(\tau-\Delta)u\|_{L^{r}}.
\end{flalign}
Furthermore, when $\tau>1$, it's easy to obtain
	\begin{flalign*}
		&\|(\tau-\Delta_{A})u\|_{L^{r}}\nonumber\\
		=&\|(\tau-\Delta)u-iA\cdot\nabla u-i\nabla\cdot Au+\arrowvert A\arrowvert^{2}u\|_{L^{r}}\nonumber\\
		\geqslant&\|(\tau-\Delta)u\|_{L^{2}}-C\|A\cdot\nabla u\|_{L^{2}}-C\|\arrowvert A\arrowvert^{2}u\|_{L^{r}},
	\end{flalign*}
	and $$\|A\cdot\nabla u\|_{L^{r}}\leqslant C\|\nabla u\|_{L^{r}}\leqslant C\|(-\Delta) u\|_{L^{r}}+C\|u\|_{L^{r}},$$
	Hence, we have
	\begin{flalign}\label{2.45}
		\|(\tau-\Delta_{A})u\|_{L^{r}}\geqslant &C\|(\tau-\Delta)u\|_{L^{2}}-C\|(-\Delta) u\|_{L^{r}}-C\|u\|_{L^{r}}\nonumber\\
		\geqslant&C\|(\tau-\Delta)u\|_{L^{r}}.
	\end{flalign}
	From \eqref{2.44} and \eqref{2.45}, we have
	\begin{flalign*}
		\|u\|_{L^{r}}\leqslant C \tau^{-1}\|(\tau-\Delta)u\|_{L^{r}}\leqslant C\tau^{-1} \|(\tau-\Delta_{A})u\|_{L^{r}},
	\end{flalign*}
	i. e.
	\begin{flalign*}
		\|(\tau-\Delta_{A})^{-1}u\|_{L^{r}}\leqslant C \tau^{-1}\|u\|_{L^{r}},
	\end{flalign*}
	then
	\begin{flalign}\label{2.46}
		\|(\tau-\Delta_{A})^{-1}\|_{L^{r}\rightarrow {L^{r}}}\leqslant C \tau^{-1},
	\end{flalign}
	which completes the estimate of the first term on the right of \eqref{2.43}.
	
	For $1\leqslant r\leqslant2$, $\beta>n$,
	\begin{flalign}\label{2.47}
		\|\langle x\rangle^{-\frac{\beta}{2}}u\|_{L^{r}}\leqslant\|\langle x\rangle^{-\frac{\beta}{2}}\|_{L^{\frac{2r}{2-r}}}\|u\|_{L^{2}}\leqslant C\|u\|_{L^{2}},\ i. e.,\ \|\langle x\rangle^{-\frac{\beta}{2}}\|_{L^{2}\rightarrow L^{r}}\leqslant C,
	\end{flalign}
which completes the estimate of the second term on the right of \eqref{2.43}.
	
	For the third term on the right of \eqref{2.43}, we estimate that
	\begin{flalign}
		&\|\langle x\rangle^{-\frac{\beta}{2}}(\tau-\Delta_{A})^{-1}u\|_{L^{2}}\nonumber\\
		=&\|\langle x\rangle^{-\frac{\beta}{2}}(R_{0}(\tau)-R_{A}V_{3}(x)R_{0}(\tau))u\|_{L^{2}}\nonumber\\
		\leqslant&\|\langle x\rangle^{-\frac{\beta}{2}}R_{0}(\tau)u\|_{L^{2}}+C\|\langle x\rangle^{-\frac{\beta}{2}}R_{0}(\tau)(I+V_{3}(x)R_{0})^{-1}\langle x\rangle^{-\beta}R_{0}(\tau)u\|_{L^{2}}\nonumber\\
		&+\|\langle x\rangle^{-\frac{\beta}{2}}R_{0}(\tau)(I+V_{3}(x)R_{0})^{-1}\langle x\rangle^{-\beta}\cdot\nabla R_{0}(\tau)u\|_{L^{2}}\nonumber\\
		\leqslant&C\|\langle x\rangle^{-\frac{\beta}{2}}R_{0}(\tau)u\|_{L^{2}}+C\|\langle x\rangle^{-\frac{\beta}{2}}\|_{L^{\infty}}\cdot\|R_{0}(\tau)\langle x\rangle^{-\frac{\beta}{2}}\cdot\langle x\rangle^{\frac{\beta}{2}}\nonumber\\
		&(I+V_{3}(x)R_{0})^{-1}\langle x\rangle^{-\frac{\beta}{2}}\cdot\langle x\rangle^{-\frac{\beta}{2}}\nabla R_{0}(\tau)u\|_{L^{2}}\nonumber\\
		\leqslant&C\|\langle x\rangle^{-\frac{\beta}{2}}R_{0}(\tau)u\|_{L^{2}}+C\tau^{-\frac{1}{2}}\|\langle x\rangle^{-\frac{\beta}{2}}(-\Delta)^{\frac{1-\alpha}{2}} R_{0}(\tau)(-\Delta)^{\frac{\alpha}{2}}u\|_{L^{2}}\nonumber\\
		\leqslant&C\|\langle x\rangle^{-\frac{\beta}{2}}R_{0}(\tau)u\|_{L^{2}}+C\tau^{-\frac{1}{2}}\|\langle x\rangle^{-\frac{\beta}{2}}\|_{L^{\frac{n}{\alpha}}}\cdot\|(-\Delta)^{\frac{1-\alpha}{2}} R_{0}(\tau)(-\Delta)^{\frac{\alpha}{2}}u\|_{L^{\frac{2n}{n-2\alpha}}}\nonumber\\
		\leqslant&C\|\langle x\rangle^{-\frac{\beta}{2}}\|_{L^{\frac{n}{\alpha}}}\cdot\|R_{0}(\tau)u\|_{L^{\frac{2n}{n-2\alpha}}}+C\tau^{-\frac{1}{2}}\|\nabla R_{0}(\tau)(-\Delta)^{\frac{\alpha}{2}}u\|_{L^{2}}\nonumber\\
		\leqslant&C\tau^{-1}\|(-\Delta)^{\frac{\alpha}{2}}u\|_{L^{2}}.\label{2.48}
	\end{flalign}
	
	Combining \eqref{2.46}, \eqref{2.47} and \eqref{2.48}, we obtain that
	\begin{flalign}\label{2.49}
		\|(\tau-\Delta_{A})^{-1}\tilde{A}(\tau-\Delta_{A})^{-1}u\|_{L_{x}^{r}}\leqslant C\tau^{-2}\|(-\Delta)^{\frac{\alpha}{2}}u\|_{L^{2}},
	\end{flalign}
	where $0<\alpha<1$ and $\beta>n$.
	
	$Setp\ 4.$ Estimate for $\tilde{\tilde{ A}}\cdot\nabla$,
	\begin{flalign}
		&\|(\tau-\Delta_{A})^{-1}\tilde{\tilde{ A}}\cdot\nabla(\tau-\Delta_{A})^{-1}u\|_{L_{x}^{r}}\nonumber\\
		\leqslant&C\|(\tau-\Delta_{A})^{-1}\langle x\rangle^{-\beta}\nabla(\tau-\Delta_{A})^{-1}u\|_{L_{x}^{r}}\nonumber\\
		\leqslant&C\|(\tau-\Delta_{A})^{-1}\|_{L^{r}\rightarrow L^{r}}\cdot\|\langle x\rangle^{-\frac{\beta}{2}} \|_{L^{2}\rightarrow L^{r}}\cdot\|\langle x\rangle^{-\frac{\beta}{2}}\nabla(\tau-\Delta_{A})^{-1}u\|_{L^{2}}.\label{2.50}
	\end{flalign}
	
	The first and second terms on the right of \eqref{2.50} have been estimated in \eqref{2.46} and \eqref{2.47} respectively. Now, we estimate the third term on the right of \eqref{2.50}.
	\begin{flalign}
		&\|\langle x\rangle^{-\frac{\beta}{2}}\nabla(\tau-\Delta_{A})^{-1}u\|_{L^{2}}\nonumber\\
		=&\|\langle x\rangle^{-\frac{\beta}{2}}\nabla(R_{0}(\tau)-R_{A}V_{3}(x)R_{0}(\tau))u\|_{L^{2}}\nonumber\\
		\leqslant&\|\langle x\rangle^{-\frac{\beta}{2}}\nabla R_{0}(\tau)u\|_{L^{2}}+C\|\langle x\rangle^{-\frac{\beta}{2}}\nabla R_{0}(\tau)(I+V_{3}(x)R_{0})^{-1}\langle x\rangle^{-\beta}R_{0}(\tau)u\|_{L^{2}}\nonumber\\
		&+\|\langle x\rangle^{-\frac{\beta}{2}}\nabla R_{0}(\tau)(I+V_{3}(x)R_{0})^{-1}\langle x\rangle^{-\beta}\cdot\nabla R_{0}(\tau)u\|_{L^{2}}\nonumber\\
		\leqslant&\|\langle x\rangle^{-\frac{\beta}{2}}\nabla R_{0}(\tau)u\|_{L^{2}}+C\|\langle x\rangle^{-\frac{\beta}{2}}R_{0}(\tau)u\|_{L^{2}}+C\|\langle x\rangle^{-\frac{\beta}{2}}\nabla R_{0}(\tau)u\|_{L^{2}}\nonumber\\
		\leqslant&C\|(-\Delta)^{\frac{1-\alpha}{2}}R_{0}(\tau)(-\Delta)^{\frac{\alpha}{2}}u\|_{L^{2}}+C\|\langle x\rangle^{-\frac{\beta}{2}}\|_{L^{\frac{n}{\alpha}}}\cdot\|R_{0}(\tau)u\|_{L^{\frac{2n}{n-2\alpha}}}\nonumber\\
		\leqslant&C\tau^{-1+\frac{1-\alpha}{2}}\|(-\Delta)^{\frac{\alpha}{2}}u\|_{L^{2}}.\label{2.51}
	\end{flalign}
	
	Combining \eqref{2.46}, \eqref{2.47} and \eqref{2.51}, we have
	\begin{flalign}\label{2.52}
		\|(\tau-\Delta_{A})^{-1}\tilde{\tilde {A}}\cdot\nabla(\tau-\Delta_{A})^{-1}u\|_{L_{x}^{r}}\leqslant C\tau^{-2+\frac{1-\alpha}{2}}\|u\|_{\dot{H}^{\alpha}_{x}},
	\end{flalign}
	where $0<\alpha<1$ and $\beta>n$.
	
	Consequently, we obtain from $Setp\ 3$ and $Setp\ 4$ that
	\begin{equation*}
		\|(\tau-\Delta_{A})^{-1}(\tilde{A}-\tilde{\tilde{ A}}\cdot\nabla)(\tau-\Delta_{A})^{-1}u\|_{L_{x}^{r}}\leqslant C\tau^{-2+\frac{1-\alpha}{2}}\|u\|_{\dot{H}^{\alpha}_{x}},
	\end{equation*}
	\begin{equation}\label{2.53}
		1\leqslant r\leqslant2,\ 0<\alpha<1,\  \beta>n.
	\end{equation}
	
	Therefore, according to \eqref{2.42} and \eqref{2.53}, we obtain that
	\begin{flalign*}
		&\|V(s)u\|_{L^{r}_{x}}\nonumber\\
		\leqslant&C\int^{1}_{0}\tau^{\frac{s}{2}}\|(\tau-\Delta_{A})^{-1}(\tilde{A}-\tilde{\tilde {A}}\cdot\nabla)(\tau-\Delta_{A})^{-1}u\|_{L_{x}^{r}}d\tau\\
		&+C\int^{\infty}_{1}\tau^{\frac{s}{2}}\|(\tau-\Delta_{A})^{-1}(\tilde{A}-\tilde{\tilde {A}}\cdot\nabla)(\tau-\Delta_{A})^{-1}u\|_{L_{x}^{r}}d\tau\\
		\leqslant&C\int^{1}_{0}\tau^{\frac{s}{2}-\frac{1}{2}-\frac{1}{r}}\|u\|_{\dot{H}^{\alpha}_{x}}d\tau+C\int^{\infty}_{1}\tau^{\frac{s}{2}-2+\frac{1-\alpha}{2}}\|u\|_{\dot{H}^{\alpha}_{x}}d\tau\\
		\leqslant&C\|u\|_{\dot{H}^{\alpha}_{x}},
	\end{flalign*}
	where $1\leqslant r\leqslant2$,  $\frac{2}{r}-1<s<1+\alpha$ and $0<\alpha<\min\{1,s\}$. This completes the proof of Lemma \ref{lemma2.6}.
\end{proof}

\begin{lemma}\label{lemma2.7}
	When $n>3$, let $V(s)$ be the operator in  Lemma \ref{lemma2.1}, for $2k+1<s<2k+2$, $k=1,2,...$, we have
	\begin{equation}\label{2.54}
		\|V(s)u\|_{L_{x}^{r}}\leqslant C \|u\|_{\dot{H}_{A}^{2k+\alpha}},\ \forall\ u\in \mathscr{S}(\mathbb{R}^{n}),
	\end{equation}
	where $1\leqslant r\leqslant2$, $0<\alpha<1$ and  $\|u\|_{\dot{H}_{A}^{2k+\alpha}}:=\|(-\Delta_{A})^{\frac{2k+\alpha}{2}}u\|_{L^{2}_{x}}$.
\end{lemma}

\begin{proof}
	Let's recall some results which have been proved before:\\
	(1) the identical equation \eqref{2.8} for $0<s<2$:
	\begin{equation*}
		V(s)=s(-\Delta_{A})^{\frac{s}{2}}+[x\cdot \nabla,(-\Delta_{A})^{\frac{s}{2}}]+i[x\cdot A,(-\Delta_{A})^{\frac{s}{2}}];
	\end{equation*}
	(2) the formula about fractional Schr\"{o}dinger operators:
	\begin{equation*}
		(-\Delta_{A})^{\frac{s}{2}}=c(s)(-\Delta_{A})\int_{0}^{\infty}\tau^{\frac{s}{2}-1}(\tau-\Delta_{A})^{-1}d\tau,
	\end{equation*}		
	where $(c(s))^{-1}=\int_{0}^{\infty}\tau^{\frac{s}{2}-1}(\tau+1)^{-1}d\tau$ with $0<s<2$;\\
	(3) the estimate of $V(s)$:
	$$\|V(s)u\|_{L^{r}_{x}(\mathbb{R}^{n})}\leqslant C\|u\|_{\dot{H}^{\alpha}_{x}(\mathbb{R}^{n})},\ \forall\ u\in\mathscr{S}(\mathbb{R}^{n}),\ n\leqslant3,$$
	where $1\leqslant r\leqslant2$, $0<\alpha<\min\{1,s\}$ and $\frac{2}{r}-1<s<1+\alpha$.
	
	Since $2k+1<s<2k+2$, we can split $s$ into two parts:
	\begin{equation}\label{2.55}
		(-\Delta_{A})^{\frac{s}{2}}=(-\Delta_{A})^{\frac{s-2k}{2}}\cdot(-\Delta_{A})^{k},
	\end{equation}
	and then
	\begin{equation}\label{2.56}
		(-\Delta_{A})^{\frac{s-2k}{2}}=c(s,k)(-\Delta_{A})\int_{0}^{\infty}\tau^{\frac{s-2k}{2}-1}(\tau-\Delta_{A})^{-1}d\tau,
	\end{equation}
	where $(c(s,k))^{-1}=\int_{0}^{\infty}\tau^{\frac{s-2k}{2}-1}(\tau+1)^{-1}d\tau$.
	
	Substituting \eqref{2.55} and \eqref{2.56} into identity \eqref{2.8} of $V(s)$, we obtain that
	\begin{flalign}
		V(s)=&(s-2k)(-\Delta_{A})^{\frac{s-2k}{2}}\cdot(-\Delta_{A})^{k}+2k(-\Delta_{A})^{\frac{s}{2}}\nonumber\\
		&+[x\cdot \nabla,(-\Delta_{A})^{\frac{s-2k}{2}}\cdot(-\Delta_{A})^{k}]+i[x\cdot A,(-\Delta_{A})^{\frac{s-2k}{2}}\cdot(-\Delta_{A})^{k}]\nonumber\\
		=&(s-2k)(-\Delta_{A})^{\frac{s-2k}{2}}\cdot(-\Delta_{A})^{k}+2k(-\Delta_{A})^{\frac{s}{2}}\nonumber\\
		&+c(s,k)\int_{0}^{\infty}\tau^{\frac{s-2k}{2}-1}[x\cdot \nabla,(-\Delta_{A})(\tau-\Delta_{A})^{-1}]d\tau(-\Delta_{A})^{k}\nonumber\\
		&+ic(s,k)\int_{0}^{\infty}\tau^{\frac{s-2k}{2}-1}[x\cdot A,(-\Delta_{A})(\tau-\Delta_{A})^{-1}]d\tau(-\Delta_{A})^{k}\nonumber\\
		&+(-\Delta_{A})^{\frac{s-2k}{2}}[x\cdot \nabla,(-\Delta_{A})^{k}]+i(-\Delta_{A})^{\frac{s-2k}{2}}[x\cdot A,(-\Delta_{A})^{k}]\nonumber\\
		=&c(s,k)\int_{0}^{\infty}\tau^{\frac{s-2k}{2}}(\tau-\Delta_{A})^{-1}V_{1}(\tau-\Delta_{A})^{-1}d\tau\cdot(-\Delta_{A})^{k}\nonumber\\
		&+ic(s,k)\int_{0}^{\infty}\tau^{\frac{s-2k}{2}}(\tau-\Delta_{A})^{-1}V_{2}(\tau-\Delta_{A})^{-1}d\tau\cdot(-\Delta_{A})^{k}\nonumber\\
		&+2k(-\Delta_{A})^{\frac{s}{2}}+(-\Delta_{A})^{\frac{s-2k}{2}}[x\cdot \nabla,(-\Delta_{A})^{k}]\nonumber\\
		&+i(-\Delta_{A})^{\frac{s-2k}{2}}[x\cdot A,(-\Delta_{A})^{k}],\label{2.57}
	\end{flalign}
	where $V_{1}$ and $V_{2}$ are same as \eqref{2.9} in Lemma \ref{lemma2.2} and \eqref{2.11} in Lemma \ref{lemma2.3} respectively.
	
	Notice that the first and second terms on the right side of \eqref{2.57} have same terms because of definition of $V_{1}$ and $V_{2}$, hence we only show the estimate of the first term. Because that $1<s-2k<2$, we can use the result in (2) directly and obtain that
	\begin{flalign}
		&\|c(s,k)\int_{0}^{\infty}\tau^{\frac{s-2k}{2}}(\tau-\Delta_{A})^{-1}V_{1}(\tau-\Delta_{A})^{-1}(-\Delta_{A})^{k}ud\tau\|_{L^{r}_{x}}\nonumber\\
		\leqslant&C\int_{0}^{\infty}\tau^{\frac{s-2k}{2}}\|(\tau-\Delta_{A})^{-1}V_{1}(\tau-\Delta_{A})^{-1}(-\Delta_{A})^{k}u\|_{L^{r}_{x}}d\tau\nonumber\\
		\leqslant&C\|(-\Delta)^{\frac{\alpha}{2}}(-\Delta_{A})^{k}u\|_{L^{2}_{x}}\nonumber\\
		\leqslant&C\|(-\Delta_{A})^{\frac{2k+\alpha}{2}}u\|_{L^{2}_{x}},\label{2.58}
	\end{flalign}
	where the third inequality of \eqref{2.58} is proved in the later Lemma \ref{lemma2.11}.
	
	Next, we discuss the commutators $[x\cdot \nabla,(-\Delta_{A})^{k}]$ and $[x\cdot A,(-\Delta_{A})^{k}]$. Since the method to estimate $[x\cdot \nabla,(-\Delta_{A})^{k}]$ and $[x\cdot A,(-\Delta_{A})^{k}]$ is same, and the estimate of $[x\cdot \nabla,(-\Delta_{A})^{k}]$ is more complicated, we only show the proof of the first commutator here. Now, we expand the integer order Schr\"{o}dinger operator.
	\begin{flalign*}
		&(-\Delta_{A})^{k}\nonumber\\
		=&(-\Delta-2iA\cdot\nabla+\arrowvert A\arrowvert^{2}-i{\rm div} A)^{k}\nonumber\\
		=&(-\Delta)^{k}-2kiA\cdot\nabla(-\Delta)^{k-1}+C_{1,A}(\nabla A)\cdot(-\Delta)^{k-1}\nonumber\\
		&+C_{2,A}((-\Delta) A)\cdot\nabla(-\Delta)^{k-2}+...+C_{2k,A}(\arrowvert A\arrowvert^{2}-i{\rm div} A)^{k}\nonumber\\
		=&(-\Delta)^{k}-2kiA\cdot\nabla(-\Delta)^{k-1}+R(A)+C_{2k,A}(\arrowvert A\arrowvert^{2}-i{\rm div} A)^{k},
	\end{flalign*}
	where $R(A)=\sum^{2k-1}\limits_{m=1}C_{A,m}((-\Delta)^{\frac{m}{2}}A)\cdot(-\Delta)^{k-\frac{1+m}{2}}$, $C_{A,m}$ is constant dependent on $m$ and $A(x)$, i.e., the order of differential operators in the remainder $R(A)$ are less than $2k-1$. Applying the  expanded formula above, we have
	\begin{flalign*}
		&[x\cdot \nabla, (-\Delta_{A})^{k}]u\nonumber\\
		=&[x\cdot\nabla, (-\Delta)^{k}]u-2ki[x\cdot\nabla,A\cdot\nabla(-\Delta)^{k-1}]u\nonumber\\
		&+C_{2k,A}[x\cdot\nabla,(\arrowvert A\arrowvert^{2}-i{\rm div} A)^{k}]u+[x\cdot\nabla,R(A)]u\nonumber\\
		=&-2k(-\Delta)^{k}u-iC_{k,A}A\cdot\nabla(-\Delta)^{k-1}u+C_{k,A}Au+x\cdot\nabla R(A)u\nonumber\\
		=&-2k(-\Delta_{A})^{k}u-iC_{k,A}A\cdot\nabla(-\Delta)^{k-1}u\\
		&+C_{k,A}Au+x\cdot\nabla R(A)u.
	\end{flalign*}
	Therefore, we estimate for the third and fourth terms on the right side of \eqref{2.57} that
	\begin{flalign}
		&\|2k(-\Delta_{A})^{\frac{s}{2}}u+(-\Delta_{A})^{\frac{s-2k}{2}}[x\cdot \nabla,(-\Delta_{A})^{k}]u\|_{L^{r}_{x}}\nonumber\\	
		=&\|(-\Delta_{A})^{\frac{s-2k}{2}}(C_{A,k}A\cdot\nabla(-\Delta)^{k-1}u+C_{k,A}Au+x\cdot\nabla R(A)u)\|_{L^{r}_{x}}\nonumber\\		
		\leqslant&C\|\langle x\rangle^{-\beta}\|_{L^{\frac{n}{\alpha_{1}}}_{x}}\cdot\|(-\Delta_{A})^{\frac{s-2k}{2}}\nabla(-\Delta)^{k}u\|_{L^{\frac{2n}{n-2\alpha_{1}}}_{x}}\nonumber\\
		&+C\|(-\Delta_{A})^{\frac{s-2k}{2}}u\|_{L^{2}_{x}\rightarrow L^{2}_{x}}\cdot\|\langle x\rangle^{-\beta}\|_{L^{\frac{n}{\alpha_{2}}}_{x}}\cdot\|\sum^{2k-1}_{m=1}(-\Delta)^{k-\frac{1+m}{2}}u\|_{L^{\frac{2n}{n-2\alpha_{2}}}_{x}}\nonumber\\
		\leqslant&C\|(-\Delta_{A})^{\frac{2k+\alpha}{2}}u\|_{L^{2}_{x}},\label{2.59}
	\end{flalign}	
	where $0<\alpha_{1}<1$, $1<\alpha_{2}\leqslant4k-s+\alpha_{3}$ with $0<\alpha_{3}<1$. Since $2k+1<s<2k+2$, it is possible to choose $0<\alpha<1$ such that $s-1+\alpha_{1}=s-2k+\alpha_{2}=2k+\alpha$.
	
	Therefore, combining \eqref{2.58} and \eqref{2.59}, we obtain that
	$$\|V(s)u\|_{L^{r}_{x}}\leqslant C\|(-\Delta_{A})^{\frac{2k+\alpha}{2}}u\|_{L^{2}_{x}},\ 0<\alpha<1,\ k=1,2,....$$
	That is, the proof of Lemma \ref{lemma2.7} is completed.
\end{proof}

\subsection{Construction of the fractional distorted Fourier transforms}

In this subsection, we constructe the distorted Fourier transforms at resolvent points
$$F^{A}u(\xi)=\mathscr{F}((I+V_{x}R^{s}_{0}(\tau))^{-1}u)(\xi)$$
of fractional magnetic Schr\"{o}dinger operators $(-\Delta_{A})^{\frac{s}{2}}$ for any $s>0$. This tool is very useful for the next subsection to obtain the equivalence property in $L^{2}$-norm for $(-\Delta)^{\frac{s}{2}}$ and $(-\Delta_{A})^{\frac{s}{2}}$ with $0<s<\frac{n}{2}$.

To make this structure meaningful, we need to study the operator $V_{x}=(-\Delta_{A})^{\frac{s}{2}}-(-\Delta)^{\frac{s}{2}}$ and the resolvent $R_{0}^{s}(\tau)=(\tau+(-\Delta)^{\frac{s}{2}})^{-1}$ for some $\tau>0$.

\subsubsection{Short range perturbations}
 Firstly, we prove that $V_{x}$ is a short range perturbation of $\mathscr{H}_{0}$, i.e. $V_{x}$ is a bounded and compact operator from $H^{s,-\sigma}$ into $L^{2,\sigma}$ for some $\sigma\geqslant0$ and any $s>0$.

\begin{lemma}\label{lemma2.8}
	Let $A(x)$ satisfy hypothesis \eqref{1.2}, then for $0<s<2$, we have
	\begin{flalign}
		&(-\Delta_{A})^{\frac{s}{2}}u-(-\Delta)^{\frac{s}{2}}u\nonumber\\
		=&c(s)\int^{\infty}_{0}\tau^{\frac{s}{2}}(\tau-\Delta_{A})^{-1}V_{3}(x)(\tau-\Delta)^{-1}ud\tau\nonumber\\
		=&c(s)\int^{\infty}_{0}\tau^{\frac{s}{2}}(\tau-\Delta)^{-1}V_{3}(x)(\tau-\Delta_{A})^{-1}ud\tau,\label{4.1}
	\end{flalign}
	where $V_{3}(x):=-iA\cdot\nabla-i\nabla\cdot A+\arrowvert A\arrowvert^{2}$ and  $(c(s))^{-1}=\int^{\infty}_{0}\tau^{\frac{s}{2}-1}(\tau+1)^{-1}d\tau$.		
\end{lemma}
\begin{proof}
	Let us recall the formula	
	\begin{equation*}
		(-\Delta_{A})^{\frac{s}{2}}=c(s)(-\Delta_{A})\int^{\infty}_{0}\tau^{\frac{s}{2}-1}(\tau-\Delta_{A})^{-1}d\tau,
	\end{equation*}
	where $(c(s))^{-1}=\int^{\infty}_{0}\tau^{\frac{s}{2}-1}(\tau+1)^{-1}d\tau$ with $0<s<2$ (also can see \cite{I}). Applying the integral representation above, we deduce that
	\begin{flalign*}
		&(-\Delta_{A})^{\frac{s}{2}}u=c(s)(-\Delta_{A})\int^{\infty}_{0}\tau^{\frac{s}{2}-1}(\tau-\Delta_{A})^{-1}ud\tau\\
		=&c(s)(-\Delta_{A})\int^{\infty}_{0}\tau^{\frac{s}{2}-1}((\tau-\Delta_{A})^{-1}-(\tau-\Delta)^{-1})ud\tau\\
		&+c(s)(-\Delta_{A})\int^{\infty}_{0}\tau^{\frac{s}{2}-1}(\tau-\Delta)^{-1}ud\tau\\
		=&-c(s)\int^{\infty}_{0}\tau^{\frac{s}{2}-1}(\arrowvert A \arrowvert^{2}-iA\cdot\nabla-i\nabla\cdot A)(\tau-\Delta)^{-1}ud\tau\\
		&+c(s)\int^{\infty}_{0}\tau^{\frac{s}{2}}(\tau-\Delta_{A})^{-1}(\arrowvert A \arrowvert^{2}-iA\cdot\nabla-i\nabla\cdot A)(\tau-\Delta)^{-1}ud\tau\\
		&+c(s)(-\Delta)\int^{\infty}_{0}\tau^{\frac{s}{2}-1}(\tau-\Delta)^{-1}ud\tau\\
		&+c(s)\int^{\infty}_{0}\tau^{\frac{s}{2}-1}(\arrowvert A \arrowvert^{2}-iA\cdot\nabla-i\nabla\cdot A)(\tau-\Delta)^{-1}ud\tau\\
		=&(-\Delta)^{\frac{s}{2}}u+c(s)\int^{\infty}_{0}\tau^{\frac{s}{2}}(\tau-\Delta_{A})^{-1}(\arrowvert A \arrowvert^{2}-iA\cdot\nabla-i\nabla\cdot A)(\tau-\Delta)^{-1}ud\tau.
	\end{flalign*}
	That is,
	\begin{flalign}
		&(-\Delta_{A})^{\frac{s}{2}}u-(-\Delta)^{\frac{s}{2}}u\nonumber\\
		=&c(s)\int^{\infty}_{0}\tau^{\frac{s}{2}}(\tau-\Delta_{A})^{-1}(\arrowvert A \arrowvert^{2}-iA\cdot\nabla-i\nabla\cdot A)(\tau-\Delta)^{-1}ud\tau\nonumber\\
		=&c(s)\int^{\infty}_{0}\tau^{\frac{s}{2}}(\tau-\Delta_{A})^{-1}V_{3}(x)(\tau-\Delta)^{-1}ud\tau,\label{4.2}
	\end{flalign}
	where $V_{3}(x):=-iA\cdot\nabla-i\nabla\cdot A+\arrowvert A\arrowvert^{2}$.
	
	Because $\Delta_{A}$ and $\Delta$ have the same position in the expression. Using the formula
	\begin{equation*}
		(-\Delta)^{\frac{s}{2}}=c(s)(-\Delta)\int^{\infty}_{0}\tau^{\frac{s}{2}-1}(\tau-\Delta)^{-1}d\tau
	\end{equation*}
	and the similar method above, we obtain
	\begin{flalign}
		(-\Delta_{A})^{\frac{s}{2}}u-(-\Delta)^{\frac{s}{2}}u=c(s)\int^{\infty}_{0}\tau^{\frac{s}{2}}(\tau-\Delta)^{-1}V_{3}(x)(\tau-\Delta_{A})^{-1}ud\tau.\label{4.3}
	\end{flalign}
	
	Combining \eqref{4.2} and \eqref{4.3}, \eqref{4.1} is proved. Hence the proof of Lemma \ref{lemma2.8} is completed.
\end{proof}

\begin{definition}\label{definition2.1}
	An operator $V_{x}$ is said to be a short range perturbation of $\mathscr{H}_{0}$ if $V_{x}$ maps the unit ball in $H^{s,-\sigma}$ into a precompact subset of $L^{2,\sigma}$ with some $\sigma\geqslant0$.
\end{definition}

Applying the integral representation of $V_{x}$ and Lemma \ref{lemma2.8}, we obtain the boundedness and compactness of $V_{x}$ as follows.

\begin{lemma}\label{lemma2.9}
	Let $A(x)$ satisfy hypothesis \eqref{1.2}, then $V_{x}$ is a short range perturbation of $\mathscr{H}_{0}$.
\end{lemma}
\begin{proof}
	First, we prove
	$$\langle x\rangle^{\delta}V_{x}: H^{s,-\sigma}\rightarrow H^{\alpha,-\sigma},$$
	is a bounded operator for $\delta>1$, $\sigma\geqslant0$ and $0<\alpha<\min\{1,s\}$.
	
	From the definition of space norm
	\begin{equation}\label{4.4}
		\|\langle x\rangle^{\delta}V_{x}u\|^{2}_{H^{\alpha,-\sigma}(\mathbb{R}^{n})}=\|\langle x\rangle^{\delta-\sigma}V_{x}u\|^{2}_{L^{2}}+\|\langle x\rangle^{\delta-\sigma}(-\Delta)^{\frac{\alpha}{2}}V_{x}u\|^{2}_{L^{2}},
	\end{equation}	
	we split proof into two parts. Let's recall \eqref{4.1} for $0<s<2$, we have
	\begin{flalign}
		V_{x}u=&((-\Delta_{A})^{\frac{s}{2}}-(-\Delta)^{\frac{s}{2}})u\nonumber\\
		=&c(s)\int^{\infty}_{0}\tau^{\frac{s}{2}}(\tau-\Delta_{A})^{-1}V_{3}(x)(\tau-\Delta)^{-1}ud\tau\nonumber\\
		=&c(s)\int^{\infty}_{0}\tau^{\frac{s}{2}}(\tau-\Delta)^{-1}V_{3}(x)(\tau-\Delta_{A})^{-1}ud\tau,\label{4.5}
	\end{flalign}
	where $V_{3}(x):=-iA\cdot\nabla-i\nabla\cdot A+\arrowvert A\arrowvert^{2}$ and  $(c(s))^{-1}=\int^{\infty}_{0}\tau^{\frac{s}{2}-1}(\tau+1)^{-1}d\tau$ for any $0<s<2$.
	
	For the first term on the right of \eqref{4.4}, we estimate that
	\begin{flalign}
		&\|\langle x\rangle^{\delta-\sigma}V_{x}u\|_{L^{2}(\mathbb{R}^{n})}\nonumber\\
		=&\|\langle x\rangle^{\delta-\sigma}c(s)\int^{\infty}_{0}\tau^{\frac{s}{2}}(\tau-\Delta)^{-1}V_{3}(x)(\tau-\Delta_{A})^{-1}ud\tau\|_{L^{2}(\mathbb{R}^{n})}\nonumber\\
		\leqslant&C\int^{\infty}_{0}\tau^{\frac{s}{2}}\|\langle x\rangle^{\delta-\sigma}(\tau-\Delta)^{-1}\langle x\rangle^{-\beta}(\tau-\Delta_{A})^{-1}u\|_{L^{2}(\mathbb{R}^{n})}d\tau\nonumber\\
		+&C\int^{\infty}_{0}\tau^{\frac{s}{2}}\|\langle x\rangle^{\delta-\sigma}(\tau-\Delta)^{-1}\langle x\rangle^{-\beta}\nabla(\tau-\Delta_{A})^{-1}u\|_{L^{2}(\mathbb{R}^{n})}d\tau.\label{4.6}
	\end{flalign}	
	
	Because the integral of $\tau$, we divide it into two parts, i.e. $\int_{0}^{\infty}=\int_{0}^{1}+\int_{1}^{\infty}$.
	
	$Step\ 1$. When $\tau\in(0,1)$, for the first term on the right of \eqref{4.6}, we have
	\begin{flalign}
		&\|\langle x\rangle^{\delta-\sigma}(\tau-\Delta)^{-1}\langle x\rangle^{-\beta}(\tau-\Delta_{A})^{-1}u\|_{L^{2}}\nonumber\\
		\leqslant&\|\langle x\rangle^{\delta-\sigma}(\tau-\Delta)^{-1}\langle x\rangle^{-\frac{\beta}{2}}\|_{L^{2}\rightarrow L^{2}}\cdot\|\langle x\rangle^{-\frac{\beta}{2}}(\tau-\Delta_{A})^{-1}u\|_{L^{2}}\nonumber\\
		\leqslant&C\tau^{-1+\frac{\frac{\beta}{2}-\delta+\sigma}{2}}\|\langle x\rangle^{-\frac{\beta}{2}}(R_{0}-R_{A}V_{3}(x)R_{0})u\|_{L^{2}}\nonumber\\
		\leqslant&C\tau^{-1+\frac{\frac{\beta}{2}-\delta+\sigma}{2}}(\|\langle x\rangle^{-\frac{\beta}{2}}R_{0}u\|_{L^{2}}+\|\langle x\rangle^{-\frac{\beta}{2}}R_{A}V_{3}(x)R_{0}u\|_{L^{2}})\nonumber\\
		\leqslant&C\tau^{-1+\frac{\frac{\beta}{2}-\delta+\sigma}{2}}(\tau^{-1+\frac{\frac{\beta}{2}-\sigma}{2}}\|u\|_{L^{2,-\sigma}}+\|\langle x\rangle^{-\frac{\beta}{2}} R_{0}(-\Delta)^{\frac{\alpha}{2}}\langle x\rangle^{\sigma}\cdot \langle x\rangle^{-\sigma}u\|_{L^{2}})\nonumber\\
		\leqslant&C\tau^{-1+\frac{\frac{\beta}{2}-\delta+\sigma}{2}}(\tau^{-1+\frac{\frac{\beta}{2}-\sigma}{2}}\|u\|_{L^{2,-\sigma}}+\tau^{-1+\frac{\frac{\beta}{2}-\sigma}{2}}\|u\|_{H^{\alpha,-\sigma}})\nonumber\\
		\leqslant&C\tau^{-2+\frac{\beta-\delta}{2}}\|u\|_{H^{\alpha,-\sigma}}.\label{4.7}
	\end{flalign}
	
	For the second term on the right of \eqref{4.6}, we have
	\begin{flalign}
		&\|\langle x\rangle^{\delta-\sigma}(\tau-\Delta)^{-1}\langle x\rangle^{-\beta}\nabla(\tau-\Delta_{A})^{-1}u\|_{L^{2}}\nonumber\\
		\leqslant&\|\langle x\rangle^{\delta-\sigma}(\tau-\Delta)^{-1}\langle x\rangle^{-\frac{\beta}{2}}\|_{L^{2}\rightarrow L^{2}}\cdot\|\langle x\rangle^{-\frac{\beta}{2}}\nabla(\tau-\Delta_{A})^{-1}u\|_{L^{2}}\nonumber\\
		\leqslant&C\tau^{-1+\frac{\frac{\beta}{2}-\delta+\sigma}{2}}\|\langle x\rangle^{-\frac{\beta}{2}}\nabla(R_{0}-R_{A}V_{3}(x)R_{0})u\|_{L^{2}}\nonumber\\
		\leqslant&C\tau^{-1+\frac{\frac{\beta}{2}-\delta+\sigma}{2}}(\|\langle x\rangle^{-\frac{\beta}{2}}\nabla R_{0}u\|_{L^{2}}+\|\langle x\rangle^{-\frac{\beta}{2}}R_{0}u\|_{L^{2}})\nonumber\\
		\leqslant&C\tau^{-1+\frac{\frac{\beta}{2}-\delta+\sigma}{2}}(\tau^{-1+\frac{\frac{\beta}{2}-\sigma}{2}}\|u\|_{H^{\alpha,-\sigma}}+\tau^{-1+\frac{\frac{\beta}{2}-\sigma}{2}}\|u\|_{L^{2,-\sigma}})\nonumber\\
		\leqslant&C\tau^{-2+\frac{\beta-\delta}{2}}\|u\|_{H^{\alpha,-\sigma}}.\label{4.8}
	\end{flalign}
	
	Substitute \eqref{4.7} and \eqref{4.8} into \eqref{4.6}, we obtain that
	\begin{flalign}
		&\int^{1}_{0}\tau^{\frac{s}{2}}\|\langle x\rangle^{\delta-\sigma}(\tau-\Delta)^{-1}V_{3}(x)(\tau-\Delta_{A})^{-1}u\|_{L^{2}(\mathbb{R}^{n})}d\tau\nonumber\\
		\leqslant&C\int^{1}_{0}\tau^{\frac{s}{2}-2+\frac{\beta-\delta}{2}}d\tau\cdot\|u\|_{H^{\alpha,-\sigma}}\nonumber\\
		\leqslant&C\|u\|_{H^{\alpha,-\sigma}},\label{4.9}
	\end{flalign}
	where the last line of \eqref{4.9} is valid as long as the conditions are met, which is $\beta\geqslant2+\delta-\frac{n}{2}$.
	
	$Step\ 2$. When $\tau>1$, for the first term on the right of \eqref{4.6}, we have
	\begin{flalign}
		&\|\langle x\rangle^{\delta-\sigma}(\tau-\Delta)^{-1}\langle x\rangle^{-\beta}(\tau-\Delta_{A})^{-1}u\|_{L^{2}}\nonumber\\
		\leqslant&\|\langle x\rangle^{\delta-\sigma}(\tau-\Delta)^{-1}\langle x\rangle^{-\frac{\beta}{4}}\|_{L^{2}\rightarrow L^{2}}\cdot\|\langle x\rangle^{-\frac{\beta}{2}}\|_{L^{2}\rightarrow L^{2}}\cdot\|\langle x\rangle^{-\frac{\beta}{4}}(\tau-\Delta_{A})^{-1}u\|_{L^{2}}\nonumber\\
		\leqslant&C\tau^{-1+\frac{\frac{\beta}{4}-\delta+\sigma}{2}}\|\langle x\rangle^{-\frac{\beta}{4}}(R_{0}-R_{A}V_{3}(x)R_{0})u\|_{L^{2}}\nonumber\\
		\leqslant&C\tau^{-1+\frac{\frac{\beta}{4}-\delta+\sigma}{2}}(\|\langle x\rangle^{-\frac{\beta}{4}}R_{0}\langle x\rangle^{\sigma}\cdot \langle x\rangle^{-\sigma}u\|_{L^{2}}+\tau^{-\frac{1}{2}}\|\langle x\rangle^{-\frac{\beta}{4}}\nabla R_{0}u\|_{L^{2}})\nonumber\\
		\leqslant&C\tau^{-1+\frac{\frac{\beta}{4}-\delta+\sigma}{2}}(\tau^{-1+\frac{\frac{\beta}{4}-\sigma}{2}}\|u\|_{L^{2,-\sigma}}+\tau^{-\frac{1}{2}+(-1+\frac{\frac{\beta}{4}-\sigma+1}{2})}\|u\|_{H^{\alpha,-\sigma}})\nonumber\\
		\leqslant&C\tau^{-2+\frac{\beta-2\delta}{4}}\|u\|_{H^{\alpha,-\sigma}}.\label{4.10}
	\end{flalign}
	
	For the second term on the right of \eqref{4.6}, we have
	\begin{flalign}
		&\|\langle x\rangle^{\delta-\sigma}(\tau-\Delta)^{-1}\langle x\rangle^{-\beta}\nabla(\tau-\Delta_{A})^{-1}u\|_{L^{2}}\nonumber\\
		\leqslant&\|\langle x\rangle^{\delta-\sigma}(\tau-\Delta)^{-1}\langle x\rangle^{-\frac{\beta}{4}}\|_{L^{2}\rightarrow L^{2}}\cdot\|\langle x\rangle^{-\frac{\beta}{2}}\|_{L^{2}\rightarrow L^{2}}\cdot\|\langle x\rangle^{-\frac{\beta}{4}}\nabla(\tau-\Delta_{A})^{-1}u\|_{L^{2}}\nonumber\\
		\leqslant&C\tau^{-1+\frac{\frac{\beta}{4}-\delta+\sigma}{2}}\|\langle x\rangle^{-\frac{\beta}{2}}\nabla(R_{0}-R_{A}V_{3}(x)R_{0})u\|_{L^{2}}\nonumber\\
		\leqslant&C\tau^{-1+\frac{\frac{\beta}{4}-\delta+\sigma}{2}}(\|\langle x\rangle^{-\frac{\beta}{4}}\nabla R_{0}u\|_{L^{2}}+\tau^{-\frac{1}{2}}\|\langle x\rangle^{-\frac{\beta}{4}}\nabla R_{0}u\|_{L^{2}})\nonumber\\
		\leqslant&C\tau^{-2+\frac{\beta-2\delta+2-2\alpha}{4}}\|u\|_{H^{\alpha,-\sigma}}.\label{4.11}
	\end{flalign}
	
	Substitute \eqref{4.10} and \eqref{4.11} into \eqref{4.6}, we obtain that
	\begin{flalign}
		&\int^{\infty}_{1}\tau^{\frac{s}{2}}\|\langle x\rangle^{\delta-\sigma}(\tau-\Delta)^{-1}V_{3}(x)(\tau-\Delta_{A})^{-1}u\|_{L^{2}(\mathbb{R}^{n})}d\tau\nonumber\\
		\leqslant&C\int^{1}_{0}\tau^{\frac{s}{2}-2+\frac{\beta-2\delta+2-2\alpha}{4}}d\tau\cdot\|u\|_{H^{\alpha,-\sigma}}\nonumber\\
		\leqslant&C\|u\|_{H^{\alpha,-\sigma}},\label{4.12}
	\end{flalign}
	where the last line of \eqref{4.12} is valid as long as the conditions are met, which is $\beta<2\delta+2+2\alpha-n$.
	
	Next, for the second term on the right of \eqref{4.4}, we estimate that
	\begin{flalign}
		&\|\langle x\rangle^{\delta-\sigma}(-\Delta)^{\frac{\alpha}{2}} V_{x}u\|_{L^{2}(\mathbb{R}^{n})}\nonumber\\
		\leqslant&C\int^{\infty}_{0}\tau^{\frac{s}{2}}\|\langle x\rangle^{\delta-\sigma}(-\Delta)^{\frac{\alpha}{2}}(\tau-\Delta)^{-1}V_{3}(x)(\tau-\Delta_{A})^{-1}u\|_{L^{2}(\mathbb{R}^{n})}d\tau\nonumber\\
		\leqslant&C\int^{\infty}_{0}\tau^{\frac{s}{2}}\|\langle x\rangle^{\delta-\sigma}(-\Delta)^{\frac{\alpha}{2}}(\tau-\Delta)^{-1}\langle x\rangle^{-\beta}(\tau-\Delta_{A})^{-1}u\|_{L^{2}(\mathbb{R}^{n})}d\tau\nonumber\\
		&+C\int^{\infty}_{0}\tau^{\frac{s}{2}}\|\langle x\rangle^{\delta-\sigma}(-\Delta)^{\frac{\alpha}{2}}(\tau-\Delta)^{-1}\langle x\rangle^{-\beta}\nabla(\tau-\Delta_{A})^{-1}u\|_{L^{2}(\mathbb{R}^{n})}d\tau.\label{4.13}
	\end{flalign}
	
	$Step\ 3$. When $\tau\in(0,1)$, for the first term on the right of \eqref{4.13}, we have
	\begin{flalign}
		&\|\langle x\rangle^{\delta-\sigma}(-\Delta)^{\frac{\alpha}{2}}(\tau-\Delta)^{-1}\langle x\rangle^{-\beta}(\tau-\Delta_{A})^{-1}u\|_{L^{2}}\nonumber\\
		\leqslant&\|\langle x\rangle^{\delta-\sigma}(-\Delta)^{\frac{\alpha}{2}}(\tau-\Delta)^{-1}\langle x\rangle^{-\frac{\beta}{2}}\|_{L^{2}\rightarrow L^{2}}\cdot\|\langle x\rangle^{-\frac{\beta}{2}}(\tau-\Delta_{A})^{-1}u\|_{L^{2}}\nonumber\\
		\leqslant&C\tau^{-1+\frac{\frac{\beta}{2}-\delta+\sigma+\alpha}{2}}(\tau^{-1+\frac{\frac{\beta}{2}-\sigma}{2}}\|\langle x\rangle^{-\frac{\beta}{2}}(\tau-\Delta_{A})^{-1}u\|_{L^{2}})\nonumber\\
		\leqslant&C\tau^{-1+\frac{\frac{\beta}{2}-\delta+\sigma+\alpha}{2}}\cdot\tau^{-1+\frac{\frac{\beta}{2}-\sigma}{2}}\|u\|_{H^{\alpha,-\sigma}}\nonumber\\
		\leqslant&C\tau^{-2+\frac{\beta-\delta+\alpha}{2}}\|u\|_{H^{\alpha,-\sigma}}.\label{4.14}
	\end{flalign}
	
	For the second term on the right of \eqref{4.13}, we have
	\begin{flalign}
		&\|\langle x\rangle^{\delta-\sigma}(-\Delta)^{\frac{\alpha}{2}}(\tau-\Delta)^{-1}\langle x\rangle^{-\beta}\nabla(\tau-\Delta_{A})^{-1}u\|_{L^{2}}\nonumber\\
		\leqslant&\|\langle x\rangle^{\delta-\sigma}(-\Delta)^{\frac{\alpha}{2}}(\tau-\Delta)^{-1}\langle x\rangle^{-\frac{\beta}{2}}\|_{L^{2}\rightarrow L^{2}}\cdot\|\langle x\rangle^{-\frac{\beta}{2}}\nabla(\tau-\Delta_{A})^{-1}u\|_{L^{2}}\nonumber\\
		\leqslant&C\tau^{-1+\frac{\frac{\beta}{2}-\delta+\sigma+\alpha}{2}}\|\langle x\rangle^{-\frac{\beta}{2}}\nabla(R_{0}-R_{A}V_{3}(x)R_{0})u\|_{L^{2}}\nonumber\\
		\leqslant&C\tau^{-2+\frac{\beta-\delta+\alpha}{2}}\|u\|_{H^{\alpha,-\sigma}}.\label{4.15}
	\end{flalign}
	
	Combining \eqref{4.14} and \eqref{4.15}, we obtain that
	\begin{flalign}
		&\int^{1}_{0}\tau^{\frac{s}{2}}\|\langle x\rangle^{\delta-\sigma}(-\Delta)^{\frac{\alpha}{2}}(\tau-\Delta)^{-1}V_{3}(x)(\tau-\Delta_{A})^{-1}u\|_{L^{2}(\mathbb{R}^{n})}d\tau\nonumber\\
		\leqslant&C\int^{1}_{0}\tau^{\frac{s}{2}-2+\frac{\beta-\delta+\alpha}{2}}d\tau\cdot\|u\|_{H^{\alpha,-\sigma}}\nonumber\\
		\leqslant&C\|u\|_{H^{\alpha,-\sigma}},\label{4.16}
	\end{flalign}
	where the last line of \eqref{4.16} is valid as long as the conditions are met, which is $\beta\geqslant1+\delta-\frac{n}{2}$.
	
	$Step\ 4$. When $\tau>1$, for the first term on the right of \eqref{4.13}, we have
	\begin{flalign}
		&\|\langle x\rangle^{\delta-\sigma}(-\Delta)^{\frac{\alpha}{2}}(\tau-\Delta)^{-1}\langle x\rangle^{-\beta}(\tau-\Delta_{A})^{-1}u\|_{L^{2}}\nonumber\\
		\leqslant&\|\langle x\rangle^{\delta-\sigma}(-\Delta)^{\frac{\alpha}{2}}(\tau-\Delta)^{-1}\langle x\rangle^{-\frac{\beta}{4}}\|_{L^{2}\rightarrow L^{2}}\cdot\|\langle x\rangle^{-\frac{\beta}{2}}\|_{L^{2}\rightarrow L^{2}}\cdot\|\langle x\rangle^{-\frac{\beta}{4}}(\tau-\Delta_{A})^{-1}u\|_{L^{2}}\nonumber\\
		\leqslant&C\tau^{-1+\frac{\frac{\beta}{4}-\delta+\sigma+\alpha}{2}}\|\langle x\rangle^{-\frac{\beta}{4}}(\tau-\Delta_{A})^{-1}u\|_{L^{2}}\nonumber\\
		\leqslant&C\tau^{-2+\frac{\beta-2\delta+2\alpha}{4}}\|u\|_{H^{\alpha,-\sigma}}.\label{4.17}	
	\end{flalign}
	
	For the second term on the right of \eqref{4.13}, we have
	\begin{flalign}
		&\|\langle x\rangle^{\delta-\sigma}(-\Delta)^{\frac{\alpha}{2}}(\tau-\Delta)^{-1}\langle x\rangle^{-\beta}\nabla(\tau-\Delta_{A})^{-1}u\|_{L^{2}}\nonumber\\
		\leqslant&\|\langle x\rangle^{\delta-\sigma}(-\Delta)^{\frac{\alpha}{2}}(\tau-\Delta)^{-1}\langle x\rangle^{-\frac{\beta}{4}}\|_{L^{2}\rightarrow L^{2}}\cdot\|\langle x\rangle^{-\frac{\beta}{4}}\nabla(\tau-\Delta_{A})^{-1}u\|_{L^{2}}\nonumber\\
		\leqslant&C\tau^{-2+\frac{\beta-2\delta+2\alpha}{4}}\|u\|_{H^{\alpha,-\sigma}}.\label{4.18}
	\end{flalign}
	
	Combining \eqref{4.17} and \eqref{4.18}, we obtain that
	\begin{flalign}
		&\int^{\infty}_{1}\tau^{\frac{s}{2}}\|\langle x\rangle^{\delta-\sigma}(-\Delta)^{\frac{\alpha}{2}}(\tau-\Delta)^{-1}V_{3}(x)(\tau-\Delta_{A})^{-1}u\|_{L^{2}(\mathbb{R}^{n})}d\tau\nonumber\\
		\leqslant&C\int^{1}_{0}\tau^{\frac{s}{2}-2+\frac{\beta-2\delta+2\alpha}{4}}d\tau\cdot\|u\|_{H^{\alpha,-\sigma}}\nonumber\\
		\leqslant&C\|u\|_{H^{\alpha,-\sigma}},\label{4.19}
	\end{flalign}
	where the last line of \eqref{4.19} is valid as long as the conditions are met, which is $\beta<2\delta+2\alpha-n$.
	
	The estimates of $Step\ 1$-$Step\ 4$  can be obtained
	$$\|\langle x\rangle^{\delta}V_{x}u\|_{H^{\alpha,-\sigma}(\mathbb{R}^{n})}\leqslant C\|u\|_{H^{s,-\sigma}}, \ 0<\alpha<\min\{1,s\}.$$
	
	Furthermore, when $n>3$ with $2k+1<s<2k+2$, $k=1,2,...$, from \eqref{2.54} in Lemma \ref{lemma2.7} and split $s=s-k+k$, using the similar methods as $n\leqslant3$, \eqref{4.4} can also be estimated. In fact, the high-dimensional and low-dimensional processing methods are the same, hence we omit the proof of this part here.
	
	Therefore, the boundedness of operator $$\langle x\rangle^{\delta}V_{x}: H^{s,-\sigma}\rightarrow H^{\alpha,-\sigma},$$
	is completed for $\delta>1$, $\sigma\geqslant0$ and $0<\alpha<\min\{1,s\}$.
	
	Finally, we shall prove the compactness of operator $V_{x}$. Recall the fact that $\mathbb{A}\circ\mathbb{B}$ is a compact operator if $\mathbb{A}$ is a compact operator and $\mathbb{B}$ is a bounded operator. Write $V_{x}=\langle x\rangle^{-\delta}\cdot\langle x\rangle^{\delta}V_{x}$ with $\delta>1$, and we have proved the boundedness of operator $\langle x\rangle^{\delta}V_{x}$, next the compactness of operator $\langle x\rangle^{-\delta}: H^{\alpha,-\sigma}\rightarrow L^{2,\sigma}$ shall be dealt with. When $\delta>1$, the operator $\langle x\rangle^{-\delta}$ is a multiplication operator, and belongs to Agmon potential class (see \cite{Ag}). Then, the compactness of operator $\langle x\rangle^{-\delta}: H^{\alpha,-\sigma}\rightarrow L^{2,\sigma}$ is obvious.
	
	Combining the boundedness of $\langle x\rangle^{\delta}V_{x}: H^{s,-\sigma}\rightarrow H^{\alpha,-\sigma}$ and the compactness of $\langle x\rangle^{-\delta}: H^{\alpha,-\sigma}\rightarrow L^{2,\sigma}$, we have the compactness of $V_{x}:H^{s,-\sigma}\rightarrow L^{2,\sigma}$. Therefore, the proof of Lemma \ref{lemma2.9} is completed.
\end{proof}

\subsubsection{The boundary values of the resolvent}
Next, the following lemmas study the resolvent operator $R_{0}^{s}(\tau)=(\tau+(-\Delta)^{\frac{s}{2}})^{-1}$.
\begin{lemma}\label{lemma2.10}
	Let $R_{0}^{s}(\tau)=(\tau+(-\Delta)^{\frac{s}{2}})^{-1}$, then $R_{0}^{s}(\tau)$ is a bounded operator from $L^{2,\sigma}$ to $H^{s,-\sigma}$ for any $s>0$, some $\sigma\geqslant0$ and $\tau>0$.
\end{lemma}
\begin{proof}
	From the Parseval identity, we have
	\begin{flalign*}	&\|R_{0}^{s}(\tau)u\|_{L^{2,-\sigma}}\leqslant C\|R_{0}^{s}(\tau)u\|_{L^{2}}=\|\mathscr{F}(R_{0}^{s}u)\|_{L^{2}}\nonumber\\
		=&\|\frac{\hat{u}}{\arrowvert\xi\arrowvert^{s}+\tau}\|_{L^{2}}\leqslant C\cdot max\{\frac{1}{\tau}, 1\}\|u\|_{L^{2}}\leqslant C\cdot max\{\frac{1}{\tau}, 1\}\|u\|_{L^{2,\sigma}}.
	\end{flalign*}
	Furthermore, using the Parseval identity again, we have
	\begin{flalign*}	&\|R_{0}^{s}(\tau)u\|_{H^{s,-\sigma}}\leqslant\|R_{0}^{s}(\tau)u\|_{H^{s}}=\|(I-\Delta)^{\frac{s}{2}}R_{0}^{s}(\tau)u\|_{L^{2}}=\|\mathscr{F}((I-\Delta)^{\frac{s}{2}}R_{0}^{s}u)\|_{L^{2}}\nonumber\\
		=&\|(1+\arrowvert\xi\arrowvert^{2})^{\frac{s}{2}}\cdot\frac{\hat{u}}{\tau+\arrowvert\xi\arrowvert^{s}}\|_{L^{2}}\leqslant C\cdot max\{\frac{1}{\tau}, 1\}\|u\|_{L^{2}}\leqslant C\cdot max\{\frac{1}{\tau}, 1\}\|u\|_{L^{2,\sigma}}.
	\end{flalign*}
	Therefore the proof of Lemma \ref{lemma2.10} is completed.
\end{proof}

Combining the property of $V_{x}$, which is a bounded operator from $H^{s,-\sigma}$ into $L^{2,\sigma}$ in Lemma \ref{lemma2.9}, and the property of $R_{0}^{s}(\tau)$, which is a bounded operator from $L^{2,\sigma}$ to $H^{s,-\sigma}$, for some $\sigma\geqslant0$ and any $s>0$, we can construct the distorted Fourier transforms at resolvent points as follows. 

\begin{lemma}\label{lemma2.11}
	If $\tau\in\mathbb{R}^{+}$ and $V_{x}$ is a short range perturbation of $\mathscr{H}_{0}$, then $(I+V_{x}R_{0}^{s}(\tau))^{-1}$ exist and is continuous in $L^{2,\sigma}$. Moreover, the map
	\begin{equation}\label{2.60}
		\mathbb{R}^+\ni\tau\mapsto(I+V_{x}R_{0}^{s}(\tau))^{-1}u\in L^{2,\sigma}
	\end{equation}
	is continuous if $u\in L^{2,\sigma}$ for any $s>0$ and some $\sigma\geqslant0$.
\end{lemma}
\begin{proof}
	When $u\in L^{2,\sigma}$ we have $R_0^{s}(\tau)u\in H^s$, and it is easy to deduce that
	\begin{equation}\label{2.61}
		R_0^{s}(\tau)u=R^{s}_{A}(\tau)(I+V_{x}R_0^{s}(\tau))u,
	\end{equation}
	where $I+V_{x}R_0^{s}(\tau)$ can be interpreted as an operator in $L^{2,\sigma}$. Thus $I+V_{x}R_0^{s}(\tau)$ has kernel $\{0\}$, which is also true when $\tau\in\mathbb{R}^{+}$. By Fredholm theory, $(I+V_{x}R_0^{s}(\tau))^{-1}$ is bounded in $L^{2,\sigma}$ and strongly continuous in $\tau\in\mathbb{R}^{+}$, where the proof needs the facts that  $\mathbb{R}^+\ni\tau\mapsto V_{x}R_0^{s}(\tau)u\in L^{2,\sigma}$ is continuous when  $u\in L^{2,\sigma}$ and $\{V_{x}R_0^{s}(\tau)u;~\|u\|_{L^{2,\sigma}}\leqslant1,\,\tau>0\}$  is precompact in $L^{2,\sigma}$ (see Lemma \ref{lemma2.9} and Lemma \ref{lemma2.10}). Therefore the proof of Lemma \ref{lemma2.11} is completed.
\end{proof}

\begin{definition}\label{definition2.2}
	If $u\in L^{2,\sigma}$ for $\sigma\geqslant0$ and $\tau>0$. then the $L^{2}$ functions defined by
	\begin{equation}\label{2.62}
		F^{A}u(\xi)=\mathscr{F}((I+V_{x}R_0^{s}(\tau ))^{-1}u)(\xi)
	\end{equation}
	is called the distorted Fourier transforms of $u$.
\end{definition}

\begin{lemma}\label{lemma2.12} (Quasi-intertwining property)
	
	If $u\in L^{2,\sigma}$ with $\sigma\geqslant0$, then we have
	\begin{equation}\label{2.63}
		\arrowvert\xi\arrowvert^{s}\mathscr{F} u=F^{A}\mathscr{H}u+\tau(F^{A}u-\mathscr{F}u),
	\end{equation}
	for $s>0$ and $\tau>0$.
\end{lemma}
\begin{proof}
	We first consider $u\in\mathscr{F}^{-1}C_c^\infty(\mathbb{R}^n\setminus\{0\})$. For any $\tau\in\mathbb{R}^{+}$, we have $(\mathscr{H}_0+\tau)u\in\mathscr{S}$ and thus $R_0^{s}(\tau)(\mathscr{H}_0+\tau)u=u$ by weak* continuity of $R_0^{s}$ and dominated convergence. Using the notation of $V_{x}$, we have $$(\mathscr{H}+\tau)u=(\mathscr{H}_{0}+\tau)u+V_{x}u.$$
	Thus
	\begin{equation*}
		(\mathscr{H}+\tau)u=(I+V_{x}R_0^{s}(\tau))(\mathscr{H}_0+\tau)u,
	\end{equation*}
	which implies
	\begin{equation*}
		(F^{A}(\mathscr{H}+\tau)u)(\xi)=(\arrowvert\xi\arrowvert^s+\tau)\hat{u}(\xi),
	\end{equation*}
	and this just means
	\begin{equation}\label{2.64}
		\arrowvert\xi\arrowvert^s\hat{u}(\xi)=F^{A}\mathscr{H}u(\xi)+\tau(F^{A}u(\xi)-\hat{u}(\xi)).
	\end{equation}
	Since $\mathscr{F}^{-1}C_c^\infty(\mathbb{R}^n\setminus\{0\})$ is dense in $H^s$ and $\mathscr{H}$ is closed, \eqref{2.64} holds when $u\in H^s$. Therefore \eqref{2.63} holds for $H^s$ is dense in $L^2$ and the proof of Lemma \ref{lemma2.12} is completed.
\end{proof}

\begin{remark}\label{remark2.2}
Similarly, we can construct the distorted Fourier transforms at spectral points, the boundness of operators $V_{x}$ and $R_0^{s}(\lambda\pm i0)$ need to be discussed, see our other paper \cite{DW} for detailed treatment. The $L^{2}$ functions defined by
\begin{equation*}
	F_\pm^{A}u(\xi)=\mathscr{F}((I+V_{x}R_0^{s}(\lambda\pm i0 ))^{-1}u)(\xi)
\end{equation*}
almost everywhere in $M_\lambda=\{\xi\in\mathbb{R}^n;~\arrowvert\xi\arrowvert^s=\lambda\}$ is called the distorted Fourier transforms of $u$. And the intertwining property of it has been proved in \cite{DW} as following
\begin{equation}\label{4.36}
F_\pm^{A}\mathscr{H}u=\arrowvert\xi\arrowvert^sF_\pm^{A}u,
\end{equation}
which will be uesd later.
\end{remark}

\subsection{Application of the distorted Fourier transforms}
In this subsection, we use the quasi-intertwining property of the distorted Fourier transforms at resolvent points, the estimates of $V_{x}$ and $R_{0}^{s}(\tau)$ to obtain the equivalence of the operators $(-\Delta)^{\frac{s}{2}}$ and $(-\Delta_{A})^{\frac{s}{2}}$ for $0<s<\frac{n}{2}$. This equivalence property is important because it is the link between operators $\arrowvert J(t)\arrowvert^{s}u$ and $\arrowvert J_{A}(t)\arrowvert^{s}u$.

Furthermore, applying the intertwining property of the distorted Fourier transforms at the spectral points in \cite{DW}, we obtain the estimate of the solutions of the linear magnetic Schr\"{o}dinger equations in $L^{\infty}$ space. This estimate is the key for the proof of Theorem\ref{theorem1.1}, which links between the $L^{\infty}$-norm of solutions and the Strichartz estimate of $\arrowvert J_{A}(t)\arrowvert^{s}u$.

\begin{lemma}\label{lemma2.13}
	Let $A(x)$ satisfy hypothesis \eqref{1.2}, then for $0<s<\frac{n}{2}$, we have
	\begin{equation}\label{2.65}
		\|(-\Delta)^{\frac{s}{2}}u\|_{L^{2}(\mathbb{R}^{n})}\sim \|(-\Delta_{A})^{\frac{s}{2}}u\|_{L^{2}(\mathbb{R}^{n})}.
	\end{equation}
\end{lemma}
\begin{proof}
	On the one hand, applying the quasi-intertwining property of the distorted Fourier transforms \eqref{2.63} in Lemma \ref{lemma2.12}, we have
	\begin{flalign}
		&\|\arrowvert\xi\arrowvert^s\mathscr{F}u\|^{2}_{L^{2}(\mathbb{R}^{n})}\nonumber\\
		=&\|F^{A}\mathscr{H}u+\tilde{\tau}(F^{A}u-\mathscr{F}u)\|^{2}_{L^{2}(\mathbb{R}^{n})}\nonumber\\
		\leqslant&\|F^{A}\mathscr{H}u\|^{2}_{L^{2}(\mathbb{R}^{n})}+\tilde{\tau}^{2}\|((I+V_{x}R^{s}_{0}(\tilde{\tau}))^{-1}-I)u\|^{2}_{L^{2}(\mathbb{R}^{n})}\nonumber\\
		=&\|F^{A}\mathscr{H}u\|^{2}_{L^{2}(\mathbb{R}^{n})}+\tilde{\tau}^{2}\|(I+V_{x}R^{s}_{0}(\tilde{\tau}))^{-1}(-V_{x}R^{s}_{0}(\tilde{\tau}))u\|^{2}_{L^{2}(\mathbb{R}^{n})}.\label{2.66}
	\end{flalign}
	
	For the second term on the right of \eqref{2.66}, we estimate that
	\begin{flalign}
		&\tilde{\tau}^{2}\|(I+V_{x}R^{s}_{0}(\tilde{\tau}))^{-1}(-V_{x}R^{s}_{0}(\tau))u\|^{2}_{L^{2}(\mathbb{R}^{n})}\nonumber\\
		\leqslant&\tilde{\tau}^{2}\|(I+V_{x}R^{s}_{0}(\tilde{\tau}))^{-1}\langle x\rangle^{-\sigma}\|_{L^{2}\rightarrow L^{2}}\cdot\|\langle x\rangle^{\sigma}V_{x}R^{s}_{0}(\tilde{\tau})u\|^{2}_{L^{2}(\mathbb{R}^{n})}\nonumber\\
		\leqslant&C\tilde{\tau}^{2}\|\langle x\rangle^{\sigma}V_{x}R^{s}_{0}(\tilde{\tau})u\|^{2}_{L^{2}(\mathbb{R}^{n})}\nonumber\\
		\leqslant&C\tilde{\tau}^{2}\|\langle x\rangle^{-\sigma}R^{s}_{0}(\tilde{\tau})u\|_{H^{s}(\mathbb{R}^{n})}^{2}\nonumber\\
		\leqslant&C\tilde{\tau}^{2}(\|\langle x\rangle^{-\sigma}R^{s}_{0}(\tilde{\tau})u\|_{L^{2}(\mathbb{R}^{n})}^{2}+\|\langle x\rangle^{-\sigma}(-\Delta)^{\frac{s}{2}}R^{s}_{0}(\tilde{\tau})u\|_{L^{2}(\mathbb{R}^{n})}^{2}).\label{2.67}
	\end{flalign}
	By scaling for the kernel of $R_{0}^{s}(\tilde{\tau})$, we have
	\begin{flalign*}
		R_{0}^{s}(\tilde{\tau},\arrowvert x-y\arrowvert)=&\int_{\mathbb{R}^{n}}\frac{e^{i(x-y)\cdot\xi}}{\arrowvert\xi\arrowvert^{s}+\tilde{\tau}}d\xi\nonumber\\
		=&\int_{\mathbb{R}^{n}}\frac{\tilde{\tau}^{-1}\cdot e^{i\tilde{\tau}^{\frac{1}{s}}(x-y)\cdot\tilde{\tau}^{-\frac{1}{s}}\xi}}{\arrowvert\tilde{\tau}^{-\frac{1}{s}}\xi\arrowvert^{s}+1}d(\tilde{\tau}^{-\frac{1}{s}}\xi\cdot\tilde{\tau}^{\frac{1}{s}})\nonumber\\
		:=&\tilde{\tau}^{-1+\frac{n}{s}}R_{0}^{s}(1,\tilde{\tau}^{-\frac{1}{s}}\arrowvert x-y\arrowvert),
	\end{flalign*}
	Then using Sobolev embedding, we have
	\begin{flalign*}
		&\|\langle x\rangle^{-\sigma}R^{s}_{0}(\tilde{\tau})u\|_{L^{2}(\mathbb{R}^{n})}+\|\langle x\rangle^{-\sigma}(-\Delta)^{\frac{s}{2}}R^{s}_{0}(\tilde{\tau})u\|_{L^{2}(\mathbb{R}^{n})}\nonumber\\
		\leqslant&\|\langle x\rangle^{-\sigma}\|_{L\frac{n}{s}}\cdot\|R^{s}_{0}(\tilde{\tau})u\|_{L^{\frac{2n}{n-2s}}(\mathbb{R}^{n})}\nonumber\\
		&+\|\langle x\rangle^{-\sigma}\|_{L^{\infty}}\cdot\|(-\Delta)^{\frac{s}{2}}R^{s}_{0}(\tilde{\tau})u\|_{L^{2}(\mathbb{R}^{n})}\nonumber\\
		\leqslant&C\tilde{\tau}^{-1+\frac{n}{s}}\|u\|_{\dot{H}^{s}},
	\end{flalign*}
	where $0<s<\frac{n}{2}$ and $\sigma>s$. Substitute this estimate into \eqref{2.67}, we have
	\begin{flalign*}
		&\tilde{\tau}^{2}\|(I+V_{x}R^{s}_{0}(\tilde{\tau}))^{-1}(-V_{x}R^{s}_{0}(\tau))u\|^{2}_{L^{2}(\mathbb{R}^{n})}\nonumber\\
		\leqslant&C\tilde{\tau}^{2}(\tilde{\tau}^{-1+\frac{n}{s}}\|u\|_{\dot{H}^{s}})^{2}\\
		=&C\tilde{\tau}^{\frac{2n}{s}}\|u\|_{\dot{H}^{s}(\mathbb{R}^{n})}^{2}.
	\end{flalign*}
	And then substitute it into \eqref{2.66}, we obtain that
	\begin{flalign*}
		\|u\|_{\dot{H}^{s}}^{2}=\|\arrowvert\xi\arrowvert^s\mathscr{F}u\|^{2}_{L^{2}(\mathbb{R}^{n})}
		\leqslant\|F^{A}\mathscr{H}u\|^{2}_{L^{2}(\mathbb{R}^{n})}+\tilde{\tau}^{\frac{2n}{s}}\|u\|_{\dot{H}^{s}(\mathbb{R}^{n})}^{2}.
	\end{flalign*}
	Choosing $\tilde{\tau}\in(0,1)$, we have
	\begin{flalign*}
		\|u\|_{\dot{H}^{s}}^{2}\leqslant C\|F^{A}\mathscr{H}u\|^{2}_{L^{2}(\mathbb{R}^{n})},
	\end{flalign*}
	i. e.
	\begin{flalign}
		\|(-\Delta)^{\frac{s}{2}}u\|_{L^{2}(\mathbb{R}^{n})}\leqslant C\|(-\Delta_{A})^{\frac{s}{2}}u\|_{L^{2}(\mathbb{R}^{n})}. \label{2.68}
	\end{flalign}
	
	On the other hand, applying the quasi-intertwining property of the distorted Fourier transforms again in \eqref{2.63}, we have
	\begin{flalign}
		&\|F^{A}\mathscr{H}u\|_{L^{2}(\mathbb{R}^{n})}\nonumber\\
		\leqslant&\|\arrowvert\xi\arrowvert^s\mathscr{F}u\|_{L^{2}(\mathbb{R}^{n})}+\tilde{\tau}\|\mathscr{F}u-F^{A}u\|_{L^{2}(\mathbb{R}^{n})}\nonumber\\
		\leqslant&C\|\arrowvert\xi\arrowvert^s\mathscr{F}u\|_{L^{2}(\mathbb{R}^{n})}+\tilde{\tau}\|(I+V_{x}R^{s}_{0}(\tilde{\tau}))^{-1}(V_{x}R^{s}_{0}(\tilde{\tau}))u\|_{L^{2}(\mathbb{R}^{n})}\nonumber\\
		\leqslant&C\|\arrowvert\xi\arrowvert^s\mathscr{F}u\|_{L^{2}(\mathbb{R}^{n})}+C\tilde{\tau}^{-1+\frac{n}{s}}\|u\|_{\dot{H}^{s}}\nonumber\\
		\leqslant&C\|\arrowvert\xi\arrowvert^s\mathscr{F}u\|_{L^{2}(\mathbb{R}^{n})}\nonumber\\
		=&C\|u\|_{\dot{H}^{s}(\mathbb{R}^{n})}.\label{2.69}
	\end{flalign}
	Since
	\begin{flalign*}
		\|u\|_{\dot{H}_{A}^{s}(\mathbb{R}^{n})}:=&\|(-\Delta_{A})^{\frac{s}{2}}u\|_{L^{2}(\mathbb{R}^{n})}\nonumber\\
		=&\|\mathscr{H}u\|_{L^{2}(\mathbb{R}^{n})}\nonumber\\
		\leqslant&C\|(I+V_{x}R^{s}_{0}(\tilde{\tau}))^{-1}\mathscr{H}u\|_{L^{2}(\mathbb{R}^{n})}\nonumber\\
		=&C\|F^{A}\mathscr{H}u\|_{L^{2}(\mathbb{R}^{n})},
	\end{flalign*}
	we substitute it into \eqref{2.69}, then obtain that
	\begin{flalign}
		\|(-\Delta_{A})^{\frac{s}{2}}u\|_{L^{2}(\mathbb{R}^{n})}\leqslant C\|(-\Delta)^{\frac{s}{2}}u\|_{L^{2}(\mathbb{R}^{n})}. \label{2.70}
	\end{flalign}
	
	Combining \eqref{2.68} and \eqref{2.70}, we have the equivalence of operators
	$$\|(-\Delta)^{\frac{s}{2}}u\|_{L^{2}(\mathbb{R}^{n})}\sim \|(-\Delta_{A})^{\frac{s}{2}}u\|_{L^{2}(\mathbb{R}^{n})}.$$
	Therefore, \eqref{2.65} follows. This completes the proof of Lemma \ref{lemma2.13}.
\end{proof}

In order to estimate the solution of the equation, we also need the properties of the distorted Fourier transforms at the spectral points and the following lemma.
\begin{lemma}\label{lemma2.14}
	Let $R_{0}^{s_{1}}(z)=((-\Delta)^{\frac{s_{1}}{2}}-z)^{-1}$, $V_{x}=(-\Delta_{A})^{\frac{s_{1}}{2}}-(-\Delta)^{\frac{s_{1}}{2}}$, and $A(x)$ satisfy hypothesis \eqref{1.2}, then for $0<s_{1}<1$, we have
	\begin{flalign}\label{2.71}	
		\|R_{0}^{s_{1}}(z)V_{x}u\|_{L^{\infty}}\leqslant C\|u\|_{L^{\infty}},
	\end{flalign}
	where $z\in\mathbb{C}^{\pm}\backslash\{0\}$.
\end{lemma}
\begin{proof}
	Denote by $K(x,y)$ the Schwartz kernel of the resolvent $R_{0}^{s_{1}}(z)$. Now we estimate $K$. Let's recall the following expression of the Green function of $(-\Delta-z)^{-1}$ (see \cite{KRS})
	\begin{flalign}\label{2.72}
		(-\Delta-z)^{-1}(x,y)=(\frac{-z}{\arrowvert x-y\arrowvert^{2}})^{\frac{n-2}{4}}K_{\frac{n-2}{2}}((-z)^{\frac{1}{2}}\arrowvert x-y\arrowvert),
	\end{flalign}
	where
	\begin{equation}
		\arrowvert K_{\frac{n-2}{2}}(z)\arrowvert\leqslant
		\left\{
		\begin{split}
			&C\arrowvert z\arrowvert^{-(\frac{n}{2}-1)},\ \arrowvert z\arrowvert\leqslant1;\nonumber\\
			&C\arrowvert z\arrowvert^{-\frac{1}{2}},\quad\quad \ \arrowvert z\arrowvert>1.
		\end{split}
		\right.
	\end{equation}
	Then
	\begin{equation}
		\arrowvert K(x,y)\arrowvert\leqslant
		\left\{
		\begin{split}
			&C\arrowvert x-y\arrowvert^{-(n-2)},\, \, \  \arrowvert x-y\arrowvert\leqslant1;\nonumber\\
			&C\arrowvert x-y\arrowvert^{-\frac{1}{2}},\quad\quad \ \arrowvert x-y\arrowvert>1.
		\end{split}
		\right.
	\end{equation}
	
	Write $z=z_{0}^{\frac{s_{1}}{2}}$, where $0<\arrowvert z_{0}\arrowvert<\pi$. We recall the following expression of the resolvent of the fractional power of the negative Laplacian (see \cite{MS1})
	\begin{flalign}\label{2.73}
		((-\Delta)^{\frac{s_{1}}{2}}-z_{0}^{\frac{s_{1}}{2}})^{-1}=&\frac{2z_{0}^{1-\frac{s_{1}}{2}}}{s_{1}}(-\Delta-z_{0})^{-1}\nonumber\\
		+&\frac{\sin\frac{s_{1}}{2}\pi}{\pi}\int_{0}^{\infty}\frac{\tau^{\frac{s_{1}}{2}}(\tau-\Delta)^{-1}}{\tau^{s_{1}}-2\tau^{\frac{s_{1}}{2}}z_{0}^{\frac{s_{1}}{2}}\cos\frac{s_{1}}{2}\pi+z_{0}^{s_{1}}}d\tau\nonumber\\
		=:&K_{1}+K_{2}.
	\end{flalign}
	First, it's easy to see that $K_{1}$ is the \textquotedblleft good term\textquotedblright\ for estimating. Indeed, from \eqref{2.73}, we have
	\begin{equation}
		\arrowvert K_{1}(x,y)\arrowvert\leqslant
		\left\{
		\begin{split}
			&C\arrowvert x-y\arrowvert^{-(n-2)},\, \, \  \arrowvert x-y\arrowvert\leqslant1;\\
			&C\arrowvert x-y\arrowvert^{-\frac{1}{2}},\quad\quad \ \arrowvert x-y\arrowvert>1.\label{2.74}
		\end{split}
		\right.
	\end{equation}
	
	Next, for $K_{2}$,  when $\arrowvert x-y\arrowvert\leqslant1$, we recall that
	\begin{flalign}\label{2.75}
		(\tau-\Delta)^{-1}(x,y)=C\int_{0}^{\infty}e^{-\tau\delta-\frac{\arrowvert x-y\arrowvert^{2}}{4\pi\delta}}\delta^{-\frac{-n+2}{2}}\frac{d\delta}{\delta}.
	\end{flalign}
	And a direct computation yields
	\begin{flalign}
		&\frac{\sin\frac{s_{1}}{2}\pi}{\pi}\int_{0}^{\infty}\frac{\tau^{\frac{s_{1}}{2}}(\tau-\Delta)^{-1}}{\tau^{s_{1}}-2\tau^{\frac{s_{1}}{2}}z_{0}^{\frac{s_{1}}{2}}\cos\frac{s_{1}}{2}\pi+z_{0}^{s_{1}}}d\tau\nonumber\\
		=&\int_{0}^{\infty}e^{-\tau\delta-\frac{\arrowvert x-y\arrowvert^{2}}{4\pi\delta}}\delta^{-\frac{-n+2}{2}}\frac{d\delta}{\delta}\frac{\sin\frac{s_{1}}{2}\pi}{\pi}\int_{0}^{\infty}\frac{\tau^{\frac{s_{1}}{2}}e^{-\tau\delta}}{\tau^{s_{1}}-2\tau^{\frac{s_{1}}{2}}z_{0}^{\frac{s_{1}}{2}}\cos\frac{s_{1}}{2}\pi+z_{0}^{s_{1}}}d\tau\nonumber\\
		=&\int_{0}^{\infty}e^{-\tau\delta-\frac{\arrowvert x-y\arrowvert^{2}}{4\pi\delta}}\delta^{-\frac{-n+2}{2}}f_{s_{1}}(\delta)\frac{d\delta}{\delta}\nonumber\\
		=&C\arrowvert x-y\arrowvert^{-(n-s_{1})}+o(\arrowvert x-y\arrowvert^{-(n-s_{1})}),\label{2.76}
	\end{flalign}
	where $f_{s_{1}}(\delta)=\int_{0}^{\infty}\frac{\tau^{\frac{s_{1}}{2}}e^{-\tau}}{\tau^{s_{1}}-2\tau^{\frac{s_{1}}{2}}\delta^{\frac{s_{1}}{2}}\cos\frac{s_{1}}{2}\pi+(\delta z_{0})^{s_{1}}}d\tau$, and $\mathop{\lim}\limits_{\delta\rightarrow0}f_{s_{1}}(\delta)=\Gamma(1-\frac{s_{1}}{2})$. The last equality follows
	\begin{flalign*}
		\arrowvert x-y\arrowvert^{-(n-s_{1})}=\frac{(4\pi)^{\frac{-(n-s_{1})}{2}}}{\Gamma(\frac{n-s_{1}}{2})}\int_{0}^{\infty}e^{-\tau\delta-\frac{\arrowvert x-y\arrowvert^{2}}{4\pi\delta}}\delta^{-\frac{-n+2}{2}}\frac{d\delta}{\delta}.
	\end{flalign*}
	Thus we have
	\begin{flalign}\label{2.77}
		\arrowvert K_{2}(x,y)\arrowvert\leqslant C\arrowvert x-y\arrowvert^{-(n-s_{1})},\  \arrowvert x-y\arrowvert\leqslant1.
	\end{flalign}
	
	When $\arrowvert x-y\arrowvert>1$, recall that
	\begin{equation}
		\arrowvert(\tau-\Delta)^{-1}(x,y)\arrowvert\leqslant
		\left\{
		\begin{split}
			&C\arrowvert x-y\arrowvert^{-(n-2)},\quad  \sqrt{\tau}\arrowvert x-y\arrowvert\leqslant1;\nonumber\\
			&C\tau^{\frac{n-2}{4}}\arrowvert x-y\arrowvert^{-\frac{n-2}{2}}e^{-\sqrt{\tau}\arrowvert x-y\arrowvert},\ \sqrt{\tau}\arrowvert x-y\arrowvert>1,
		\end{split}
		\right.
	\end{equation}
	we obtain that
	\begin{flalign}
		&\int_{0}^{\infty}\arrowvert\frac{\tau^{\frac{s_{1}}{2}}(\tau-\Delta)^{-1}(x,y)}{\tau^{\frac{s_{1}}{2}}-z_{0}^{\frac{s_{1}}{2}}e^{\pm i\pi\frac{s_{1}}{2}}}\arrowvert d\tau\nonumber\\
		\leqslant&C\arrowvert x-y\arrowvert^{-(n-2)}\int_{0}^{\arrowvert x-y\arrowvert^{-2}}\arrowvert\frac{\tau^{\frac{s_{1}}{2}}}{\tau^{\frac{s_{1}}{2}}-z_{0}^{\frac{s_{1}}{2}}e^{\pm i\pi\frac{s_{1}}{2}}}\arrowvert d\tau\nonumber\\
		&+C\arrowvert x-y\arrowvert^{-n}\int_{1}^{\infty}\frac{t^{s_{1}+\frac{n}{2}}e^{-t}}{\arrowvert t^{s_{1}}-\arrowvert x-y\arrowvert^{s_{1}}z_{0}^{\frac{s_{1}}{2}}e^{\pm i\pi\frac{s_{1}}{2}}\arrowvert}dt\nonumber\\
		\leqslant&C\arrowvert x-y\arrowvert^{-(n+s_{1})},\label{2.78}
	\end{flalign}
	where the last inequality follows from the fact that when $\arrowvert \arg z_{0}^{\frac{s_{1}}{2}}\arrowvert\leqslant\frac{s_{1}\pi}{4}$, then
	\begin{flalign*}
		\arrowvert\tau^{\frac{s_{1}}{2}}-z_{0}^{\frac{s_{1}}{2}}e^{\pm i\pi\frac{s_{1}}{2}}\arrowvert\geqslant C_{s_{1}}, \ if\ 0<\tau<1,
	\end{flalign*}
	and
	\begin{flalign*}
		\arrowvert t^{s_{1}}-\arrowvert x-y\arrowvert^{s_{1}}z_{0}^{\frac{s_{1}}{2}}e^{\pm i\pi\frac{s_{1}}{2}}\arrowvert\geqslant C_{s_{1}}\arrowvert x-y\arrowvert^{s_{1}}, \ if\ \arrowvert x-y\arrowvert>1,\ t>0.
	\end{flalign*}
	Hence, we have
	\begin{flalign}\label{2.79}
		\arrowvert K_{2}(x,y)\arrowvert\leqslant C\arrowvert x-y\arrowvert^{-(n+s_{1})},\  \arrowvert x-y\arrowvert>1.
	\end{flalign}
	
	Combining estimates of $K_{1}$ and $K_{2}$, we obtain that
	\begin{equation}
		\arrowvert K(x,y)\arrowvert\leqslant
		\left\{
		\begin{split}
			&C\arrowvert x-y\arrowvert^{-(n-2)},\, \, \, \  \arrowvert x-y\arrowvert\leqslant1;\\
			&C\arrowvert x-y\arrowvert^{-(n+s_{1})},\, \, \  \arrowvert x-y\arrowvert>1.\label{2.80}
		\end{split}
		\right.
	\end{equation}
	Then, using the same method to estimate $V_{x}$ in Lemma \ref{lemma4.2}, we obtain that
	\begin{flalign}
		&\|R_{0}^{s_{1}}(z)V_{x}u\|_{L^{\infty}}\nonumber\\
		=&\|\int_{\mathbb{R}^{n}}K(x,y)V_{y}u(y)dy\|_{L^{\infty}}\nonumber\\
		\leqslant&C\|\langle x\rangle^{-\sigma}*V_{x}u\|_{L^{\infty}}\nonumber\\
		\leqslant&C\|V_{x}u\|_{L^{\infty}}\nonumber\\
		\leqslant&C\int_{0}^{\infty}\tau^{\frac{s_{1}}{2}}\|(\tau-\Delta)^{-1}V_{3}(x)(\tau-\Delta_{A})^{-1}u\|_{L^{\infty}}d\tau\nonumber\\
		\leqslant&C\|u\|_{L^{\infty}},\label{2.81}
	\end{flalign}
	where $0<s_{1}<1$. Therefore, the proof of Lemma \ref{lemma2.14} is completed.
\end{proof}

Now, applying the properties of the distorted Fourier transforms at the spectral points and the estimate of Lemma \ref{lemma2.14}, we obtain the estimate of solutions of Schr\"{o}dinger magnetic equation.

\begin{lemma}\label{lemma2.15}
	For any $s>\frac{n}{2}$, and $u\in C_{0}^{\infty}(\mathbb{R}^{n})$, there exists a fixed $C>0$ such that
	\begin{equation}\label{2.82}
		\|u\|_{L^{\infty}_{x}(\mathbb{R}^{n})}\leqslant C(\frac{n}{2s-n})^{\frac{n}{4s}}\sqrt{\frac{\pi^{\frac{n}{2}}}{\Gamma(\frac{n}{2}+1)}}\|u\|_{L^{2}_{x}(\mathbb{R}^{n})}^{1-\frac{n}{2s}}\cdot\|(-\Delta_{A})^{\frac{s}{2}}u\|_{L^{2}_{x}(\mathbb{R}^{n})}^{\frac{n}{2s}}.
	\end{equation}
\end{lemma}
\begin{proof}
	From definition of the distorted Fourier transform in Remark \ref{remark2.2}, we have
	\begin{flalign}	
		&(F_\pm^{A}u(\xi),u(\xi))\nonumber\\
		=&(\mathscr{F}((I+V_{x}R_0^{s_{1}}(\lambda\pm i0 ))^{-1}u)(\xi),u(\xi))\nonumber\\
		=&(\int_{\mathbb{R}^{n}}e^{-ix\cdot\xi}(I+V_{x}R_0^{s_{1}}(\lambda\pm i0 ))^{-1}u(x)dx,u(\xi))\nonumber\\
		=&((I+V_{x}R_0^{s_{1}}(\lambda\pm i0 ))^{-1}u(x),\int_{\mathbb{R}^{n}}e^{ix\cdot\xi}u(\xi)d\xi)\nonumber\\
		=&(u(x),\int_{\mathbb{R}^{n}}(I+R_0^{s_{1}}(\lambda\pm i0 )V_{x})^{-1}e^{ix\cdot\xi}u(\xi)d\xi),\label{2.83}
	\end{flalign}
	where $(\cdot,\cdot)$ represents inner product. Therefore, we obtain the adjoint operator of $F_\pm^{A}$	
	$$(F_\pm^{A})^{*}=(I+R_0^{s_{1}}(\lambda\pm i0 )V_{x})^{-1}\mathscr{F}^{-1}.$$
	Combining $\mathscr{F}^{-1}: L^{1}\rightarrow L^{\infty}$ and Lemma \ref{lemma2.14} implies that $$(I+R_0^{s_{1}}(\lambda\pm i0 )V_{x})^{-1}: L^{\infty}\rightarrow L^{\infty},\ 0<s_{1}<1,$$  we have
	$$(F_\pm^{A})^{*}: L^{1}\rightarrow L^{\infty}.$$	
	
	Furthermore, from the Theorem 14.6.4 in \cite{H}
	\begin{flalign}\label{2.84}
		\|E_{c}u\|_{L^{2}(\mathbb{R}^{n})}=\|F_\pm^{A}u\|_{L^{2}(\mathbb{R}^{n})},
	\end{flalign}
	we have the left side of \eqref{2.84} is:
	\begin{flalign*}	
		(E_{c}u,E_{c}u)=(u,E_{c}u),
	\end{flalign*}
	and the right side of \eqref{2.84} is:
	\begin{flalign*}	
		(F_\pm^{A}u,F_\pm^{A}u)=(u,(F_\pm^{A})^{*}F_\pm^{A}u).
	\end{flalign*}
	By combining the left and right equalities, we obtain that
	\begin{flalign*}	
		E_{c}=(F_\pm^{A})^{*}F_\pm^{A},
	\end{flalign*}	
	then
	\begin{flalign}\label{2.85}
		E_{c}(F_\pm^{A})^{-1}=(F_\pm^{A})^{*}.
	\end{flalign}
	From the properties of the Fourier transform and H\"older's inequality, it follows that for any $l>0$	
	\begin{flalign}	
		&\|E_{c}u\|_{L_{x}^{\infty}(\mathbb{R}^{n})}\nonumber\\	
		=&\|E_{c}(F_\pm^{A})^{-1}(F_\pm^{A}u)\|_{L_{x}^{\infty}(\mathbb{R}^{n})}\nonumber\\	
		=&\|(F_\pm^{A})^{*}(F_\pm^{A}u)\|_{L_{x}^{\infty}(\mathbb{R}^{n})}\nonumber\\
		\leqslant&C\|F_\pm^{A}u\|_{L_{x}^{1}(\mathbb{R}^{n})}\nonumber\\
		\leqslant&C \|F_\pm^{A}u\|_{L^{1}(\arrowvert\xi\arrowvert\leqslant l)}+C\|F_\pm^{A}u\|_{L^{1}(\arrowvert\xi\arrowvert\geqslant l)}\nonumber\\
		\leqslant&C \|F_\pm^{A}u\|_{L^{2}(\arrowvert\xi\arrowvert\leqslant l)}(\int_{\arrowvert\xi\arrowvert\leqslant l}1d\xi)^{\frac{1}{2}}+C\int_{\arrowvert\xi\arrowvert\geqslant l}(F_\pm^{A}u)\cdot\arrowvert\xi\arrowvert^{s}\cdot\arrowvert\xi\arrowvert^{-s}d\xi\nonumber\\
		\leqslant&C \|F_\pm^{A}u\|_{L^{2}(\arrowvert\xi\arrowvert\leqslant l)}\cdot l^{\frac{n}{2}}\sqrt{\frac{\pi^{\frac{n}{2}}}{\Gamma(\frac{n}{2}+1)}}+C\|(F_\pm^{A}u)\cdot\arrowvert\xi\arrowvert^{s}\|_{L^{2}(\arrowvert\xi\arrowvert\geqslant l)}\|\arrowvert\xi\arrowvert^{-s}\|_{L^{2}(\arrowvert\xi\arrowvert\geqslant l)}\nonumber\\
		\leqslant&C \|F_\pm^{A}u\|_{L^{2}(\mathbb{R}^{n})}\cdot l^{\frac{n}{2}}\sqrt{\frac{\pi^{\frac{n}{2}}}{\Gamma(\frac{n}{2}+1)}}+Cl^{\frac{n}{2}-s}\sqrt{\frac{2\pi^{\frac{n}{2}}}{\Gamma(\frac{n}{2})(2s-n)}}\|\arrowvert\xi\arrowvert^{s}F_\pm^{A}u\|_{L^{2}(\mathbb{R}^{n})}\nonumber\\
		=&\|E_{c}u\|_{L^{2}(\mathbb{R}^{n})}\cdot l^{\frac{n}{2}}\sqrt{\frac{\pi^{\frac{n}{2}}}{\Gamma(\frac{n}{2}+1)}}+Cl^{\frac{n}{2}-s}\sqrt{\frac{2\pi^{\frac{n}{2}}}{\Gamma(\frac{n}{2})(2s-n)}}\|\arrowvert\xi\arrowvert^{s}F_\pm^{A}u\|_{L^{2}(\mathbb{R}^{n})}\nonumber\\
		\leqslant&C\|u\|_{L^{2}(\mathbb{R}^{n})}\cdot l^{\frac{n}{2}}\sqrt{\frac{\pi^{\frac{n}{2}}}{\Gamma(\frac{n}{2}+1)}}+Cl^{\frac{n}{2}-s}\sqrt{\frac{2\pi^{\frac{n}{2}}}{\Gamma(\frac{n}{2})(2s-n)}}\|\arrowvert\xi\arrowvert^{s}F_\pm^{A}u\|_{L^{2}(\mathbb{R}^{n})}.\label{2.86}
	\end{flalign}	
	Let $s=ks_{1}$, where $0<s_{1}<1$ and $k=1,2,...$, we have
	\begin{flalign}	
		&\|\arrowvert\xi\arrowvert^{s}F_\pm^{A}u\|_{L^{2}(\mathbb{R}^{n})}\nonumber\\
		=&\|\arrowvert\xi\arrowvert^{s-ks_{1}}\cdot\arrowvert\xi\arrowvert^{ks_{1}}F_\pm^{A}u\|_{L^{2}(\mathbb{R}^{n})}\nonumber\\
		=&\|\arrowvert\xi\arrowvert^{s-ks_{1}}F_\pm^{A}(-\Delta_{A})^{\frac{ks_{1}}{2}}u\|_{L^{2}(\mathbb{R}^{n})}\nonumber\\
		=&\|F_\pm^{A}(-\Delta_{A})^{\frac{s}{2}}u\|_{L^{2}(\mathbb{R}^{n})}.\label{2.87}
	\end{flalign}

Since $(-\Delta_{A})^{\frac{s_{1}}{2}}$ has no embedded eigenvalues (see \cite{DW}), and substituting \eqref{2.87} into \eqref{2.86}, we obtain that
	\begin{flalign}
		\|E_{c}u\|_{L_{x}^{\infty}(\mathbb{R}^{n})}=&\|u\|_{L_{x}^{\infty}(\mathbb{R}^{n})}\nonumber\\
		\leqslant&Cl^{\frac{n}{2}}\sqrt{\frac{\pi^{\frac{n}{2}}}{\Gamma(\frac{n}{2}+1)}}\cdot\|u\|_{L^{2}(\mathbb{R}^{n})}\nonumber\\
		+&Cl^{\frac{n}{2}-s}\sqrt{\frac{2\pi^{\frac{n}{2}}}{\Gamma(\frac{n}{2})(2s-n)}}\cdot\|(-\Delta_{A})^{\frac{s}{2}}u\|_{L^{2}(\mathbb{R}^{n})}.\label{2.88}
	\end{flalign}
	Let $l=(\sqrt{\frac{n}{2s-n}}\|(-\Delta_{A})^{\frac{s}{2}}u\|_{L^{2}})^{\frac{1}{s}}\cdot\|u\|^{-\frac{1}{s}}_{L_{x}^{2}}$, then the last two terms of \eqref{2.88} are equal, that is,
	\begin{equation*}
		\|u\|_{L^{\infty}_{x}(\mathbb{R}^{n})}\leqslant C(\frac{n}{2s-n})^{\frac{n}{4s}}\sqrt{\frac{\pi^{\frac{n}{2}}}{\Gamma(\frac{n}{2}+1)}}\|u\|_{L^{2}_{x}(\mathbb{R}^{n})}^{1-\frac{n}{2s}}\cdot\|(-\Delta_{A})^{\frac{s}{2}}u\|_{L^{2}(\mathbb{R}^{n})}^{\frac{n}{2s}},
	\end{equation*}
	therefore, \eqref{2.82} follows. This completes the proof of Lemma \ref{lemma2.15}.
\end{proof}

\subsection{The proof of Theorem \ref{theorem1.1}}
In this subsection, we prove Theorem \ref{theorem1.1}. First, from \eqref{2.82}, we need to estimate the operator $\arrowvert J_{A}\arrowvert^{s}$. For the linear magnetic Schr\"{o}dinger equation
$$(i\partial_{t}+\Delta_{A})u=0,$$
by exchanging operators and equation, we have
\begin{flalign}
	&(i\partial_{t}+\Delta_{A})\arrowvert J_{A}\arrowvert^{s}u\nonumber\\
	=&[i\partial_{t}+\Delta_{A},\arrowvert J_{A}\arrowvert^{s}]u+\arrowvert J_{A}\arrowvert^{s}(i\partial_{t}+\Delta_{A})u\nonumber\\
	=&[i\partial_{t}+\Delta_{A},\arrowvert J_{A}\arrowvert^{s}]u.\label{2.89}
\end{flalign}

Next, we prove Theorem \ref{theorem1.1}, and the proof process is divided into four steps.

\hspace{-5mm}\textbf{Proof of Theorem \ref{theorem1.1}:}

$Step\ 1$\quad Let's review the result of the commutator in \eqref{2.8}
$$[i\partial_{t}+\Delta_{A},\arrowvert J_{A}(t)\arrowvert^{s}]=it^{s-1}M(t)V(s)M(-t),$$
where
\begin{equation*}
	V(s)=s(-\Delta_{A})^{\frac{s}{2}}+[x\cdot \nabla,(-\Delta_{A})^{\frac{s}{2}}]-i[(-\Delta_{A})^{\frac{s}{2}},x\cdot A].
\end{equation*}
Then, we have the following equation:
\begin{flalign}
	(i\partial_{t}+\Delta_{A})\arrowvert J_{A}\arrowvert^{s}u=[i\partial_{t}+\Delta_{A},\arrowvert J_{A}\arrowvert^{s}]u=it^{s-1}M(t)V(s)M(-t)u.\label{2.90}
\end{flalign}

Applying the Strichartz estimate to $\arrowvert J_{A}\arrowvert^{s}u$ of nonlinear equation \eqref{2.90}, we obtain that
\begin{flalign}
	&\|\arrowvert J_{A}\arrowvert^{s}u\|_{L_{t}^{\infty}((1,T);L_{x}^{2}(\mathbb{R}^{n}))}\nonumber\\
	\leqslant&C\|\arrowvert J_{A}\arrowvert^{s}u(1)\|_{L_{x}^{2}(\mathbb{R}^{n})}\nonumber\\
	+&C_{s}\|t^{s-1}M(t)V(s)M(-t)u\|_{L_{t}^{q}((1,T);L_{x}^{r}(\mathbb{R}^{n}))},\label{2.91}
\end{flalign}
where  $\frac{2}{q'}=n(\frac{1}{2}-\frac{1}{r'})$, $\frac{1}{q}+\frac{1}{q'}=1$, $\frac{1}{r}+\frac{1}{r'}=1$. In particular, $q=\frac{4}{4-n}$, $r=1$ when $n\leqslant3$, and $r=\frac{2nq}{q(n+4)-4}\in(1,2)$ when $n>3$.

$Step\ 2$\quad  From the Strichartz estimate of $\arrowvert J_{A}\arrowvert^{s}u$, it is necessary to estimate the second term of \eqref{2.91}. Before that, we prove the relationship of $\arrowvert J\arrowvert^{s}$ and $\arrowvert J_{A}\arrowvert^{s}$ for $s>\frac{n}{2}$ by using the equivalence property in Lemma \ref{lemma2.13}.

\begin{lemma}\label{lemma2.16}
	Let $\arrowvert J\arrowvert^{s}$ and $\arrowvert J_{A}\arrowvert^{s}$ as definition of \eqref{2.6} and \eqref{2.7} respectively. Then for $\forall\ \theta\in(0,1)$ and $s>\frac{n}{2}$, we have
	\begin{equation}\label{2.92}
		\|\arrowvert J_{A}\arrowvert^{s}u\|_{L^{2}_{x}(\mathbb{R}^{n})}\leqslant Ct^{s+\theta-\frac{n}{2}}(\|\arrowvert J\arrowvert^{\frac{n}{2}-\theta}u\|_{L^{2}_{x}(\mathbb{R}^{n})}+\|\arrowvert J\arrowvert^{s}u\|_{L^{2}_{x}(\mathbb{R}^{n})}),
	\end{equation}
	\begin{equation}\label{2.93}
		\|\arrowvert J\arrowvert^{s}u\|_{L^{2}_{x}(\mathbb{R}^{n})}\leqslant Ct^{s+\theta-\frac{n}{2}}(\|\arrowvert J_{A}\arrowvert^{\frac{n}{2}-\theta}u\|_{L^{2}_{x}(\mathbb{R}^{n})}+\|\arrowvert J_{A}\arrowvert^{s}u\|_{L^{2}_{x}(\mathbb{R}^{n})}).
	\end{equation}
\end{lemma}

\begin{proof}
	According to Lemma \ref{lemma2.8}, we can exchange $-\Delta_{A}$ and $-\Delta$, $\arrowvert J\arrowvert^{s}$ and $\arrowvert J_{A}\arrowvert^{s}$, therefore we only just prove \eqref{2.92}.
	
	Note that
	\begin{equation}\label{2.94}
		\|\sqrt{-\Delta_{A}}\,u\|^{2}_{L^{2}_{x}}\leqslant C_{1}\|\sqrt{-\Delta}\,u\|^{2}_{L^{2}_{x}}+C_{2}\|(\nabla\cdot A+A\cdot\nabla+\arrowvert A\arrowvert^{2})\arrowvert u\arrowvert^{2}\|_{L_{x}^{1}}.
	\end{equation}
	
	For the second term on the right-hand side of \eqref{2.94}, by H\"older's inequality and Sobolev embedding, we obtain that
	\begin{flalign}
		&\|(\nabla\cdot A+A\cdot\nabla+\arrowvert A\arrowvert^{2})\arrowvert u\arrowvert^{2}\|_{L_{x}^{1}}\nonumber\\
		\leqslant&\|(\nabla\cdot A+\arrowvert A\arrowvert^{2})\arrowvert u\arrowvert^{2}\|_{L^{1}_{x}}+\|A\cdot\nabla \arrowvert u\arrowvert^{2}\|_{L^{1}_{x}}\nonumber\\
		\leqslant&\|\nabla\cdot A+\arrowvert A\arrowvert^{2}\|_{L_{x}^{m'_{1}}}\cdot\|u\|_{L_{x}^{2m_{1}}}^{2}+\|A\|_{L_{x}^{m'_{2}}}\cdot\|(-\Delta)^{\frac{1}{4}}u\|^{2}_{L^{2m_{2}}}\nonumber\\
		\leqslant& C\|(-\Delta)^{\frac{n}{2}(\frac{1}{2}-\delta)}u\|^{2}_{L_{x}^{2}},\label{2.95}
	\end{flalign}	
	where $$\frac{1}{2m_1}=\frac{n-2n(\frac{1}{2}-\delta)}{2n}=\frac{1}{2}-(\frac{1}{2}-\delta)=\delta,\ 0<\delta<1, $$
	and $$\frac{1}{2m_2}=\frac{n-2(\frac{n}{2}-n\delta'-1)}{2n}=\frac{1}{2}-(\frac{1}{2}-\delta'-\frac{1}{n})=\delta'+\frac{1}{n}=\delta.$$
	
	Then, combining \eqref{2.94} and \eqref{2.95}, we have
	\begin{flalign}
		\|\sqrt{-\Delta_{A}}\,u\|^{2}_{L^{2}_{x}}\leqslant &C_{1}\|\sqrt{-\Delta}\,u\|^{2}_{L^{2}_{x}}+C_{2}\|(-\Delta)^{\frac{n}{2}(\frac{1}{2}-\delta)}u\|^{2}_{L_{x}^{2}}\nonumber\\
		\leqslant& C\|(-\Delta)^{\frac{n}{4}-\frac{n}{2}\delta}(1+(-\Delta)^{-\frac{n}{4}+\frac{n}{2}\delta+\frac{1}{2}})u\|^{2}_{L_{x}^{2}}.\label{2.96}
	\end{flalign}
	
	Furthermore, for every $s_{1}'>\frac{n}{2}$, from \eqref{4.19} in Lemma \ref{lemma2.9}, we obtain that
	\begin{flalign}
		&\|(-\Delta_{A})^{\frac{s'_{1}}{2}}u\|^{2}_{L_{x}^{2}}\nonumber\\
		\leqslant&\|(-\Delta_{A})^{\frac{s'_{1}}{2}}u-(-\Delta)^{\frac{s'_{1}}{2}}u\|^{2}_{L_{x}^{2}}+\|(-\Delta)^{\frac{s'_{1}}{2}}u\|^{2}_{L_{x}^{2}}\nonumber\\
		\leqslant&C\|u\|^{2}_{\dot{H}_{x}^{2k+\alpha}}+\|(-\Delta)^{\frac{s'_{1}}{2}}u\|^{2}_{L_{x}^{2}}\nonumber\\
		\leqslant&C\|(-\Delta)^{\frac{n}{4}-\frac{n}{2}\delta}(1+(-\Delta)^{-\frac{n}{4}+\frac{n}{2}\delta+\frac{2k+\alpha}{2}}+(-\Delta)^{-\frac{n}{4}+\frac{n}{2}\delta+\frac{s'_{1}}{2}})u\|^{2}_{L_{x}^{2}},\label{2.97}
	\end{flalign}
	where $0<\alpha<1$ and $k=0,1,2,...$.
	
	Applying the interpolation method to \eqref{2.96} and \eqref{2.97}, for every $\frac{1}{2}<s<s'_{1}$, we have
	\begin{flalign}
		\|(-\Delta_{A})^{\frac{s}{2}}u\|_{L_{x}^{2}}\leqslant& C\|(-\Delta)^{\frac{n}{4}-\frac{n}{2}\delta}(1+(-\Delta)^{\frac{s}{2}-\frac{n}{4}+\frac{n}{2}\delta})u\|_{L_{x}^{2}}\nonumber\\
		\leqslant&C(\|(-\Delta)^{\frac{n}{4}-\frac{n}{2}\delta}u\|_{L_{x}^{2}}+\|(-\Delta)^{\frac{s}{2}}u\|_{L_{x}^{2}}).\label{2.98}
	\end{flalign}	
	Multiplying both sides of \eqref{2.98} by $t^{s}$ gives the estimate
	$$\|\arrowvert J_{A}\arrowvert^{s}u\|_{L^{2}_{x}}\leqslant Ct^{s+n\delta-\frac{n}{2}}(\|\arrowvert J\arrowvert^{\frac{n}{2}-n\delta}u\|_{L^{2}_{x}}+\|\arrowvert J\arrowvert^{s}u\|_{L^{2}_{x}}),$$
	and letting $\theta=n\delta$, \eqref{2.92} follows. The proof of Lemma \ref{lemma2.16} is completed.
\end{proof}

Now, applying the estimates in Lemma \ref{lemma2.6}, Lemma \ref{lemma2.7}, and the relationship above to estimate $t^{s-1}M(t)V(s)M(-t)u$. The following lemmas will give us the desired results.
\begin{lemma}\label{lemma2.17}
	When $n\leqslant3$, let $A(x)$ satisfy hypothesis \eqref{1.2}, $M(-t)$ and $\arrowvert J_{A}\arrowvert^{s}$ as in \eqref{2.1} and \eqref{2.7} respectively. Then for $\forall\ \theta\in(0,1)$, $t\geqslant1$ and $0<s<2$, we have the following estimate:
	\begin{flalign}
		&\|t^{s-1}V(s)M(-t)u\|_{L^{\frac{4}{4-n}}_{t}((1,\infty);L^{1}_{x}(\mathbb{R}^{n}))}\nonumber\\
		\leqslant&C_{s}(\|\arrowvert J_{A}\arrowvert^{\frac{n}{2}-\theta}u\|_{L^{\infty}_{t}L^{2}_{x}}+\|\arrowvert J_{A}\arrowvert u\|_{L^{\infty}_{t}L^{2}_{x}}).\label{2.99}
	\end{flalign}
\end{lemma}

\begin{proof}	
	According to \eqref{2.26} in Lemma \ref{lemma2.6}, it follows that
	\begin{flalign}
		&\|t^{s-1}V(s)M(-t)u\|_{L^{\frac{4}{4-n}}_{t}((1,\infty);L^{1}_{x}(\mathbb{R}^{n}))}\nonumber\\
		\leqslant&C_{s}\|t^{s-1}\|V(s)M(-t)u\|_{L^{1}_{x}}\|_{L^{\frac{4}{4-n}}_{t}}\nonumber\\
		\leqslant&C_{s}\|t^{s-1}\|M(-t)u\|_{\dot{H}^{\alpha}_{x}}\|_{L^{\frac{4}{4-n}}_{t}},\ 0<\alpha<\min\{1,s\}.\label{2.100}
	\end{flalign}
	
	Furthermore,
	\begin{flalign}
		&\|M(-t)u\|_{\dot{H}^{\alpha}_{x}(\mathbb{R}^{n})}\nonumber\\
		=&\|(-\Delta)^{\frac{\alpha}{2}}M(-t)u\|_{L^{2}_{x}}\nonumber\\
		=&\|M(t)(-t^{2}\Delta)^{\frac{\alpha}{2}}M(-t)u\|_{L^{2}_{x}}\cdot t^{-\alpha}\nonumber\\
		=&t^{-\alpha}\|\arrowvert J\arrowvert u\|_{L^{2}_{x}}.\label{2.101}
	\end{flalign}	
	Applying \eqref{2.93} in Lemma \ref{lemma2.16} to \eqref{2.101}, we obtain that
	\begin{equation}\label{2.102}
		\|M(-t)u\|_{\dot{H}^{\alpha}_{x}(\mathbb{R}^{n})}\leqslant Ct^{\theta-\frac{n}{2}}(\|\arrowvert J_{A}\arrowvert^{\frac{n}{2}-\theta}u\|_{L^{2}_{x}}+\|\arrowvert J_{A}\arrowvert u\|_{L^{2}_{x}}).
	\end{equation}
	
	Therefore, substituting \eqref{2.102} into \eqref{2.100} and letting $s<\frac{3n}{4}-\theta$ such that $$t^{s-1+\theta-\frac{n}{2}}\in L^{\frac{4}{4-n}}_{t}(1,\infty),$$\eqref{2.99} is obtained. The proof of Lemma \ref{lemma2.17} is completed.
\end{proof}
\begin{lemma}\label{lemma2.18}
	When $n>3$, let $A(x)$ satisfy hypothesis \eqref{1.2}, $M(-t)$ and $\arrowvert J_{A}\arrowvert^{s}$ as in \eqref{2.1} and \eqref{2.7} respectively. Then for $t\geqslant1$, $2k+1<s<2k+2$, $k=1,2,...$, and $0<\alpha<1$, we have the following estimate:
	\begin{equation}\label{2.103}
		\|t^{s-1}V(s)M(-t)u\|_{L^{q}_{t}((1,\infty);L^{r}_{x}(\mathbb{R}^{n}))}\leqslant C_{s,k}\|\arrowvert J_{A}\arrowvert^{2k+\alpha}u\|_{L^{\infty}_{t}L^{2}_{x}},
	\end{equation}
	where $\frac{1}{q}+\frac{1}{q'}=1$, $\frac{1}{r}+\frac{1}{r'}=1$ and  $(q',r')$ is admissible pair, i.e. $\frac{2}{q'}=n(\frac{1}{2}-\frac{1}{r'})$.
\end{lemma}
\begin{proof}	
	According to \eqref{2.54} in Lemma \ref{lemma2.7}, it follows that
	\begin{flalign}
		&\|t^{s-1}V(s)M(-t)u\|_{L^{q}_{t}((1,\infty);L^{r}_{x}(\mathbb{R}^{n}))}\nonumber\\
		\leqslant&C_{s,k}\|t^{s-1}\|V(s)M(-t)u\|_{L^{r}_{x}}\|_{L^{q}_{t}}\nonumber\\
		\leqslant&C_{s,k}\|t^{s-1}\|M(-t)u\|_{\dot{H}_{A}^{2k+\alpha}}\|_{L^{q}_{t}}\nonumber\\
		=&C_{s,k}\|t^{s-1}\cdot t^{-(2k+\alpha)}\|\arrowvert J_{A}\arrowvert^{2k+\alpha}u\|_{L^{2}_{x}}\|_{L^{q}_{t}}\nonumber\\
		\leqslant&C_{s,k}\|\arrowvert J_{A}\arrowvert^{2k+\alpha}u\|_{L^{\infty}_{t}L^{2}_{x}},\label{2.104}
	\end{flalign}
	since $r=\frac{2nq}{q(n+4)-4}$, letting $1<r<2$ such that $$t^{s-1-2k-\alpha}\in L^{q}_{t}(1,\infty),$$ \eqref{2.103} is obtained. The proof of Lemma \ref{lemma2.18} is completed.
\end{proof}

\begin{remark}\label{remark2.3}
	Since $r=\frac{2nq}{q(n+4)-4}\in(1,2)$, the value of $r$ will change with the value of $n$. According to the following examples, this result will be more clearly reflected:
	
	(1). choose $r=\frac{5}{4}$, we have $q=\frac{20}{20-3n}$, this implies $n<7$;
	
	(2). choose $r=\frac{3}{2}$, we have $q=\frac{12}{12-n}$, this implies $n<12$;
	
	(3). choose $r=\frac{15}{8}$, we have $q=\frac{60}{60-n}$, this implies $n<60$.
	
	From these examples, we conclude that the larger the selection $r$ is, the larger the $n$ is. Furthermore, the larger $n$ is, the closer $r$ is to $2$.
\end{remark}

$Step\ 3$\quad Applying the estimates in Lemma \ref{lemma2.17} and Lemma \ref{lemma2.18}, we can obtain the Strichartz estimate of $\arrowvert J_{A}\arrowvert^{s}u$.

When $n\leqslant3$, substituting \eqref{2.99} to \eqref{2.91}, for any $s>\frac{n}{2}$, it follows that
\begin{flalign}
	&\|\arrowvert J_{A}\arrowvert^{s}u\|_{L_{t}^{\infty}((1,T);L_{x}^{2}(\mathbb{R}^{n}))}\nonumber\\
	\leqslant&C\|\arrowvert J_{A}\arrowvert^{s}u(1)\|_{L_{x}^{2}}+C_{s}(\|\arrowvert J_{A}\arrowvert^{\frac{n}{2}-\theta}u\|_{L^{\infty}_{t}L^{2}_{x}}+\|\arrowvert J_{A}\arrowvert u\|_{L^{\infty}_{t}L^{2}_{x}})\nonumber\\
	\leqslant&C\|\arrowvert J_{A}\arrowvert^{s}u(1)\|_{L_{x}^{2}}+C_{s}(2C_{s}\|u_{0}\|_{L^{2}}+\frac{1}{2C_{s}}\|\arrowvert J_{A}\arrowvert^{s}u\|_{L_{t}^{\infty}L_{x}^{2}}),\label{2.105}
\end{flalign}
and when $n>3$, substituting \eqref{2.103} to \eqref{2.91}, for any $s>\frac{n}{2}$, it follows that
\begin{flalign}
	&\|\arrowvert J_{A}\arrowvert^{s}u\|_{L_{t}^{\infty}((1,T);L_{x}^{2}(\mathbb{R}^{n}))}\nonumber\\
	\leqslant&C\|\arrowvert J_{A}\arrowvert^{s}u(1)\|_{L_{x}^{2}}+C_{s,k}\|\arrowvert J_{A}\arrowvert^{2k+\alpha}u\|_{L^{\infty}_{t}L^{2}_{x}}\nonumber\\
	\leqslant&C\|\arrowvert J_{A}\arrowvert^{s}u(1)\|_{L_{x}^{2}}+C_{s}(2C_{s}\|u_{0}\|_{L^{2}}+\frac{1}{2C_{s}}\|\arrowvert J_{A}\arrowvert^{s}u\|_{L_{t}^{\infty}L_{x}^{2}}).\label{2.106}
\end{flalign}

Hence, from \eqref{2.105} and \eqref{2.106}, we obtain that
\begin{equation}\label{2.107}
	\|\arrowvert J_{A}\arrowvert^{s}u\|_{L_{t}^{\infty}((1,T);L_{x}^{2}(\mathbb{R}^{n}))}\leqslant C\|\arrowvert J_{A}\arrowvert^{s}u(1)\|_{L_{x}^{2}}.
\end{equation}
We claim that
\begin{equation}\label{2.108}
	\|\arrowvert J_{A}\arrowvert^{s}u\|_{L_{t}^{\infty}((1,\infty);L_{x}^{2}(\mathbb{R}^{n}))}\leqslant C\|u_{0}\|_{\Sigma_{s}}.
\end{equation}
In fact,
\begin{flalign}
	&\|\arrowvert J_{A}\arrowvert^{s}u(1)\|_{L_{x}^{2}}\nonumber\\
	\leqslant&C(\|\arrowvert J\arrowvert^{\frac{n}{2}-\theta}u(1)\|_{L^{2}_{x}}+\|\arrowvert J\arrowvert^{s}u(1)\|_{L^{2}_{x}})\nonumber\\
	\leqslant&C(\|u_{0}\|_{L^{2}_{x}}+\|\arrowvert J\arrowvert^{s}u(1)\|_{L^{2}_{x}})\nonumber\\
	\leqslant&C\|u_{0}\|_{\Sigma_{s}}.\label{2.109}
\end{flalign}
From \eqref{2.107} and \eqref{2.109}, it follows that
\begin{flalign*}
	\|\arrowvert J_{A}\arrowvert^{s}u\|_{L_{t}^{\infty}((1,T);L_{x}^{2}(\mathbb{R}^{n}))}\leqslant C\|\arrowvert J_{A}\arrowvert^{s}u(1)\|_{L_{x}^{2}}\leqslant C\|u_{0}\|_{\Sigma_{s}}.
\end{flalign*}
This implies the claim in \eqref{2.108}.

$Step\ 4$\quad Combining the estimate of Lemma \ref{lemma2.15} and \eqref{2.108}, we obtain that
\begin{flalign*}
	\|u\|_{L^{\infty}_{x}(\mathbb{R}^{n})}\leqslant &C\|M(-t)u\|_{L^{2}_{x}(\mathbb{R}^{n})}^{1-\frac{n}{2s}}\cdot\|(-\Delta_{A})^{\frac{s}{2}}M(-t)u\|_{L^{2}(\mathbb{R}^{n})}^{\frac{n}{2s}}\nonumber\\
	=&Ct^{-\frac{n}{2}}\|u\|_{L^{2}_{x}(\mathbb{R}^{n})}^{1-\frac{n}{2s}}\cdot\|\arrowvert J_{A}\arrowvert^{s}u\|_{L^{2}_{x}(\mathbb{R}^{n})}^{\frac{n}{2s}}\nonumber\\
	\leqslant&Ct^{-\frac{n}{2}}\|u_{0}\|_{L^{2}_{x}(\mathbb{R}^{n})}^{1-\frac{n}{2s}}\cdot\|u_{0}\|_{\Sigma_{s}}^{\frac{n}{2s}}\nonumber\\
	=&Ct^{-\frac{n}{2}}\|u_{0}\|_{\Sigma_{s}}.
\end{flalign*}
This implies \eqref{1.10} is obtained. Therefore, the proof of Theorem \ref{theorem1.1} is completed.

\section{The dispersive decay of the nonlinear magnetic Schr\"odinger equation}

In this section, we main prove the dispersive bounds in Theorem \ref{theorem1.2} of solutions of the nonlinear magnetic Schr\"odinger equation.

First, from the nonlinear magnetic Schr\"{o}dinger equation
$$(i\partial_{t}+\Delta_{A})u=\arrowvert u\arrowvert^{p-1}u,$$
by exchanging operators and equation, we have
\begin{flalign}
	&(i\partial_{t}+\Delta_{A})\arrowvert J_{A}\arrowvert^{s}u\nonumber\\
	=&[i\partial_{t}+\Delta_{A},\arrowvert J_{A}\arrowvert^{s}]u+\arrowvert J_{A}\arrowvert^{s}(i\partial_{t}+\Delta_{A})u\nonumber\\
	=&it^{s-1}M(t)V(s)M(-t)u+\arrowvert J_{A}\arrowvert^{s}(\arrowvert u\arrowvert^{p-1}u).\label{3.1}
\end{flalign}
Since the estimate of $it^{s-1}M(t)V(s)M(-t)u$ has been obtained in Section 2, we now need to estimate the term $\arrowvert J_{A}\arrowvert^{s}(\arrowvert u\arrowvert^{p-1}u)$. Before that, we obtain the properties of derivatives with the term $\arrowvert u\arrowvert^{p-1}u$ as the following lemma.

\subsection{The properties of derivatives with the term $\arrowvert u\arrowvert^{p-1}u$}

\begin{lemma}\label{lemma3.1}
	Let $f(u)=\arrowvert u\arrowvert^{p-1}u$, for $p>1+\lceil\frac{n}{2}\rceil$. Then $f(u)$ satisfies the estimate
	\begin{flalign}
		\|f(u)\|_{\dot{B}^{s}_{2,2}}\leqslant C\|u\|_{\dot{B}^{s}_{2,2}}\|u\|_{L^{\infty}}^{p-1}.\label{3.2}
	\end{flalign}
	where $0<s<m$, $\mathbb{Z}^{+}\ni m>\frac{n}{2}$.
\end{lemma}
\begin{proof}
	We have the following equivalent norm for the Besov space $\dot{B}^{s}_{k,l}$ (see P147 in \cite{BL} )
	\begin{flalign}
		\|v\|_{\dot{B}^{s}_{k,l}}\sim(\int_{0}^{\infty}t^{-1-sl}\mathop{\sup}\limits_{\arrowvert y\arrowvert<t}\|\Delta^{m}_{y}v\|^{l}_{L^{k}}dt)^{\frac{1}{l}}, \ 0<s<m,\label{3.3}
	\end{flalign}
	where $\Delta^{m}_{y}v(x)=\sum\limits_{j=0}^{m}C_{m}^{j}(-1)^{j}v(x+jy)$.
	
	When $m=1$, we have
	\begin{flalign}
		\|v\|_{\dot{B}^{s}_{k,l}}\sim(\int_{0}^{\infty}t^{-1-sl}\mathop{\sup}\limits_{\arrowvert y\arrowvert<t}\|\tau_{y}v-v\|^{l}_{L^{k}}dt)^{\frac{1}{l}},\label{3.4}
	\end{flalign}
	which is valid for $0<s<1$, where $\tau_{y}v$ is the translation by $y\in\mathbb{R}^{n}$, when $m=2$, as well as the norm
	\begin{flalign}
		\|v\|_{\dot{B}^{s}_{k,l}}\sim(\int_{0}^{\infty}t^{-1-sl}\mathop{\sup}\limits_{\arrowvert y\arrowvert<t}\|\tau_{y}v+\tau_{-y}v-2v\|^{l}_{L^{k}}dt)^{\frac{1}{l}},\label{3.5}
	\end{flalign}
	which is valid for $0<s<2$.
	
	Choose $m=3$, we have
	\begin{flalign}
		\|v\|_{\dot{B}^{s}_{k,l}}\sim(\int_{0}^{\infty}t^{-1-sl}\mathop{\sup}\limits_{\arrowvert y\arrowvert<t}\|\tau_{-y}v+3\tau_{y}v-3v-\tau_{2y}v\|^{l}_{L^{k}}dt)^{\frac{1}{l}},\label{3.6}
	\end{flalign}
	which is valid for $0<s<3$. Let $\omega=\arrowvert u\arrowvert$, then we have $f(u)=\omega^{p-1}u$, now we write (pointwise in $x$)
	\begin{flalign}
		&f(u(x-y))-3f(u(x))+3f(u(x+y))-f(u(x+2y))\nonumber\\	
		=&\omega^{p-1}u(x-y)-3\omega^{p-1}u(x)+3\omega^{p-1}u(x+y)-\omega^{p-1}u(x+2y)\nonumber\\
		=&\omega^{p-1}(x-y)(u(x-y)-3u(x)+3u(x+y)-u(x+2y))\nonumber\\
		+&u(x)(3\omega^{p-1}(x-y)-3\omega^{p-1}(x))+u(x+y)(3\omega^{p-1}(x+y)-3\omega^{p-1}(x-y))\nonumber\\
		+&u(x+2y)(\omega^{p-1}(x-y)-\omega^{p-1}(x+2y))\nonumber\\
		=&\omega^{p-1}(x-y)(u(x-y)-3u(x)+3u(x+y)-u(x+2y))\nonumber\\
		+&3(\omega^{p-1}(x-y)-\omega^{p-1}(x))(u(x)+u(x+2y)-2u(x+y))\nonumber\\
		+&3u(x+y)(\omega^{p-1}(x+y)+\omega^{p-1}(x-y)-2\omega^{p-1}(x))\nonumber\\
		+&u(x+2y)(-2\omega^{p-1}(x-y)+3\omega^{p-1}(x)-\omega^{p-1}(x+2y))\nonumber\\
		=&\omega^{p-1}(x-y)(u(x-y)-3u(x)+3u(x+y)-u(x+2y))\nonumber\\
		+&3(\omega^{p-1}(x-y)-\omega^{p-1}(x))(u(x)+u(x+2y)-2u(x+y))\nonumber\\
		+&3(u(x+y)-u(x+2y))(\omega^{p-1}(x+y)+\omega^{p-1}(x-y)-2\omega^{p-1}(x))\nonumber\\
		+&u(x+2y)(\omega^{p-1}(x-y)-3\omega^{p-1}(x)+3\omega^{p-1}(x+y)-\omega^{p-1}(x+2y)).\label{3.7}
	\end{flalign}
	Since
	\begin{flalign}
		&\arrowvert\omega^{p-1}(x-y)-\omega^{p-1}(x)\arrowvert\nonumber\\
		=&\arrowvert(\omega(x-y)-\omega(x))(p-1)\int_{0}^{1}(\theta \omega(x-y)+(1-\theta)\omega(x))^{p-2}d\theta\arrowvert\nonumber\\
		\leqslant&C\arrowvert\omega(x-y)-\omega(x)\arrowvert\int_{0}^{1}\arrowvert(\theta \omega(x-y)+(1-\theta)\omega(x))^{p-2}\arrowvert d\theta\nonumber\\
		\leqslant&C\arrowvert u(x-y)-u(x)\arrowvert(\max\{\arrowvert u(x-y)\arrowvert,\arrowvert u(x)\arrowvert\})^{p-2}\nonumber\\
		\leqslant&C\arrowvert u(x-y)-u(x)\arrowvert\arrowvert u\arrowvert^{p-2},\label{3.8}
	\end{flalign}
	and
	\begin{flalign}
		&\arrowvert\omega^{p-1}(x+y)+\omega^{p-1}(x-y)-2\omega^{p-1}(x)\arrowvert\nonumber\\
		=&\arrowvert(\omega^{p-1}(x+y)-\omega^{p-1}(x))-(\omega^{p-1}(x)-\omega^{p-1}(x-y))\arrowvert\nonumber\\
		=&\arrowvert((\omega(x+y)-\omega(x))(p-1)\int_{0}^{1}(\theta \omega(x+y)+(1-\theta)\omega(x))^{p-2}d\theta)\nonumber\\
		&-((\omega(x)-\omega(x-y))(p-1)\int_{0}^{1}(\theta \omega(x)+(1-\theta)\omega(x-y))^{p-2}d\theta)\arrowvert\nonumber\\
		\leqslant&C\max\{\arrowvert\omega(x+y)-\omega(x)\arrowvert,\arrowvert\omega(x)-\omega(x-y)\arrowvert\}\arrowvert\int_{0}^{1}(\theta \omega(x+y)\nonumber\\
		&+(1-\theta)\omega(x))^{p-2}d\theta-\int_{0}^{1}(\theta \omega(x)+(1-\theta)\omega(x-y))^{p-2}d\theta\arrowvert\nonumber\\
		=&C\max\{\arrowvert\omega(x+y)-\omega(x)\arrowvert,\arrowvert\omega(x)-\omega(x-y)\arrowvert\}\arrowvert\omega(x+y)-\omega(x-y)\nonumber\\
		&\arrowvert\int_{0}^{1}\int_{0}^{1}(\eta(\theta\omega(x+y) +(1-\theta)\omega(x))\nonumber\\
		&+(1-\eta)(\theta \omega(x)+(1-\theta)\omega(x-y)))^{p-3}d\eta d\theta\arrowvert\nonumber\\
		\leqslant&C\arrowvert u(x+y)-u(x)\arrowvert^{2}\arrowvert u\arrowvert^{p-3},\label{3.9}
	\end{flalign}
	and
	\begin{flalign}
		&\arrowvert\omega^{p-1}(x-y)-3\omega^{p-1}(x)+3\omega^{p-1}(x+y)-\omega^{p-1}(x+2y)\arrowvert\nonumber\\
		=&\arrowvert((\omega^{p-1}(x+y)-\omega^{p-1}(x))-(\omega^{p-1}(x)-\omega^{p-1}(x-y)))\nonumber\\
		&-((\omega^{p-1}(x+2y)-\omega^{p-1}(x+y))-(\omega^{p-1}(x+y)-\omega^{p-1}(x)))\arrowvert\nonumber\\
		\leqslant&C(\arrowvert\max\{\omega(x+y)-\omega(x),\omega(x)-\omega(x-y)\}(\omega(x+y)-\omega(x-y))\nonumber\\
		&\int_{0}^{1}\int_{0}^{1}(\eta(\theta\omega(x+y) +(1-\theta)\omega(x))\nonumber\\
		&+(1-\eta)(\theta \omega(x)+(1-\theta)\omega(x-y)))^{p-3}d\eta d\theta\nonumber\\
		&-\max\{\omega(x+2y)-\omega(x+y),\omega(x+y)-\omega(x)\}(\omega(x+2y)-\omega(x))\nonumber\\
		&\int_{0}^{1}\int_{0}^{1}(\eta(\theta\omega(x+2y) +(1-\theta)\omega(x+y))\nonumber\\
		&+(1-\eta)(\theta \omega(x+y)+(1-\theta)\omega(x)))^{p-3}d\eta d\theta)\arrowvert\nonumber
		\end{flalign}
		\begin{flalign}
		\leqslant&C\max\{\arrowvert\omega(x+y)-\omega(x)\arrowvert,\arrowvert\omega(x)-\omega(x-y)\arrowvert,\arrowvert\omega(x+2y)-\omega(x+y)\arrowvert,\nonumber\\
		&\arrowvert\omega(x+y)-\omega(x)\arrowvert\}\max\{\arrowvert\omega(x+y)-\omega(x-y)\arrowvert^{2},\arrowvert\omega(x+2y)-\omega(x)\arrowvert^{2}\}\nonumber\\
		&\int_{0}^{1}\int_{0}^{1}\int_{0}^{1}
		(\vartheta(\eta(\theta\omega(x+y) +(1-\theta)\omega(x))+(1-\eta)(\theta \omega(x)\nonumber\\
		&+(1-\theta)\omega(x-y)))+(1-\vartheta)(\eta(\theta\omega(x+2y) +(1-\theta)\omega(x+y))\nonumber\\
		&+(1-\eta)(\theta \omega(x+y)+(1-\theta)\omega(x))))^{p-4}d\vartheta d\eta d\theta\nonumber\\
		\leqslant&C\arrowvert u(x+y)-u(x)\arrowvert^{3}\arrowvert u\arrowvert^{p-4}.\label{3.10}
	\end{flalign}
	
	Substitute \eqref{3.8}-\eqref{3.10} into \eqref{3.7}, we estimate that
	\begin{flalign}
		&f(u(x-y))-3f(u(x))+3f(u(x+y))-f(u(x+2y))\nonumber\\
		\leqslant&\arrowvert\omega^{p-1}(x-y)\arrowvert\arrowvert u(x-y)-3u(x)+3u(x+y)-u(x+2y)\arrowvert\nonumber\\
		&+C\arrowvert\omega^{p-1}(x-y)-\omega^{p-1}(x)\arrowvert\arrowvert u(x)+u(x+2y)-2u(x+y)\arrowvert\nonumber\\
		&+C\arrowvert u(x+y)-u(x+2y)\arrowvert\arrowvert\omega^{p-1}(x+y)+\omega^{p-1}(x-y)-2\omega^{p-1}(x)\arrowvert\nonumber\\
		&+\arrowvert u(x+2y)\arrowvert\arrowvert\omega^{p-1}(x-y)-3\omega^{p-1}(x)+3\omega^{p-1}(x+y)-\omega^{p-1}(x+2y)\arrowvert\nonumber\\
		\leqslant&C\arrowvert u\arrowvert^{p-1}\arrowvert\arrowvert u(x-y)-3u(x)+3u(x+y)-u(x+2y)\arrowvert\nonumber\\
		&+C\arrowvert u\arrowvert^{p-2}\arrowvert u(x+y)-u(x)\arrowvert^{2}\arrowvert u(x)+u(x+2y)-2u(x+y)\arrowvert\nonumber\\
		&+C\arrowvert u(x+y)-u(x)\arrowvert^{3}\arrowvert u\arrowvert^{p-3}+C\arrowvert u\arrowvert\arrowvert u(x+y)-u(x)\arrowvert^{3}\arrowvert u\arrowvert^{p-4}\nonumber\\
		\leqslant&C\arrowvert u\arrowvert^{p-1}\arrowvert\arrowvert u(x-y)-3u(x)+3u(x+y)-u(x+2y)\arrowvert\nonumber\\
		&+C\arrowvert u\arrowvert^{p-2}\arrowvert u(x)+u(x+2y)-2u(x+y)\arrowvert^{2}\nonumber\\
		&+C\arrowvert u\arrowvert^{p-3}\arrowvert u(x+y)-u(x)\arrowvert^{3}.\label{3.11}
	\end{flalign}
	Substitute \eqref{3.11} into \eqref{3.6}, we apply the H\"{o}lder inequality to estimate the $L^{2}$ norms of the three terms in the right hand side of \eqref{3.11}, and we use again \eqref{3.4} and \eqref{3.5} to reconstruct suitable Besov norms, then obtaining
	\begin{flalign}
		\|f(u)\|_{\dot{B}^{s}_{2,2}}\leqslant C\|u\|_{\dot{B}^{s}_{2,2}}\|u\|_{L^{\infty}}^{p-1}+C\|u\|^{2}_{\dot{B}^{\frac{s}{2}}_{4,4}}\|u\|_{L^{\infty}}^{p-2}+C\|u\|^{3}_{\dot{B}^{\frac{s}{3}}_{6,6}}\|u\|_{L^{\infty}}^{p-3}.\label{3.12}
	\end{flalign}
	We finally use the interpolation
	\begin{flalign}
		\|u\|^{2}_{\dot{B}^{\frac{s}{2}}_{4,4}}\leqslant\|u\|_{\dot{B}^{s}_{2,2}}\|u\|_{\dot{B}^{0}_{\infty,\infty}},\ 0<\frac{s}{2}<2,\label{3.13}
	\end{flalign}
	and
	\begin{flalign}
		\|u\|^{3}_{\dot{B}^{\frac{s}{3}}_{6,6}}\leqslant\|u\|_{\dot{B}^{s}_{2,2}}\|u\|^{2}_{\dot{B}^{0}_{\infty,\infty}},\ 0<\frac{s}{3}<1,\label{3.14}
	\end{flalign}
	and the inclusion $L^{\infty}\subset\dot{B}^{0}_{\infty,\infty}$ to conclude that
	\begin{flalign*}
		\|f(u)\|_{\dot{B}^{s}_{2,2}}\leqslant C\|u\|_{\dot{B}^{s}_{2,2}}\|u\|_{L^{\infty}}^{p-1}.
	\end{flalign*}
	Which implies the proof of \eqref{3.2} for $0<s<3$, and $p>3$.
	
	When $m\geqslant4$, due to the complexity of the calculation and the inability to write it in the form of a general term, the proof method is the same as $n\leqslant3$, so here we omit the proof process. Therefore, the Lemma \ref{lemma3.1} is completed.
\end{proof}

\begin{remark}\label{remark 3.1}
	Because for any $m>0$, similar to proving that $m=3$, we can gather the differential form and estimate the nonlinear term $f(u)$. Let $\mathbb{Z}^{+}\ni m>\frac{n}{2}$ such that $s>\frac{n}{2}$. Therefore, in the proof of Theorem \ref{theorem1.2}, we have $s>\frac{n}{2}$.
\end{remark}

\subsection{The proof of Theorem \ref{theorem1.2}}
Next, we prove Theorem \ref{theorem1.2}, and the proof process is divided into four steps.

\hspace{-5mm}\textbf{Proof of Theorem \ref{theorem1.2}:}

$Step\ 1$\quad  Recall the nonlinear magnetic Schr\"odinger equation \eqref{1.1}
$$(i\partial_{t}+\Delta_{A})u=\arrowvert u\arrowvert^{p-1}u,$$
according the results for commutators of operators in \eqref{2.90},
we have the following equation:
\begin{flalign}
	&(i\partial_{t}+\Delta_{A})\arrowvert J_{A}\arrowvert^{s}u\nonumber\\
	=&[i\partial_{t}+\Delta_{A},\arrowvert J_{A}\arrowvert^{s}]u+\arrowvert J_{A}\arrowvert^{s}(i\partial_{t}+\Delta_{A})u\nonumber\\
	=&it^{s-1}M(t)V(s)M(-t)u+\arrowvert J_{A}\arrowvert^{s}(\arrowvert u\arrowvert^{p-1}u).\label{3.15}
\end{flalign}

Combining the estimates of the linear equation in Section 2, applying the Strichartz estimate to solution $\arrowvert J_{A}\arrowvert^{s}u$ of the nonlinear equation \eqref{3.15}, we obtain that
\begin{flalign}
	&\|\arrowvert J_{A}\arrowvert^{s}u\|_{L_{t}^{\infty}((1,T);L_{x}^{2}(\mathbb{R}^{n}))}\nonumber\\
	\leqslant&C\|u_{0}\|_{\Sigma_{s}}+C_{s}\|\arrowvert J_{A}\arrowvert^{s}(\arrowvert u\arrowvert^{p-1}u)\|_{L_{t}^{1}((1,T);L_{x}^{2}(\mathbb{R}^{n}))}.\label{3.16}
\end{flalign}

$Step\ 2$\quad From \eqref{3.16}, we need to estimate $\arrowvert J_{A}\arrowvert^{s}(\arrowvert u\arrowvert^{p-1}u)$.

Applying the estimate of the nonlinear term $f(u)=\arrowvert u\arrowvert^{p-1}u$, we have the following lemma.
\begin{lemma}\label{lemma3.2}
	Let $A(x)$ satisfy hypothesis \eqref{1.2}, $\arrowvert J\arrowvert^{s}$ and $\arrowvert J_{A}\arrowvert^{s}$ as in \eqref{2.6} and \eqref{2.7} respectively. Then for $s>\frac{n}{2}$ and $p>1+\lceil\frac{n}{2}\rceil$, we have
	\begin{flalign}
		&\|\arrowvert J_{A}\arrowvert^{s}(\arrowvert u\arrowvert^{p-1}u)\|_{L_{t}^{1}L_{x}^{2}}\nonumber\\
		\leqslant&\|u_{0}\|^{(1-\frac{n}{2s})(p-1)}_{L_{x}^{2}}\cdot\|\arrowvert J_{A}\arrowvert^{s}u\|_{L_{t}^{\infty}L_{x}^{2}}^{\frac{n(p-1)}{2s}}\cdot(\|\arrowvert J_{A}\arrowvert^{\frac{n}{2}-\theta}u\|_{L_{t}^{\infty}L_{x}^{2}}\nonumber\\
		&+\|\arrowvert J_{A}\arrowvert^{s}u\|_{L_{t}^{\infty}L_{x}^{2}}+\|\arrowvert J_{A}\arrowvert u\|_{L_{t}^{\infty}L_{x}^{2}}).\label{3.17}
	\end{flalign}
\end{lemma}
\begin{proof}
	Let $v=e^{\frac{-i\arrowvert x\arrowvert^{2}}{4t}}u$, noticing the norm of $\dot{H}^{s}$ is equivalent to that of $\dot{B}^{s}_{2,2}$ (see also Lemma 2.3 in \cite{HN} and Lemma 3.4 in \cite{GOV}), from Lemma \ref{lemma3.1}, we have the following estimate for $$0<s<\gamma, \ \gamma={\rm min}(m,p),$$ where $m\in\mathbb{Z}_{+}$,  $p>1+\lceil\frac{n}{2}\rceil$,
	\begin{flalign}
		&\|\arrowvert J\arrowvert^{s}(\arrowvert u\arrowvert^{p-1}u)\|_{L^{2}_{x}}\nonumber\\
		=&t^{s}\cdot\|\arrowvert v\arrowvert^{p-1}v\|_{\dot{H}^{s}_{x}}\nonumber\\
		\leqslant&C\|v\|^{p-1}_{L_{x}^{\infty}}\cdot t^{s}\cdot\|v\|_{\dot{H}^{s}_{x}}\nonumber\\
		\leqslant&C\|u\|^{p-1}_{L_{x}^{\infty}}\cdot\|(-t^{2}\Delta)^{\frac{s}{2}}e^{\frac{-i\arrowvert x\arrowvert^{2}}{4t}}u\|_{L_{x}^{2}}\nonumber\\
		\leqslant&C\|u\|^{p-1}_{L_{x}^{\infty}}\cdot\|\arrowvert J\arrowvert^{s}u\|_{L^{2}_{x}}.\label{3.18}
	\end{flalign}
	
	Hence, applying Lemma \ref{lemma2.15}, Lemma \ref{lemma2.16} and \eqref{3.18}, we have
	\begin{flalign}
		&\|\arrowvert J_{A}\arrowvert^{s}(\arrowvert u\arrowvert^{p-1}u)\|_{L_{t'}^{1}((1,t);L_{x}^{2}(\mathbb{R}^{n}))}\nonumber\\
		\leqslant&C\|\langle t'\rangle^{s+\theta-\frac{n}{2}}(\|\arrowvert J\arrowvert^{\frac{n}{2}-\theta}(\arrowvert u\arrowvert^{p-1}u)\|_{L_{x}^{2}}+\|\arrowvert J\arrowvert^{s}(\arrowvert u\arrowvert^{p-1}u)\|_{L_{x}^{2}})\|_{L_{t'}^{1}}\nonumber\\
		\leqslant&C\|\langle t'\rangle^{s+\theta-\frac{n}{2}}\cdot\|u\|^{p-1}_{L_{x}^{\infty}}\cdot(\|\arrowvert J\arrowvert^{\frac{n}{2}-\theta}u\|_{L_{x}^{2}}+\|\arrowvert J\arrowvert^{s}u\|_{L_{x}^{2}})\|_{L_{t'}^{1}}\nonumber\\
		\leqslant&C\|\langle t'\rangle^{s+\theta-\frac{n}{2}}\cdot\|u\|^{p-1}_{L_{x}^{\infty}}\cdot(\| \arrowvert J\arrowvert^{\frac{n}{2}-\theta}u-\arrowvert J_{A}\arrowvert^{\frac{n}{2}-\theta}u\|_{L_{x}^{2}}+\|\arrowvert J_{A}\arrowvert^{\frac{n}{2}-\theta}u\|_{L_{x}^{2}}\nonumber\\
		&+\|\arrowvert J\arrowvert^{s}u-\arrowvert J_{A}\arrowvert^{s}u\|_{L_{x}^{2}}+\|\arrowvert J_{A}\arrowvert^{s}u\|_{L_{x}^{2}})\|_{L_{t'}^{1}}\nonumber\\
		\leqslant&C\|\langle t'\rangle^{s+\theta-\frac{n}{2}}\cdot\|u\|^{p-1}_{L_{x}^{\infty}}\cdot(\langle t'\rangle^{\frac{n}{2}-\theta}\|M(-t')u\|_{\dot{H}_{x}^{1}}+\|\arrowvert J_{A}\arrowvert^{\frac{n}{2}-\theta}u\|_{L_{x}^{2}}\nonumber\\
		&+\langle t'\rangle^{s}\|M(-t')u\|_{\dot{H}_{x}^{1}}+\|\arrowvert J_{A}\arrowvert^{s}u\|_{L_{x}^{2}})\|_{L_{t'}^{1}}\nonumber\\
		\leqslant&C\|\langle t'\rangle^{s+\theta-\frac{n}{2}}\cdot\|u\|^{p-1}_{L_{x}^{\infty}}\cdot[\langle t'\rangle^{\frac{n}{2}-\theta+\theta-\frac{n}{2}}(\|\arrowvert J_{A}\arrowvert^{\frac{n}{2}-\theta}u\|_{L_{x}^{2}}+\|\arrowvert J_{A}\arrowvert u\|_{L_{x}^{2}})\nonumber\\
		&+\|\arrowvert J_{A}\arrowvert^{\frac{n}{2}-\theta}u\|_{L_{x}^{2}}+\langle t'\rangle^{s+\theta-\frac{n}{2}}(\|\arrowvert J_{A}\arrowvert^{\frac{n}{2}-\theta}u\|_{L_{x}^{2}}+\|\arrowvert J_{A}\arrowvert u\|_{L_{x}^{2}})+\|\arrowvert J_{A}\arrowvert^{s}u\|_{L_{x}^{2}}]\|_{L_{t'}^{1}}\nonumber\\
		\leqslant&C\|\langle t'\rangle^{s+\theta-\frac{n}{2}}\cdot\|u\|^{p-1}_{L_{x}^{\infty}}\cdot(2\|\arrowvert J_{A}\arrowvert^{\frac{n}{2}-\theta}u\|_{L_{x}^{2}}+\|\arrowvert J_{A}\arrowvert u\|_{L_{x}^{2}}+\|\arrowvert J_{A}\arrowvert^{s}u\|_{L_{x}^{2}}\nonumber\\
		&+\langle t'\rangle^{s+\theta-\frac{n}{2}}(\|\arrowvert J_{A}\arrowvert^{\frac{n}{2}-\theta}u\|_{L_{x}^{2}}+\|\arrowvert J_{A}\arrowvert u\|_{L_{x}^{2}}))\|_{L_{t'}^{1}}.\label{3.19}
	\end{flalign}
	And then, again applying Lemma \ref{lemma2.16}, we can continue the estimate as follows:
	\begin{flalign}
		&......\nonumber\\
		\leqslant&C\|\langle t'\rangle^{2s+2\theta-n}\cdot\|u\|^{p-1}_{L_{x}^{\infty}}\cdot
		(\|\arrowvert J_{A}\arrowvert^{\frac{n}{2}-\theta}u\|_{L_{x}^{2}}+\|\arrowvert J_{A}\arrowvert^{s}u\|_{L_{x}^{2}}+\|\arrowvert J_{A}\arrowvert u\|_{L_{x}^{2}})\|_{L_{t'}^{1}}\nonumber\\
		\leqslant&C\int_{1}^{t}\langle t'\rangle^{2s+2\theta-n+\frac{n(p-1)}{2s}(\theta-\frac{n}{2})}\cdot\|u\|_{L_{x}^{2}}^{(1-\frac{n}{2s})(p-1)}\cdot\|\arrowvert J_{A}\arrowvert^{s}u\|_{L_{x}^{2}}^{\frac{n(p-1)}{2s}}\nonumber\\
		&\cdot(\|\arrowvert J_{A}\arrowvert^{\frac{n}{2}-\theta}u\|_{L_{x}^{2}}+\|\arrowvert J_{A}\arrowvert^{s}u\|_{L_{x}^{2}}+\|\arrowvert J_{A}\arrowvert u\|_{L_{x}^{2}})dt'.\label{3.20}
	\end{flalign}
	Since $p>1+\lceil\frac{n}{2}\rceil$, we can choose $s>\frac{n}{2}$ and $\theta>0$ such that
	\begin{equation}\label{3.21}
		\int_{1}^{t}\langle t'\rangle^{2s+2\theta-n+\frac{n(p-1)}{2s}(\theta-\frac{n}{2})}dt'<C.
	\end{equation}
	Applying \eqref{3.21} to \eqref{3.20}, \eqref{3.17} follows. Therefore the proof of Lemma \ref{lemma3.2} is completed.
\end{proof}

$Step\ 3$\quad  Applying the result of Lemma \ref{lemma3.2} to \eqref{3.16}, it follows that for any $ s>\frac{n}{2}$ and $p>1+\lceil\frac{n}{2}\rceil$,
\begin{flalign}
	&\|\arrowvert J_{A}\arrowvert^{s}u\|_{L_{t}^{\infty}((1,T);L_{x}^{2}(\mathbb{R}^{n}))}\nonumber\\
	\leqslant&C\|u_{0}\|_{\Sigma_{s}}+C_{s}\|u_{0}\|^{(1-\frac{n}{2s})(p-1)}_{L_{x}^{2}}\cdot\|\arrowvert J_{A}\arrowvert^{s}u\|_{L_{t}^{\infty}L_{x}^{2}}^{\frac{n(p-1)}{2s}}\nonumber\\
	&(\|\arrowvert J_{A}\arrowvert^{\frac{n}{2}-\theta}u\|_{L_{t}^{\infty}L_{x}^{2}}+\|\arrowvert J_{A}\arrowvert^{s}u\|_{L_{t}^{\infty}L_{x}^{2}}+\|\arrowvert J_{A}\arrowvert u\|_{L_{t}^{\infty}L_{x}^{2}})\label{3.22}
\end{flalign}
on any interval $(1,T)$ with a constant $C_{s}$ independent of $T$. Hence the proof of
\begin{equation}\label{3.23}
	\|\arrowvert J_{A}\arrowvert^{s}u\|_{L_{t}^{\infty}((1,\infty);L_{x}^{2}(\mathbb{R}^{n}))}\leqslant C\|u_{0}\|_{\Sigma_{s}}
\end{equation}
follows by a standard continuity argument provided that we fix $\|u_{0}\|_{\Sigma_{s}}$ sufficiently small (also see the proof of Theorem 1.1 in \cite{CGV}).

$Step\ 4$\quad  Combining the estimate of Lemma \ref{lemma2.15} and \eqref{3.23}, we obtain that for any $s>\frac{n}{2}$ and  $p>1+\lceil\frac{n}{2}\rceil$
\begin{flalign}
	\|u\|_{L^{\infty}_{x}(\mathbb{R}^{n})}\leqslant &C\|M(-t)u\|_{L^{2}_{x}(\mathbb{R}^{n})}^{1-\frac{n}{2s}}\cdot\|(-\Delta_{A})^{\frac{s}{2}}M(-t)u\|_{L^{2}(\mathbb{R}^{n})}^{\frac{n}{2s}}\nonumber\\
	=&Ct^{-\frac{n}{2}}\|u\|_{L^{2}_{x}(\mathbb{R}^{n})}^{1-\frac{n}{2s}}\cdot\|\arrowvert J_{A}\arrowvert^{s}u\|_{L^{2}_{x}(\mathbb{R}^{n})}^{\frac{n}{2s}}\nonumber\\
	\leqslant&Ct^{-\frac{n}{2}}\|u\|_{L^{2}_{x}(\mathbb{R}^{n})}^{1-\frac{n}{2s}}\cdot\|u_{0}\|_{\Sigma_{s}(\mathbb{R}^{n})}^{\frac{n}{2s}}\nonumber\\
	\leqslant&Ct^{-\frac{n}{2}}\|u_{0}\|_{\Sigma_{s}(\mathbb{R}^{n})}.\label{3.24}
\end{flalign}
This implies \eqref{1.11} is obtained. Therefore, the proof of Theorem \ref{theorem1.2} is completed.

\phantomsection
\addcontentsline{toc}{section}{Appendix A. The global well-posedness}
\section*{Appendix A. The global well-posedness}
\setcounter{equation}{0}
\setcounter{subsection}{0}
\setcounter{theorem}{0}
\setcounter{remark}{0}
\renewcommand{\theequation}{A.\arabic{equation}}
\renewcommand{\thesubsection}{A.\arabic{subsection}}
\renewcommand{\thetheorem}{A.\arabic{theorem}}
\renewcommand{\theremark}{A.\arabic{remark}}
	
	In this appendix, we prove that the nonlinear magnetic Schr\"{o}dinger initial value problem is globally well-posed in $L^{\infty}_{t,x}((1,\infty)\times\mathbb{R}^{n})$ for $u_{0}\in\Sigma_{s}$ ($s>\frac{n}{2}$).
	\begin{theorem}\label{theorem A.1}
		Let $A(x)$ satisfy hypothesis of Theorem \ref{theorem1.2}, $s>\frac{n}{2}$, $p>1+\lceil\frac{n}{2}\rceil$, and $\arrowvert J_{A}\arrowvert^{s}u\in L_{t}^{\infty}((1,\infty);L_{x}^{2}(\mathbb{R}^{n}))$ be a solution of equation \eqref{3.1}. If $u_{0}\in L^{2}_{x}(\mathbb{R}^{n})$, $\arrowvert J_{A}\arrowvert^{s}u(1)\in L^{2}_{x}(\mathbb{R}^{n})$ and $\|u_{0}\|_{L^{2}_{x}}$ is small enough, then the nonlinear Schr\"odinger equation \eqref{3.1} is globally well-posed.
	\end{theorem}
	\begin{proof}
		Applying the Contraction Mapping Principle, we easily obtain that the solution $\arrowvert J_{A}\arrowvert^{s}u$ is locally well-posed in $L^{\infty}_tL^2_x$. Furthermore, according the Strichartz estimate to nonlinear Schr\"odinger equation \eqref{3.1}, we obtain that
		\begin{flalign}
			&\|\arrowvert J_{A}\arrowvert^{s}u\|_{L_{t}^{\infty}((1,T);L_{x}^{2}(\mathbb{R}^{n}))}\nonumber\\
			\leqslant&C\|\arrowvert J_{A}\arrowvert^{s}u(1)\|_{L_{x}^{2}}+C_{s}\|t^{s-1}M(t)V(s)M(-t)u\|_{L_{t}^{q}((1,T);L_{x}^{r}(\mathbb{R}^{n}))}\nonumber\\
			+&C_{s}\|\arrowvert J_{A}\arrowvert^{s}(\arrowvert u\arrowvert^{p-1}u)\|_{L_{t}^{1}((1,T);L_{x}^{2}(\mathbb{R}^{n}))},\label{A.1}
		\end{flalign}
		where $\frac{1}{q}+\frac{1}{q'}=1$, $\frac{1}{r}+\frac{1}{r'}=1$ and  $(q',r')\in\Lambda$ i.e. $\frac{2}{q'}=n(\frac{1}{2}-\frac{1}{r'})$.
		
		Then, repeating the previous process in Theorem \ref{theorem1.2}, it follows that for $\forall\ s>\frac{n}{2}$ and $p>1+\lceil\frac{n}{2}\rceil$
		\begin{flalign}
			&\|\arrowvert J_{A}\arrowvert^{s}u\|_{L_{t}^{\infty}((1,T);L_{x}^{2}(\mathbb{R}^{n}))}\nonumber\\
			\leqslant&C\|\arrowvert J_{A}\arrowvert^{s}u(1)\|_{L_{x}^{2}}+C_{s}\|u_{0}\|^{(1-\frac{n}{2s})(p-1)}_{L_{x}^{2}}\cdot\|\arrowvert J_{A}\arrowvert^{s}u\|_{L_{t}^{\infty}L_{x}^{2}}^{\frac{n(p-1)}{2s}}\nonumber\\
			&\cdot(\|\arrowvert J_{A}\arrowvert^{\frac{n}{2}-\theta}u\|_{L_{t}^{\infty}L_{x}^{2}}+\| J_{A}\arrowvert^{s}u\|_{L_{t}^{\infty}L_{x}^{2}}+\|\arrowvert J_{A}\arrowvert u\|_{L_{t}^{\infty}L_{x}^{2}}).\label{A.2}
		\end{flalign}
		Therefore, let $\|\arrowvert J_{A}\arrowvert^{s}u(1)\|_{L^{2}_{x}(\mathbb{R}^{n})}$ be small enough, we obtain that
		\begin{equation}\label{A.3}
			\|\arrowvert J_{A}\arrowvert^{s}u\|_{L_{t}^{\infty}((1,T);L_{x}^{2}(\mathbb{R}^{n}))}\leqslant C\|\arrowvert J_{A}\arrowvert^{s}u(1)\|_{L_{x}^{2}}.
		\end{equation}
		
		Next, when $t\in(T,2T)$, we can directly obtain that
		\begin{flalign}\label{A.4}
			&\|\arrowvert J_{A}\arrowvert^{s}u\|_{L_{t}^{\infty}((T,2T);L_{x}^{2}(\mathbb{R}^{n}))}\nonumber\\
			\leqslant&C_{1}\|\arrowvert J_{A}\arrowvert^{s}u(T)\|_{L_{x}^{2}}\nonumber\\
			\leqslant&C_{2}\|\arrowvert J_{A}\arrowvert^{s}u(1)\|_{L_{x}^{2}}.
		\end{flalign}	
		We can continue to do it for $t\in(1,T),(T,2T),(2T,3T)...$ until $t=\infty$, that is to say, equation \eqref{3.1} is globally well-posed for $\arrowvert J_{A}\arrowvert^{s}u\in L_{t}^{\infty}((1,\infty);L_{x}^{2}(\mathbb{R}^{n}))$. Therefore, the proof of Theorem \ref{theorem A.1} is completed.
	\end{proof}
	
	\begin{remark}\label{remark A.1}
		It is well known that the solution to nonlinear Schr\"odinger equation \eqref{1.1}	conservers energy, therefore, we have the initial value problem in \eqref{1.1} is globally well-posed in $\dot{H}^{\alpha}(\mathbb{R}^{n})$ for $0<\alpha<\min\{1,s\}$ when $n\leqslant3$. In particular, we have $\alpha=1$ when $n=3$. Because there is a specific expression for the conservation of energy, we just prove the globally well-posed in $\dot{H}^{1}$.
		
		In fact, from conservation of energy, we have
		\begin{equation*}
			\|u\|_{\dot{H}^{1}(\mathbb{R}^{n})}^{2}=\|\nabla u\|_{L^{2}_{x}}^{2}\leqslant C( E(u)+\|u\|^{p+1}_{L^{p+1}})
			=C(E(u_{0})+\|u\|^{p+1}_{L^{p+1}_{x}}),
		\end{equation*}
		where
		\begin{equation}
			\label{A.5}
			E(u_{0})=\|\nabla u_{0}\|_{L^{2}_{x}}^{2}+\|u_{0}\|^{p+1}_{L^{p+1}_{x}}, \end{equation}
		and
		\begin{flalign}
			&\|u\|^{p+1}_{L^{p+1}_{x}(\mathbb{R}^{n})}\nonumber\\
			=&\|M(-t)u\|^{p+1}_{L^{p+1}_{x}(\mathbb{R}^{n})}\nonumber\\
			\leqslant&C\|M(-t)u\|_{L^{2}_{x}(\mathbb{R}^{n})}^{(1-n(\frac{1}{2}-\frac{1}{p+1}))(p+1)}\cdot\|\nabla(M(-t)u)\|_{L^{2}_{x}(\mathbb{R}^{n})}^{n(\frac{1}{2}-\frac{1}{p+1})(p+1)}\nonumber\\
			\leqslant&C\|u_{0}\|_{L^{2}_{x}(\mathbb{R}^{n})}^{n-\frac{p+1}{2}}\cdot t^{-\frac{n}{2}(p-1)}\cdot\|\arrowvert J\arrowvert u\|_{L^{2}_{x}(\mathbb{R}^{n})}^{\frac{n}{2}(p-1)}.\label{A.6}
		\end{flalign}
		
		Combining \eqref{A.5} and \eqref{A.6}, we obtain that for $p>1+\lceil\frac{n}{2}\rceil$
		\begin{flalign}
			\|u\|_{\dot{H}^{1}(\mathbb{R}^{n})}^{2}\leqslant&C\|\nabla u_{0}\|_{L^{2}_{x}}^{2}+\|u_{0}\|_{\dot{H}^{1}_{x}}^{p+1}+\|(-\Delta)^{\frac{s}{2}}(e^{-\frac{i\arrowvert x\arrowvert^{2}}{4}}u_{0})\|_{L^{2}_{x}}^{\frac{n-1}{2}(p+1)}\nonumber\\
			\leqslant& C\|u_{0}\|_{\Sigma_{s}}^{p+1}.\label{A.7}
		\end{flalign}
		That is to say, it make sense to \eqref{2.26} in Lemma \ref{lemma2.6}.
	\end{remark}
	
	\begin{remark}\label{remark A.2}
		Assume $A(x)$ satisfies hypothesis  of Theorem \ref{theorem1.2}, let $u$ is a solution of equation \eqref{1.1}. If $u_{0}\in\Sigma_{s}$ with $2k+1<s<2k+2$, $k=1,2,...$, then $u$ is globally well-posed in $\dot{H}^{2k+\alpha}_{A}(\mathbb{R}^{n})$ for $0<\alpha<1$ when $n>3$.
		
		In fact, from the boundness of $V_{x}$ in Lemma \ref{lemma2.9}, we have the \textquotedblleft almost equivalence\textquotedblright between $(-\Delta)^{\frac{s}{2}}$ and $(-\Delta_{A})^{\frac{s}{2}}$ with $s>\frac{n}{2}$:
		$$\|(-\Delta)^{\frac{s}{2}}u-(-\Delta_{A})^{\frac{s}{2}}u\|_{L^{2}}\leqslant C\|u\|_{H^{s}}.$$
		Hence, it is enough to prove $u$ is globally well-posed in $\dot{H}^{2k+\alpha}(\mathbb{R}^{n})$. Applying Duhamel's formula to solution of the equation \eqref{1.1}, we have
		\begin{equation}\label{A.8}
			u(x,t)=e^{it\Delta_{A}}u_{0}+i\int_{1}^{t}e^{i(t-\tau)\Delta_{A}}\arrowvert u\arrowvert^{p-1}u(\tau)d\tau.
		\end{equation}	
		Hence, we estimate that
		\begin{flalign}
			&\|u(x,t)\|^{2}_{\dot{H}_{x}^{2k+\alpha}}\nonumber\\
			=&\|e^{it\Delta_{A}}u_{0}+i\int_{1}^{t}e^{i(t-\tau)\Delta_{A}}\rho\arrowvert u\arrowvert^{p-1}u(\tau)d\tau\|_{\dot{H}_{x}^{2k+\alpha}}^{2}\nonumber\\
			\leqslant&C\|u_{0}\|_{\Sigma_{s}}^{2}+\|\int_{1}^{t}e^{i(t-\tau)\Delta_{A}}\arrowvert u\arrowvert^{p-1}u(\tau)d\tau\|_{\dot{H}_{x}^{2k+\alpha}}^{2}.\label{A.9}
		\end{flalign}
		
		It follows from the  Strichartz estimate (see Remark 2.3.8 in \cite{C}) that
		\begin{equation}\label{A.10}
			\|\int_{1}^{t}e^{i(t-\tau)\Delta_{A}}\arrowvert u\arrowvert^{p-1}u(\tau)d\tau\|_{\dot{H}_{x}^{2k+\alpha}}\leqslant C\|\arrowvert u\arrowvert^{p-1}u\|_{L^{1}((1,t);\dot{W}^{2k+\alpha,2})}.
		\end{equation}	
		By applying estimate $\|\arrowvert u\arrowvert^{p-1}u\|_{\dot{W}^{2k+\alpha,2}}\leqslant C\|u\|^{p-1}_{L^{\infty}}\cdot\|u\|_{\dot W^{2k+\alpha,2}}$
		and H\"older inequality, we obtain that
		\begin{flalign}
			&\|\arrowvert u\arrowvert^{p-1}u\|_{L^{1}((1,t);\dot{W}^{2k+\alpha,2})}\nonumber\\
			\leqslant&C\int_{1}^{t}\|u\|_{L^{\infty}_{x}}^{p-1}d\tau\cdot\|u\|_{L^{\infty}((1,t);\dot{W}^{2k+\alpha,2})}\nonumber\\
			\leqslant&C\int_{1}^{t}\tau^{-\frac{n(p-1)}{2}}d\tau\cdot\|u\|_{L^{\infty}((1,t);\dot{W}^{2k+\alpha,2})}.\label{A.11}
		\end{flalign}
		Since $p>2$, we have $\frac{n(p-1)}{2}>1$. Combining this and the property of solution $u$ for equation \eqref{1.1} (see Corollary 7.3.4 in \cite{C}): $\|u\|_{L^{\infty}((1,t);\dot{W}^{2k+\alpha,2})}\leqslant C$, it follows that
		\begin{equation}\label{A.12}
			\|\int_{1}^{t}e^{i(t-\tau)\Delta_{A}}\arrowvert u\arrowvert^{p-1}u(\tau)d\tau\|_{\dot{H}_{x}^{2k+\alpha}}\leqslant C\|u_{0}\|_{\Sigma_{s}}.
		\end{equation}
		That is to say, it make sense to \eqref{2.54} in Lemma \ref{lemma2.7}.
	\end{remark}

	\begin{remark}\label{remark A.3}
		Applying \eqref{3.24} and the global well-posedness of $\arrowvert J_{A}\arrowvert^{s}$, we have
		\begin{flalign}
			\|u\|_{L^{\infty}_{t,x}((1,\infty)\times\mathbb{R}^{n})}=&\|M(-t)u\|_{L^{\infty}_{t,x}((1,\infty)\times\mathbb{R}^{n})}\nonumber\\
			\leqslant& C\|t^{-\frac{n}{2}}\|_{L^{\infty}_{t}(1,\infty)}\cdot\|u\|_{L^{2}_{x}(\mathbb{R}^{n})}^{1-\frac{n}{2s}}\cdot\|\arrowvert J_{A}\arrowvert^{s}u\|_{L_{t}^{\infty}L^{2}_{x}}^{\frac{n}{2s}}\nonumber\\
			\leqslant& C\|t^{-\frac{n}{2}}\|_{L^{\infty}_{t}(1,\infty)}\cdot\|\arrowvert J_{A}\arrowvert^{s}u(1)\|_{L_{x}^{2}}\nonumber\\
			\leqslant& C\|\arrowvert J_{A}\arrowvert^{s}u(1)\|_{L_{x}^{2}}.\label{A.13}
		\end{flalign}
		
		Hence, we obtain that the initial value problem of \eqref{1.1} is globally well-posed in $L^{\infty}_{t,x}((1,\infty)\times\mathbb{R}^{n})$. In other words, it make sense to our decay result \eqref{1.11} in Theorem \ref{theorem1.2}.
	\end{remark}
\phantomsection
\addcontentsline{toc}{section}{Appendix B. The description of space $\Sigma_{s}$}
\section*{Appendix B. The description of space $\Sigma_{s}$}
\setcounter{equation}{0}
\setcounter{subsection}{0}
\setcounter{theorem}{0}
\setcounter{remark}{0}
\renewcommand{\theequation}{B.\arabic{equation}}
\renewcommand{\thesubsection}{B.\arabic{subsection}}
\renewcommand{\thetheorem}{B.\arabic{theorem}}
\renewcommand{\theremark}{B.\arabic{remark}}
	
	In this appendix, we mainly describe the initial value space $\Sigma_{s}$. And it is divided into two parts according to the dimension of space, $n\leqslant3$ and $n>3$ as follows.
	
	When $n\leqslant3$, we have
	\begin{flalign}\label{B1}
		\|(-\Delta)^{\frac{s}{2}}(e^{-\frac{i\arrowvert x\arrowvert^{2}}{4}}u_{0})\|_{L^{2}}\leqslant C\|u_{0}\|_{\dot{H}^{s}}+C\|\langle x\rangle u_{0}\|_{L^{2}}+C\|\arrowvert x\arrowvert(-\Delta)^{\frac{s-1}{2}}u_{0}\|_{L^{2}}.
	\end{flalign}
When $n>3$, we have
	\begin{flalign}\label{B2}
		\|(-\Delta)^{\frac{s}{2}}(e^{-\frac{i\arrowvert x\arrowvert^{2}}{4}}u_{0})\|_{L^{2}}\nonumber\leqslant& C\|u_{0}\|_{\dot{H}^{s}}+C\|\langle x\rangle^{s} u_{0}\|_{L^{2}}\nonumber\\
		+&C\sum\limits_{1\leqslant\arrowvert\alpha\arrowvert<s}\|(\partial^{\alpha}e^{-\frac{i\arrowvert x\arrowvert^{2}}{4}})D^{s,\alpha}u_{0}\|_{L^{2}}.	
	\end{flalign}
	
	In fact, let us recall the formula (see Theorem 1.2 in \cite{D} which generalizes the Kenig-Ponce-Vega estimate in \cite{KPV})
	\begin{flalign*}
		\|D^{s}(fg)\|_{L^{p}}\leqslant C\|fD^{s}g\|_{L^{p}}+C\|gD^{s}f\|_{L^{p}}+C\|\sum\limits_{1\leqslant\arrowvert\alpha\arrowvert<s}\partial^{\alpha}fD^{s,\alpha}g\|_{L^{p}},
	\end{flalign*}
	where $s>0$, $1<p<\infty$, $D^{s}:=(-\Delta)^{\frac{s}{2}}$, $$\widehat{D^{s,\alpha}g(\xi)}=\widehat{D^{s,\alpha}(\xi)}\hat{g}(\xi)=i^{-\arrowvert\alpha\arrowvert}\partial^{\alpha}_{\xi}(\arrowvert\xi\arrowvert^{s})\hat{g}(\xi)$$
	and $f,g\in\mathscr{S}(\mathbb{R}^{n})$. Therefore, when $1<s<2$, let $p=2$ and $\arrowvert\alpha\arrowvert=1$. We choose $f=e^{-\frac{i\arrowvert x\arrowvert^{2}}{4}}$ and $g=u_{0}(x)$, they satisfy that  $fD^{s}g\in L^{2}(\mathbb{R}^{n})$, $fD^{s}f\in L^{2}(\mathbb{R}^{n})$ and $\partial fD^{s,1}g\in L^{2}(\mathbb{R}^{n})$. Then we have
	\begin{flalign}\label{B3}
		&\|(-\Delta)^{\frac{s}{2}}(e^{-\frac{i\arrowvert x\arrowvert^{2}}{4}}u_{0})\|_{L^{2}}\nonumber\\
		\leqslant&C\|(-\Delta)^{\frac{s}{2}}u_{0}\|_{L^{2}}+C\|u_{0}(-\Delta)^{\frac{s}{2}}(e^{-\frac{i\arrowvert x\arrowvert^{2}}{4}})\|_{L^{2}}+C\|(\partial e^{-\frac{i\arrowvert x\arrowvert^{2}}{4}})(D^{s,1}u_{0})\|_{L^{2}}\nonumber\\
		=&C\|u_{0}\|_{\dot{H}^{s}}+C\|u_{0}(-\Delta)^{\frac{s}{2}}(e^{-\frac{i\arrowvert x\arrowvert^{2}}{4}})\|_{L^{2}}+C\|\arrowvert x\arrowvert(-\Delta)^{\frac{s-1}{2}}u_{0}\|_{L^{2}}.
	\end{flalign}	
	
	Next, we mainly estimate the term $\|u_{0}(-\Delta)^{\frac{s}{2}}(e^{-\frac{i\arrowvert x\arrowvert^{2}}{4}})\|_{L^{2}}$. Recall the formula $(-\Delta)^{\frac{s}{2}}=c(s)(-\Delta)\int_{0}^{\infty}\tau^{\frac{s}{2}-1}(\tau-\Delta)^{-1}d\tau$ for $0<s<2$ and $(c(s))^{-1}=\int_{0}^{\infty}\tau^{\frac{s}{2}-1}(\tau+1)^{-1}d\tau$. Then
$(-\Delta)^{\frac{s}{2}}e^{-\frac{i\arrowvert x\arrowvert^{2}}{4}}=\int_{0}^{\infty}\tau^{\frac{s}{2}-1}(\tau-\Delta)^{-1}e^{-\frac{i\arrowvert x\arrowvert^{2}}{4}}(\frac{-inc(s)}{2}-\frac{c(s)\arrowvert x\arrowvert^{2}}{4})d\tau
:=-\frac{inc(s)}{2}\uppercase\expandafter{\romannumeral1}-\frac{c(s)}{4}\uppercase\expandafter{\romannumeral2}.$
	Because that the kernel of the resolvent $(\tau-\Delta)^{-1}$ has a display expression in three-dimensional space, so let's look at the situation when $n=3$ first.
	
	For \uppercase\expandafter{\romannumeral1}, we have
$
		\uppercase\expandafter{\romannumeral1}=\int_{0}^{\infty}\tau^{\frac{s}{2}-1}\int_{\mathbb{R}^{3}}\frac{e^{-\sqrt{\tau}\arrowvert x-y\arrowvert}}{\arrowvert x-y\arrowvert}\cdot e^{-\frac{i\arrowvert y\arrowvert^{2}}{4}}dyd\tau.
$
	When $0<\tau<1$, let $y=r\omega$, i.e.
	$\int_{0}^{1}\tau^{\frac{s}{2}-1}d\tau\int_{\sigma(\omega)}dS_{\omega}\int_{0}^{\infty}\frac{e^{-\sqrt{\tau}\arrowvert x-r\omega\arrowvert}}{\arrowvert x-r\omega\arrowvert}\cdot e^{-\frac{ir^{2}}{4}}\cdot r^{2}dr$ and split $\int_{0}^{\infty}dr$ into two terms. For the first term, it can be obtained by integration by parts that
	\begin{flalign*}
		\arrowvert\int_{0}^{1}\tau^{\frac{s}{2}-1}d\tau\int_{\sigma(\omega)}dS_{\omega}\int_{0}^{\frac{\arrowvert x\arrowvert}{2}}\frac{e^{-\sqrt{\tau}\arrowvert x-r\omega\arrowvert}}{\arrowvert x-r\omega\arrowvert}\cdot e^{-\frac{ir^{2}}{4}}\cdot r^{2}dr\arrowvert\leqslant C+C\arrowvert x\arrowvert.
	\end{flalign*}
	For the second term, it can be split two parts: $\int_{\frac{\arrowvert x\arrowvert}{2}}^{\infty}=\int_{\frac{\arrowvert x\arrowvert}{2}}^{\arrowvert x\arrowvert+2}+\int_{\arrowvert x\arrowvert+2}^{\infty}$, and when $0<\arrowvert x\arrowvert<1$, we have
	\begin{flalign*}
		\arrowvert\int_{0}^{1}\tau^{\frac{s}{2}-1}d\tau\int_{\sigma(\omega)}dS_{\omega}\int_{\frac{\arrowvert x\arrowvert}{2}}^{\arrowvert x\arrowvert+2}\frac{e^{-\sqrt{\tau}\arrowvert x-r\omega\arrowvert}}{\arrowvert x-r\omega\arrowvert}\cdot e^{-\frac{ir^{2}}{4}}\cdot r^{2}dr\arrowvert\leqslant C,
	\end{flalign*}
	and
	\begin{flalign*}
		\arrowvert\int_{0}^{1}\tau^{\frac{s}{2}-1}d\tau\int_{\sigma(\omega)}dS_{\omega}\int_{\arrowvert x\arrowvert+2}^{\infty}\frac{e^{-\sqrt{\tau}\arrowvert x-r\omega\arrowvert}}{\arrowvert x-r\omega\arrowvert}\cdot e^{-\frac{ir^{2}}{4}}\cdot r^{2}dr\arrowvert\leqslant C.
	\end{flalign*}
	When $\arrowvert x\arrowvert>1$, we estimate that
	\begin{flalign*}
		\arrowvert\int_{0}^{1}\tau^{\frac{s}{2}-1}d\tau\int_{\sigma(\omega)}dS_{\omega}\int_{\frac{\arrowvert x\arrowvert}{2}}^{\arrowvert x\arrowvert+2}\frac{e^{-\sqrt{\tau}\arrowvert x-r\omega\arrowvert}}{\arrowvert x-r\omega\arrowvert}\cdot e^{-\frac{ir^{2}}{4}}\cdot r^{2}dr\arrowvert\leqslant C\arrowvert x\arrowvert,
	\end{flalign*}
	and
	\begin{flalign*}
		\arrowvert\int_{0}^{1}\tau^{\frac{s}{2}-1}d\tau\int_{\sigma(\omega)}dS_{\omega}\int_{\arrowvert x\arrowvert+2}^{\infty}\frac{e^{-\sqrt{\tau}\arrowvert x-r\omega\arrowvert}}{\arrowvert x-r\omega\arrowvert}\cdot e^{-\frac{ir^{2}}{4}}\cdot r^{2}dr\arrowvert\leqslant C.
	\end{flalign*}
	Therefore, we have the estimate of the second term
	$$\arrowvert\int_{0}^{1}\tau^{\frac{s}{2}-1}d\tau\int_{\sigma(\omega)}dS_{\omega}\int_{\frac{\arrowvert x\arrowvert}{2}}^{\infty}\frac{e^{-\sqrt{\tau}\arrowvert x-r\omega\arrowvert}}{\arrowvert x-r\omega\arrowvert}\cdot e^{-\frac{ir^{2}}{4}}\cdot r^{2}dr\arrowvert\leqslant C+C\arrowvert x\arrowvert.$$
	When $\tau>1$, we let $x-y=r\omega$ and obtain that
	\begin{flalign*}
		\arrowvert\int_{1}^{\infty}\tau^{\frac{s}{2}-1}d\tau\int_{\sigma(\omega)}dS_{\omega}\int_{0}^{\infty}\frac{e^{-\sqrt{\tau}r}}{r}\cdot e^{-\frac{i\arrowvert x-r\omega\arrowvert^{2}}{4}}\cdot r^{2}dr\arrowvert\leqslant C.
	\end{flalign*}
	Combining these results of $0<\tau<1$ and $\tau>1$, we have
	\begin{flalign}
		\arrowvert\uppercase\expandafter{\romannumeral1}\arrowvert\leqslant C+C\arrowvert x\arrowvert.\label{B6}
	\end{flalign}
	
	For \uppercase\expandafter{\romannumeral2}, we have
	$
		\uppercase\expandafter{\romannumeral2}=\int_{0}^{\infty}\tau^{\frac{s}{2}-1}\int_{\mathbb{R}^{3}}\frac{e^{-\sqrt{\tau}\arrowvert x-y\arrowvert}}{\arrowvert x-y\arrowvert}\cdot e^{-\frac{i\arrowvert y\arrowvert^{2}}{4}}\cdot\arrowvert y\arrowvert^{2}dyd\tau.
$
	When $0<\tau<1$, let $y=r\omega$, i.e. $\int_{0}^{1}\tau^{\frac{s}{2}-1}d\tau\int_{\sigma(\omega)}dS_{\omega}\int_{0}^{\infty}\frac{e^{-\sqrt{\tau}\arrowvert x-r\omega\arrowvert}}{\arrowvert x-r\omega\arrowvert}\cdot e^{-\frac{ir^{2}}{4}}\cdot r^{4}dr$ and split $\int_{0}^{\infty}dr$ into two terms. For the first term, it can be obtained by integration by parts that
	\begin{flalign*}
		\arrowvert\int_{0}^{1}\tau^{\frac{s}{2}-1}d\tau\int_{\sigma(\omega)}dS_{\omega}\int_{0}^{\frac{\arrowvert x\arrowvert}{2}}\frac{e^{-\sqrt{\tau}\arrowvert x-r\omega\arrowvert}}{\arrowvert x-r\omega\arrowvert}\cdot e^{-\frac{ir^{2}}{4}}\cdot r^{4}dr\arrowvert\leqslant C+C\arrowvert x\arrowvert.
	\end{flalign*}
	For the second term, it can be split two parts: $\int_{\frac{\arrowvert x\arrowvert}{2}}^{\infty}=\int_{\frac{\arrowvert x\arrowvert}{2}}^{\arrowvert x\arrowvert+2}+\int_{\arrowvert x\arrowvert+2}^{\infty}$, and when $0<\arrowvert x\arrowvert<1$, we have
	\begin{flalign*}
		\arrowvert\int_{0}^{1}\tau^{\frac{s}{2}-1}d\tau\int_{\sigma(\omega)}dS_{\omega}\int_{\frac{\arrowvert x\arrowvert}{2}}^{\arrowvert x\arrowvert+2}\frac{e^{-\sqrt{\tau}\arrowvert x-r\omega\arrowvert}}{\arrowvert x-r\omega\arrowvert}\cdot e^{-\frac{ir^{2}}{4}}\cdot r^{4}dr\arrowvert\leqslant C,
	\end{flalign*}
	and
	\begin{flalign*}
		\arrowvert\int_{0}^{1}\tau^{\frac{s}{2}-1}d\tau\int_{\sigma(\omega)}dS_{\omega}\int_{\arrowvert x\arrowvert+2}^{\infty}\frac{e^{-\sqrt{\tau}\arrowvert x-r\omega\arrowvert}}{\arrowvert x-r\omega\arrowvert}\cdot e^{-\frac{ir^{2}}{4}}\cdot r^{4}dr\arrowvert\leqslant C.
	\end{flalign*}
	When $\arrowvert x\arrowvert>1$, we have
	\begin{flalign*}
		\arrowvert\int_{0}^{1}\tau^{\frac{s}{2}-1}d\tau\int_{\sigma(\omega)}dS_{\omega}\int_{\frac{\arrowvert x\arrowvert}{2}}^{\arrowvert x\arrowvert+2}\frac{e^{-\sqrt{\tau}\arrowvert x-r\omega\arrowvert}}{\arrowvert x-r\omega\arrowvert}\cdot e^{-\frac{ir^{2}}{4}}\cdot r^{4}dr\arrowvert\leqslant C\arrowvert x\arrowvert.
	\end{flalign*}
	And
	\begin{flalign*}
		\arrowvert\int_{0}^{1}\tau^{\frac{s}{2}-1}d\tau\int_{\sigma(\omega)}dS_{\omega}\int_{\arrowvert x\arrowvert+2}^{\infty}\frac{e^{-\sqrt{\tau}\arrowvert x-r\omega\arrowvert}}{\arrowvert x-r\omega\arrowvert}\cdot e^{-\frac{ir^{2}}{4}}\cdot r^{4}dr\arrowvert\leqslant C+C\arrowvert x\arrowvert.
	\end{flalign*}
	Therefore, we have the estimate of the second term
	$$\arrowvert\int_{0}^{1}\tau^{\frac{s}{2}-1}d\tau\int_{\sigma(\omega)}dS_{\omega}\int_{\frac{\arrowvert x\arrowvert}{2}}^{\infty}\frac{e^{-\sqrt{\tau}\arrowvert x-r\omega\arrowvert}}{\arrowvert x-r\omega\arrowvert}\cdot e^{-\frac{ir^{2}}{4}}\cdot r^{4}dr\arrowvert\leqslant C+C\arrowvert x\arrowvert.$$
	When $\tau>1$, we let $x-y=r\omega$ and it can be obtained by integration by parts that
\begin{flalign*}
		\arrowvert\int_{1}^{\infty}\tau^{\frac{s}{2}-1}d\tau\int_{\sigma(\omega)}dS_{\omega}\int_{0}^{\infty}\frac{e^{-\sqrt{\tau}r}}{r}\cdot r^{2}\cdot e^{-\frac{i\arrowvert x-r\omega\arrowvert^{2}}{4}}\cdot\arrowvert x-r\omega\arrowvert^{2}dr\arrowvert\leqslant C+C\arrowvert x\arrowvert.
	\end{flalign*}
	Combining these results of $0<\tau<1$ and $\tau>1$, we have
	\begin{flalign}
		\arrowvert\uppercase\expandafter{\romannumeral2}\arrowvert
		\leqslant C+C\arrowvert x\arrowvert.\label{B8}
	\end{flalign}
	
	Therefore, substituting estimates of \eqref{B6} and \eqref{B8} into the estimate of $(-\Delta)^{\frac{s}{2}}e^{-\frac{i\arrowvert x\arrowvert^{2}}{4}}$, we obtain that
	$$\arrowvert(-\Delta)^{\frac{s}{2}}e^{-\frac{i\arrowvert x\arrowvert^{2}}{4}}\arrowvert\leqslant C\arrowvert\uppercase\expandafter{\romannumeral1}\arrowvert+C\arrowvert\uppercase\expandafter{\romannumeral2}\arrowvert\leqslant C\langle x\rangle,$$
	for $n=3$. This implies that
	$$\|u_{0}(x)(-\Delta)^{\frac{s}{2}}(e^{-\frac{i\arrowvert x\arrowvert^{2}}{4}})\|_{L^{2}_{x}}\leqslant C\|\langle x\rangle u_{0}\|_{L^{2}_{x}}.$$
	
	When $n=2$, the resolvent $R_{0}(\tau)=(\tau-\Delta)^{-1}$ is an integral operator with the integral kernel$\frac{1}{4}H_{0}^{(1)}(i\tau^{\frac{1}{2}}\arrowvert x-y\arrowvert),$ where $H_{0}^{(1)}$ is the first Hankel function. Then we have
	\begin{flalign}
		(-\Delta)^{\frac{s}{2}}e^{-\frac{i\arrowvert x\arrowvert^{2}}{4}}=&c(s)\int_{0}^{\infty}\tau^{\frac{s}{2}-1}(\tau-\Delta)^{-1}e^{-\frac{i\arrowvert x\arrowvert^{2}}{4}}(-i-\frac{\arrowvert x\arrowvert^{2}}{4})d\tau\nonumber\\
		=&-\frac{ic(s)}{4}\int_{0}^{\infty}\tau^{\frac{s}{2}-1}\int_{\mathbb{R}^{2}}H_{0}^{(1)}(i\tau^{\frac{1}{2}}\arrowvert x-y\arrowvert)\cdot e^{-\frac{i\arrowvert y\arrowvert^{2}}{4}}dyd\tau\nonumber\\
		&-\frac{c(s)}{16}\int_{0}^{\infty}\tau^{\frac{s}{2}-1}H_{0}^{(1)}(i\tau^{\frac{1}{2}}\arrowvert x-y\arrowvert)\cdot e^{-\frac{i\arrowvert y\arrowvert^{2}}{4}}\cdot\arrowvert y\arrowvert^{2}dyd\tau\nonumber\\
		:=&-\frac{ic(s)}{4}\uppercase\expandafter{\romannumeral1'}-\frac{c(s)}{16}\uppercase\expandafter{\romannumeral2'}.\label{B9}
	\end{flalign}
	Let's recall some properties of Hankel functions (see \cite{AS}). (1) We have the following recurrence formula:
	$$\frac{d}{d\zeta}(\zeta^{\nu}H_{\nu}^{(1)}(\zeta))=\zeta^{\nu}H_{\nu-1}^{(1)}(\zeta),\quad  \zeta>0,\ Re\nu>\frac{1}{2};$$
	(2) we have the following estimates:
	\begin{equation*}
		\arrowvert H_{\nu}^{(1)}(i\tau^{\frac{1}{2}}\arrowvert x-y\arrowvert)\arrowvert\leqslant
		\left\{
		\begin{split}
			&C(\tau^{\frac{1}{2}}\arrowvert x-y\arrowvert)^{-\nu}, \qquad\qquad\,\,\,\,\,     \tau^{\frac{1}{2}}\arrowvert x-y\arrowvert\leqslant1;\\
			&C\tau^{-\frac{1}{4}}\arrowvert x-y\arrowvert^{-\frac{1}{2}}e^{-\tau^{\frac{1}{2}}\arrowvert x-y\arrowvert},\ \ \tau^{\frac{1}{2}}\arrowvert x-y\arrowvert>1.
		\end{split}
		\right.
	\end{equation*}
	Therefore, we can divide into $\uppercase\expandafter{\romannumeral1'}$ two parts, $0<\tau<1$ and $\tau>1$, and agree to apply to $\uppercase\expandafter{\romannumeral2'}$.  Applying the same method of $n=3$ to \eqref{B9} and using above properties, we also obtain that
	$$\arrowvert(-\Delta)^{\frac{s}{2}}e^{-\frac{i\arrowvert x\arrowvert^{2}}{4}}\arrowvert\leqslant C\arrowvert\uppercase\expandafter{\romannumeral1'}\arrowvert+C\arrowvert\uppercase\expandafter{\romannumeral2'}\arrowvert\leqslant C\langle x\rangle,$$
	hence
	$$\|u_{0}(x)(-\Delta)^{\frac{s}{2}}(e^{-\frac{i\arrowvert x\arrowvert^{2}}{4}})\|_{L^{2}_{x}}\leqslant C\|\langle x\rangle u_{0}\|_{L^{2}_{x}}.$$
	
	When $n>3$, there imply $s>2$. In order to simplify the form, we do not adopt the method of calculating the integral kernel of the free resolvent. We use the result of oscillation integral to estimate directly $\|u_{0}(x)(-\Delta)^{\frac{s}{2}}(e^{-\frac{i\arrowvert x\arrowvert^{2}}{4}})\|_{L^{2}_{x}}$ and get a slightly rough result.
	\begin{flalign*}
		\|u_{0}(x)(-\Delta)^{\frac{s}{2}}(e^{-\frac{i\arrowvert x\arrowvert^{2}}{4}})\|^{2}_{L^{2}_{x}}=\frac{1}{(\pi i)^{n}}\|u_{0}(x)\int_{\mathbb{R}^{n}}e^{ix\cdot\xi}\arrowvert\xi\arrowvert^{s}e^{i\arrowvert \xi\arrowvert^{2}}d\xi\|^{2}_{L^{2}_{x}}\leqslant C\|\langle x\rangle^{s} u_{0}\|^{2}_{L^{2}_{x}},
	\end{flalign*}
	since that
	\begin{flalign*}
		\mathscr{F}(e^{-\frac{i\arrowvert x\arrowvert^{2}}{4}})(\xi)=(2\pi)^{-\frac{n}{2}}e^{i\arrowvert\xi\arrowvert^{2}}\int_{\mathbb{R}}e^{-i(\xi_{1}+\frac{x_{1}}{2})^{2}}dx_{1}...\int_{\mathbb{R}}e^{-i(\xi_{n}+\frac{x_{n}}{2})^{2}}dx_{n}=(\frac{2}{i})^{\frac{n}{2}}e^{i\arrowvert\xi\arrowvert^{2}},
	\end{flalign*}
	and we have estimates of oscillation integral as following (see Lemma 2.1 in \cite{HHZ})
	$$\arrowvert\int_{\mathbb{R}^{n}}e^{ix\cdot\xi}\arrowvert\xi\arrowvert^{s}e^{i\arrowvert \xi\arrowvert^{2}}d\xi\arrowvert\leqslant C\arrowvert x\arrowvert^{s}+C.$$
	
	Combining with the estimates of all dimensions, we divide the result into two parts. When $n\leqslant3$, we obtain the estimate in \eqref{B1}. When $n>3$, we obtain the estimate in \eqref{B2}.
	
	Therefore, the description of initial value space $$\|u_{0}\|_{\Sigma_{s}}:=\|(-\Delta)^{\frac{s}{2}}(e^{-\frac{i\arrowvert x\arrowvert^{2}}{4}}u_{0})\|_{L^{2}_{x}}$$ is complete.

\vspace{2em}
\phantomsection
	

\addcontentsline{toc}{section}{References}
\begin{thebibliography}{00}
	
	
	
	\bibitem{AS}
	Abramowitz, M., Stegun, I. A.: Handbook of Mathematical Functions with Formulas, Graphs and Mathematical Tables. U.S. Government Printing Office (1965)
	
	
	\bibitem{Ag}
	Agmon, S.: Spectral properties of Schr\"{o}dinger operators and scattering theory. Ann. Sc. Norm. Sup. Pisa Cl. Sci. \textbf{2}, 151-218 (1975)
	
	\bibitem{Ba}
	Barab, J. E.: Nonexistence of asymptotically free solutions for a nonlinear Schr\"{o}dinger equation. J. Math. Phys. \textbf{25}, 3270-3273 (1984)
	
	\bibitem{Bo}
	Bouard, A. D.: Nonlinear Schr\"odinger equations with magnetic fields. Diff. Int. Eqns. \textbf{4}, 73-88 (1991)
	
	\bibitem{BL}	
	Bergh, J., L\"{o}fstr\"{o}m, J.: Interpolation Spaces. Springer-Verlag, New York (1976)	
	
	
	\bibitem{C}
	Cazenave,T.: Semilinear Schr\"odinger equations, Courant Lecture Notes in Mathematics 10. New York University, Courant Institute of Mathematical Sciences (2003)
	
	\bibitem{CCV}
	Cardoso, F., Cuevas, C., Vodev, G.:  Dispersive estimates for the Schr\"{o}dinger equation in dimensions four and five. Asym. Anal. \textbf{62}, 125-145 (2009)
	
	\bibitem{CGV}
	Cuccagna, S., Georgiev, V., Visciglia, N.: Decay and scattering of small solutions of pure power NLS in $\mathbb{R}$ with $p>3$ and with a potential. Comm. Pure Appl. Math. \textbf{67}, 0957-0981 (2014)
	
	\bibitem{CK}
	Cuenin, J. C., Kenig, C.: $L^{p}$ resolvent estimates for magnetic Schr\"{o}dinger operators with unbounded background fields. Comm. Part. Diff. Equs. \textbf{42}, 235-260 (2017)
	
	\bibitem{D}	
	Dong, L.: On Kato-Ponce and fractional Leibniz. Rev. Mat. Iberoam. \textbf{35}, 23-100 (2019)	
	
	\bibitem{DW}	
	Duan, Z., Wei, L.: Scattering for the fractional magnetic Schr\"odinger operators. arXiv: 2304. 07846 (2023) 
	
	
	
	
	
	\bibitem{FY}
	Finco, D., Yajima, K.: The $L^{p}$ boundedness of wave operators for Schr\"{o}dinger operators with threshold singularities \uppercase\expandafter{\romannumeral2} : Even dimensional case. J. Math. Sci. Univ. Tokyo \textbf{13}, 277-346 (2006)
	
	
	\bibitem{Gl}
	Glassey, R. T.: On the asymptotic behavior of nonlinear wave equations, Trans. Amer. Math. Soc. \textbf{182}, 187-200 (1973)
	
	\bibitem{Go2}
	Goldberg, M.: Dispersive bounds for the three-dimensional Schr\"{o}dinger equation with almost critical potentials. Geom. Func. Anal. \textbf{16}, 517-536 (2006)
	
	
	
	
	\bibitem{GOV}
	Ginibre, J., Ozawa, T., Velo, G.: On the existence of the wave operators for a class of nonlinear Schr\"{o}dinger equations. Ann. Inst. H. Poincar\'{e} Phys. Th\'{e}or. \textbf{60}, 211-239 (1994)
	
	\bibitem{GS}
	Goldberg, M., Schlag, W.: Dispersive estimates for Schr\"{o}dinger operators in dimensions one and three. Comm. Math. Phys. \textbf{251}, 157-178 (2004)
	
	
	\bibitem{H}
	H\"{o}rmander, L.: The analysis of linear partial differential operators  \uppercase\expandafter{\romannumeral1}-\uppercase\expandafter{\romannumeral4}. Springer-Verlag, Berlin (1983)-(1985)
	
	
	\bibitem{HHZ}Huang, T., Huang, S., Zheng, Q.: Inhomogeneous oscillatory integrals and global smoothing effects for dispersive equations. J. Differential Equations \textbf{263}, 8606–8629 (2017)
	
	\bibitem{HN}
	Hayashi, N., Naumkin, P.: Asymptotics for large time of solutions to the nonlinear Schr\"{o}dinger and Hartree equations. Amer. J. Math. \textbf{120},  369-389 (1998)	
	
	\bibitem{I}
	Iwatsuka, A.: Spectral representation for Schr\"{o}dinger operator with magnetic vector potential. J. Math. Kyoto Univ. \textbf{22}, 223-242 (1982)
	
	
	\bibitem{JK}
	Jensen, A., Kato, T.: Spectral properties of Schr\"{o}dinger operators and time-decay of the wave functions. Duke Math. J. \textbf{46}, 583-611 (1979)
	
	\bibitem{JSS}
	Journ\'{e}, J. L., Soffer,  A., Sogge, C. D.: Decay estimates for Schr\"{o}dinger operators. Comm. Pure Appl. Math. \textbf{44}, 573-604 (1991)
	
	\bibitem{K1}
	Kato, T.: On nonlinear Schr\"odinger equations. Ann. Inst. Henri Poincar\'{e}, Physique Th\'{e}orique \textbf{46}, 113-129 (1987)
	
	\bibitem{K2}
	Kato, T.: Nonlinear Schr\"odinger equations, in \textquotedblleft Schr\"{o}dinger Operators\textquotedblright. Lecture Notes in Phys. \textbf{345}, 218-263 (1989)
	
	
	
	\bibitem{KRS}
	Kenig, C., Ruiz, A., Sogge, C. D.: Uniform Sobolev inequalities and unique continuation for second order constant coefficient differential operators. Duke Math. J. \textbf{55}, 329-347 (1987)
	
	
	\bibitem{KPV}
	Kenig, C., Ponce, G., Vega, L.: Well-posedness and scattering results for the generalized Korteweg-de Vries equation via the contraction principle. Comm. Pure Appl. Math. \textbf{46}, 527-620 (1993)
	
	
	\bibitem{KK}
	Komech, A. I., Kopylova, E. A.: Dispersive decay for the magnetic Schr\"{o}dinger equation. J. Funct. Anal. \textbf{264}, 735-751 (2013)
	
	\bibitem{KT}
	Koch, H., Tataru, D.: Carleman estimates and absence of embedded eigenvalues. Commun. Math. Phys. \textbf{267}, 419-449 (2006)
	
	
	\bibitem{MS}
	McKean, H. P., Shatah, J.: The nonlinear Schr\"{o}dinger equation and the nonlinear heat equation reduction to linear form. Comm. Pure Appl. Math. \textbf{44}, 1067-1080 (1991)
	
	\bibitem{MS1}	
	Martinez, C., Sanz. M.: The Theory of Fractional Powers of Operators. North Holland (2001)	
	
	
	\bibitem{MMS}
	Merono, C. J., Machado, L. P., Salo, M.: Resolvent estimates for the magnetic Schr\"{o}dinger operator in dimensions $n\geqslant 2$. Arxiv:1904.00693
	
	\bibitem{Na}
	Nakamura, Y.: Local solvability and smoothing effects of nonlinear Schr\"odinger equations with magnetic fields. Funkcialaj Ekvac. \textbf{44}, 1-18 (2001)
	
	\bibitem{NS}
	Nakamura, Y., Shimomura, A.: Local well-posedness and smoothing effects of strong solutions for nonlinear Schr\"{o}dinger equations with potentials and magnetic fields. Hokkaido Math. J. \textbf{34}, 37-63 (2005)
	
	
	
	
	
	\bibitem{R}
	Rauch, J.: Local decay of scattering solutions to Schr\"{o}dinger's equation. Comm. Math. Phys. \textbf{61}, 149-168 (1978)
	
	
	\bibitem{RS1}
	Rodnianski, I., Schlag, W.: Time decay for solutions of Schr\"{o}dinger equations with rough and time-dependent potentials. Invent. Math. \textbf{155}, 451-513 (2004)
	
	\bibitem{Sc1}
	Schlag, W.: Dispersive estimates for Schr\"{o}dinger operators in dimension two. Comm. Math. Phys. \textbf{257}, 87-117 (2005)
	
	
	\bibitem{St1}
	Strauss, W.: Nonlinear scattering theory, Scattering Theory in Mathematical Physics. NATO Advanced Science Institutes \textbf{9}, 53-78 (1974)
	
	\bibitem{St3}
	Strauss, W.: Nonlinear scattering at low energy: sequel. J. Funct. Anal. \textbf{41}, 110-133 (1981)
	
	\bibitem{St2}
	Strauss, W.: Nonlinear scattering theory at low energy: sequel. J. Funct. Anal. \textbf{43}, 281-293 (1981)
	
	\bibitem{T1}
	Tsutsumi, Y.: $L^{2}$-solutions for nonlinear Schr\"odinger equations and nonlinear groups. Funkcial Ekvac. \textbf{30}, 115-125 (1987)
	
	\bibitem{T2}
	Tsutsumi, Y.: Global strong solutions for nonlinear Schr\"odinger equations. Nonlinear Anal. \textbf{11}, 1143-1154 (1987)
	
	\bibitem{Ya1}
	Yajima, K.: Schr\"odinger evolution equations with magnetic fields. J. Anal. Math. \textbf{56}, 29-76 (1991)
	
	\bibitem{Ya2}
	Yajima, K.: The $W_{k,p}$-continuity of wave operators for Schr\"{o}dinger operators. J. Math. Soc. Japan \textbf{47}, 551-581 (1995)
	
	\bibitem{Ya3}
	Yajima, K.: The $L^{p}$ boundedness of wave operators for Schr\"{o}dinger operators with threshold singularities I: the odd dimensional case. J. Math. Sci. Univ. Tokyo \textbf{13}, 43-94 (2006)
	
\end{thebibliography}
\end{document}